\documentclass{amsart}
\usepackage[foot]{amsaddr}
\usepackage[dvipsnames,table]{xcolor}
\usepackage[english]{babel}
\usepackage{soul}
\usepackage[normalem]{ulem}
\usepackage{amsmath, amsfonts, amssymb, amsthm}
\usepackage{graphicx}
\usepackage{caption}
\usepackage{subcaption}
\usepackage{tikz}
\usepackage{pgfplots}
\usepackage{dsfont}
\usepackage{mathtools,nccmath}
\usepackage{esint}
\usepackage{xfrac}
\usepackage[short labels]{enumitem}
\usepackage{cite}
\usepackage{accents}
\newcommand{\dbtilde}[1]{\accentset{\approx}{#1}}

\usepackage{hyperref}
\hypersetup{
    colorlinks = true,
    linkcolor = blue,
    citecolor = Green
    }

\usepackage[
        noabbrev,
        capitalise,
        nameinlink,
    ]
    {cleveref}

\crefrangeformat{section}{Sections~#3#1#4 through~#5#2#6}

\usepackage{autonum}
\usepackage{fullpage}

\renewcommand{\epsilon}{\varepsilon}
\newcommand{\ep}{\epsilon}
\newcommand{\eps}{\ep}

\newcommand{\rsig}{w}
\newcommand{\kirch}{B}
\newcommand{\Bo}{B_0}

\newcommand{\Hr}{H_{\rm r}}
\newcommand{\Hf}{H_{\rm f}}
\newcommand{\Lr}{L_{\rm r}}
\newcommand{\Lf}{L_{\rm f}}
\newcommand{\crr}{c_{\rm r}}
\newcommand{\cf}{c_{\rm f}}
\newcommand{\oH}{\overline{\mathbb{H}^2}}
\newcommand{\vt}{\vartheta}
\newcommand{\cA}{\mathcal A}
\newcommand{\cP}{\mathcal P}
\newcommand{\cW}{\mathcal W}

\newcommand{\tH}{\Hr}

\newcommand{\1}{\mathds{1}}
\newcommand{\R}{\mathbb R}
\renewcommand{\H}{\mathbb H}

\DeclareMathOperator{\dist}{dist}

\newtheorem{theorem}{Theorem}[section]
\newtheorem{lemma}[theorem]{Lemma}

\newtheorem*{claim*}{Claim}
\newtheorem{proposition}[theorem]{Proposition}
\newtheorem{corollary}[theorem]{Corollary}

\theoremstyle{definition}
\newtheorem{definition}[theorem]{Definition}

\theoremstyle{remark}

\numberwithin{equation}{section}

\newcommand{\be}{\begin{equation}}
\newcommand{\ee}{\end{equation}}

\begin{document}

\title[Hamilton-Jacobi approach for road-field models]
	{A Hamilton-Jacobi approach to road-field reaction-diffusion models}

\author{Christopher Henderson}
\address[Christopher Henderson]{Department of Mathematics, University of Arizona, Tucson, AZ, 85719, USA}
\email{ckhenderson@math.arizona.edu}

\author{King-Yeung Lam}
\address[King-Yeung Lam]{Department of Mathematics, The Ohio State University, Columbus, OH, 43210, USA}
\email{lam.184@osu.edu}

\begin{abstract}
We consider the road-field reaction-diffusion model introduced by Berestycki, Roquejoffre, and Rossi.  By performing a ``thin-front limit,'' we are able to deduce a Hamilton-Jacobi equation with a suitable effective Hamiltonian on the road that governs the front location of the road-field model.  Our main motivation is to apply the theory of strong (flux-limited) viscosity solutions in order to determine a control formulation interpretation of the front location.  In view of the ecological meaning of the road-field model, this is natural as it casts the invasion problem as one of finding optimal paths that balance the positive growth rate in the field with the fast diffusion on the road.

Our main contribution is a nearly complete picture of the behavior on two-road conical domains.  When the diffusivities on each road are the same, we show that the  propagation speed in each direction in the cone can be computed via those associated with one-road half-space problem.   When the diffusivities differ, we show that the speed along the faster road is unchanged, while the speed along the slower road can be enhanced.  Along the way we provide a new proof of known results on the one-road half-space problem via our approach.
\end{abstract}

\maketitle

\section{Introduction}

\subsection{The model and main questions}

In \cite{BRR_JMB}, Berestycki, Roquejoffre, and Rossi introduced a model for the invasion of a species that can inhabit two different environments, a ``field'' in which each individual moves slowly and reproduces and a ``road,'' on which it moves quickly but cannot reproduce.  We refer to~\cite{BRR_JMB} for a more in-depth discussion of the ecological relevance of the model; however, the reader may find it helpful to have in mind the example of wolf packs in Western Canada, which have been observed to move quickly in the forest using seismic lines~\cite{wolves}; see also \cite{HillenPainter}.
We also mention \cite{Chen2023asymptotic,LiWang,Huang2024preprint} in which other effective boundary conditions, different from what we consider here, were derived by letting the width of the road to zero.

The Berestycki-Roquejoffre-Rossi model that we study in this paper is given by
\begin{equation}\label{e.unscaled_berestycki}
    \begin{cases}
        V_t - \tilde d\Delta V = \tilde rV(1-V)
            \qquad&\text{ for  }(t,x,y) \in (0,\infty) \times \H,\\
        U_t - \tilde{D} U_{xx} = \tilde \nu V|_{y=0} - \tilde \mu U
            \qquad &\text{ for  }(t,x) \in (0,\infty) \times \R,\\
        -\tilde dV_y(t,x,0) = \tilde \mu U(t,x) - \tilde \nu  V(t,x,0)  
            \qquad&\text{ for  }(t,x) \in (0,\infty) \times \R,
        \end{cases}
\end{equation}
where we use the following notation for the upper half-space:
\begin{equation}
    \H = \R_x \times (0,\infty)_y.
\end{equation}
Here, $V(t,x,y)$ represents the population density of the species at time $t$ in the field at $(x,y) \in \H$, $U(t,x)$ represents the population density at time $t$ on the road $(x,0) \in \R\times \{0\}$, and $\tilde d, \tilde D, \tilde r, \tilde \nu$, and $\tilde \mu$ are positive parameters related to the diffusivity of the population off the road and on the road, the reproduction rate of the species, and the exchange rate to and from the road, respectively.  After a suitable scaling,~\eqref{e.unscaled_berestycki} becomes\footnote{We find an extra parameter here compared to~\cite{BRR_JMB}.  See also the discussion in~\cite{LiWang}}
\begin{equation}\label{e.berestycki}
    \begin{cases}
        V_t - \Delta V = V(1-V)
           \qquad&\text{ in }  (0,\infty) \times \H,\\
        U_t - D U_{xx} = \nu V|_{y=0} - \mu U
            \qquad&\text{ in } (0,\infty) \times \R,\\
        -V_y = \kappa(\mu U - \nu  V)
            \qquad&\text{ in } (0,\infty) \times \R \times \{0\},
        \end{cases}
\end{equation}
and this is the model that we analyze.  We also consider this model posed on conical domains, bounded by two roads. That system takes significantly more space and care to write, so we postpone it to~\eqref{e.eqn11}.

The main question is understanding how the unexplored area ($U,V \sim 0$) is invaded and populated up to the fully populated steady state $(U,V) \sim (\sfrac\mu\nu,1)$.  A rephrasing of this is to ask how the levels sets of $U$ and $V$ propagate starting from initial data that is a compact perturbation of $(0,0)$; that is, 
\be\label{e.initial_data1}
    		U(0,\cdot) \in L^\infty(\R)
		\quad\text{ and }\quad
		V(0,\cdot) \in L^\infty(\H)
			~\text{ are nonnegative, nontrivial, and compactly supported.}
\ee
Actually, we make the seemingly stronger assumption that
\be\label{e.initial_data2}
	U(0,\cdot) \leq \frac{\nu}{\mu}
	\quad\text{ and }\quad
	V(0,\cdot) \leq 1.
\ee
In practice, this is not a strong assumption because, by a simple comparison principle argument, $U-\sfrac\nu\mu$ and $V-1$ have an exponentially decaying-in-time upper bound.

Importantly, one wishes to understand how the behavior of $(U,V)$ deviates from the homogeneous, or ``non-road,'' case
\be\label{e.homogeneous}
	\begin{cases}
		\tilde V_t - \Delta \tilde V = \tilde V(1-\tilde V)
   	        \qquad&\text{ in }  (0,\infty) \times \H,\\
        -\tilde V_y = 0
            \qquad&\text{ in } (0,\infty) \times \R \times \{0\},
	\end{cases}
\ee
and how the propagation depends quantitatively on the parameters $D$, $\mu$, and $\nu$.  The interesting case is when the diffusion on the road is faster than it is on the field.  As such, we make the standing assumption
\be
	D > 1
\ee
throughout the paper, even when not explicitly stated.

This model has attracted a huge amount of attention; see, e.g.,~\cite{AlfaroDucasseTreton, BRR_JMB, Berestycki2016shape, Ducasse2018influence, BRR_Non, BCRR, DR, Pauthier} and the many references therein.  While many questions are investigated, the basic results are that, when $D\leq 2$, the invasion occurs with speed $2$ in all directions, just as in~\eqref{e.homogeneous}, and when $D >2$, the speed is directionally dependent and scales like $\sqrt D$ as $D\to \infty$ along the road. We discuss many of these below.  These works are dependent on direct analysis of~\eqref{e.berestycki} relying on the careful construction of sub- and supersolutions.  This is, in a sense, an {\em Eulerian} approach to the problem.

Our goal here is to provide a more {\em Lagrangian} perspective.  Specifically, we perform a thin front limit {\em \`a la} Evans and Souganidis~\cite{Evans1989pde} (see also~\cite{BarlesEvansSouganidis, MajdaSouganidis, Freidlin1986geometric,Freidlin(N)}) to connect~\eqref{e.berestycki} with a suitable Hamilton-Jacobi equation.  This equation is a ``junction'' type problem where the road and field meet.  Using recent developments in the theory of Hamilton-Jacobi equations with junctions~\cite{Imbert2017flux,Imbert2017quasi,Forcadel2023non, Barles2018flux, Lions2017well,Guerand2017effective}, we characterize the longtime behavior of $(U,V)$ up to $o(t)$ fluctuations in space in terms of the zero set $\{J=0\}$ of the solution to the control problem
\begin{equation}\label{e.c031901}
    J(t,x,y)
    	= \min
		\int_0^t
			\hat{L}(\gamma(s), \dot \gamma(s)) \,ds,
\end{equation}
where we clarify the set on which minimum is taken and the Lagrangian $\hat L$ in the sequel (see~\eqref{e.N} and~\eqref{e.hatL}, respectively).  
Here, 
\be
\hat L(x,y,v_1,v_2) = \left( \tfrac{1}{4}|(v_1,v_2)|^2 -1\right) \1_{\{y>0\}} + \Lr(v_1) \1_{\{y=0\}}\ee 
represents a running cost that accounts for the faster diffusion on the road and the reproduction on the field, in which the running cost $L_r(v_1)$ on the road is defined implicitly in terms of all the parameters of the problem (see~\eqref{e.hatL} for the full definition of $\hat L$). 
In a sense, the optimizers in the above problem provide the ``optimal path'' that individuals should take.  This characterizes when the individual should be on the road versus the field and how fast it should move on each.  This, in a sense, provides a natural ecological interpretation of previous results, and it is appropriate given the motivating question: ``how do the paths that individuals take affect the population's expansion?''  Additionally, it reduces the study of~\eqref{e.berestycki}, which requires careful analysis via sub- and supersolutions, to a simple control problem~\eqref{e.c031901} that is numerically tractable and easy to approximate by hand.

Before diving into the specifics, a brief summary of the main results of the paper is the following:
\begin{itemize}

	\item We deduce the appropriate Lagrangian $\hat L$ and show that the ``front'' of~\eqref{e.berestycki} is given by $\{J(t,x,y) = 0\}$.  We can then recover all propagation results from~\cite{BRR_JMB, Berestycki2016shape} purely through an analysis of $J$ via~\eqref{e.c031901}.  
	
	\item We demonstrate a connection between the road and field problem~\eqref{e.berestycki} and the ongoing work on Hamilton-Jacobi equations with junctions to the road-field model~\eqref{e.berestycki}. 

    \item Our results generalize from $\H$ to conical domains $\Omega_a$ bounded by two roads forming an angle $0 < 2a \leq \pi$.  In fact, the front in this case can be given explicitly in terms of the front in the half-space case.  An interesting consequence is that the populated region is ``often'' non-convex, which is in contrast to the case on $\H$, where it is strictly convex for all $D>1$.  In a sense, this is the main novelty of our work because it seems that estimates on the spreading behavior in the interior of $\Omega_a$ are not accessible via the types of planar supersolutions typically used in the analysis of~\eqref{e.berestycki}.  With our approach, they are an immediate consequence of the analysis of the half-space case.

\end{itemize}

In summary, we marry two ongoing currents of research, providing both a new perspective to the road-field model and a new application to the theory of Hamilton-Jacobi equations with junctions.

\subsection{Statement of main results}\label{s.main}

In this section, we give the more technical statement of our results.  To motivate each result, we give an idea of the main computations that lead to the control formulation for~\eqref{e.berestycki}.

\subsubsection{Finding the correct Hamilton-Jacobi equation and the propagation result}
As we expect ballistic propagation, we perform a standard Hopf-Cole transform and rescaling of the equation: let
\begin{equation}\label{e.thin_front_scaling}
	u^\ep(t,x) = -\ep \log U(\sfrac{t}{\ep},\sfrac{x}{\ep})
	\quad\text{ and }\quad
	v^\ep(t,x) = - \ep \log V(\sfrac{t}{\ep},\sfrac{x}{\eps},\sfrac{y}{\ep}),
\end{equation}
then
\begin{equation}\label{e.scaled_eqn}
    \begin{cases}
        v^\ep_t - \ep \Delta v^\ep + |v^\ep_x|^2 + |v^\ep_y|^2 + (1-e^{-\frac{v^\ep}{\ep}}) =0
            \qquad&\text{ in }  (0,\infty) \times \H,\\
        u^\ep_t - \ep D u^\ep_{xx}+ D|u^\ep_x|^2 +  \nu e^{\frac{u^\ep - v^\ep}{\ep}} - \mu  =0
            \qquad&\text{ in }  (0,\infty) \times \R,\\
         v^\ep_y = \kappa(\mu e^{\frac{v^\ep - u^\ep}{\ep}} - \nu )  
            \qquad&\text{ in }  (0,\infty) \times \R \times \{0\}.
        \end{cases}
\end{equation}
If we momentarily suppose that $v^\ep \to \rsig$, which we prove in \Cref{t.uvw}, then it is a classical result \cite{Evans1989pde,Freidlin1986geometric} that $\rsig$ satisfies
\begin{equation}\label{e.HJ}
    \min\{\rsig, \rsig_t + \Hf(\nabla\rsig)\} = 0 \quad \text{ in }(0,\infty)\times \H,
\end{equation}
where we introduce the Hamiltonian
\begin{equation}\label{e.HF0}
    \Hf(q,p)=q^2 + p^2 + 1
\end{equation}
To understand the equation~\eqref{e.HJ}, notice that, by the comparison principle, $v^\eps\geq 0$, and, if $v^\eps > 0$, the exponential term in the first equation of~\eqref{e.scaled_eqn} tends to zero, while the other terms yield the Hamilton-Jacobi equation equation in~\eqref{e.HJ}.  

The boundary condition for $w$ at $y=0$, however, comes from a homogenization process when individuals transition between the two states $U$ and $V$.
 Observe from 
the third equation in~\eqref{e.scaled_eqn} that
\begin{equation}\label{e.c080302}
    \nu e^{\frac{u^\eps - v^\eps}{\eps}}
        = - \frac{\mu v_y^\eps}{\kappa\nu + v_y^\eps} + \mu.
\end{equation}
Additionally, from the exponential terms in~\eqref{e.scaled_eqn}, we expect that $v^\eps - u^\eps \to 0$, so that $u^\eps \to \rsig$ as well. And we obtain, from the second and third equations of \eqref{e.c031901}, 
an additional condition on the boundary: 
\begin{equation}\label{e.HJ_F_0}
         \min\{\rsig,\rsig_t + F_0(\rsig_x,\rsig_y)\} =0
            \quad \text{ on } (0,\infty)\times \mathbb{R}\times \{0\},
\end{equation}
where we have introduced the function
\begin{equation}\label{e.B}
F_0(q,p) = Dq^2 + \Bo(p), \quad \text{ where }\quad 
    \Bo(p) = \begin{cases}
            - \frac{\mu p}{\kappa\nu + p}
                \qquad&\text{ if } p > - \kappa\nu,\\
            +\infty
                \qquad&\text{ if } p \leq - \kappa\nu.
    \end{cases}
\end{equation}
The dynamic boundary condition \eqref{e.HJ_F_0} is understood in the relaxed viscosity sense that we make clear in \Cref{s.strong_solutions} via the definition of ``weak viscosity solution.''  Roughly, it is the standard definition of a viscosity solution, where the subsolution (resp. supersolution) condition involves minimizing (resp. maximizing) the equation and boundary condition when $w$ is ``touched'' by a smooth test function at the boundary.  This leads to our first result:

\begin{theorem}\label{t.uvw}
Suppose that $(U,V)$ solve~\eqref{e.berestycki} with initial data satisfying~\eqref{e.initial_data1}-\eqref{e.initial_data2}.  Then: 
	\begin{enumerate}[(i)]

		\item \label{i.convergence}
			 the rescaled solutions $v^\eps$ and $u^\eps$ converge locally uniformly to a weak viscosity solution $\rsig$ to \eqref{e.HJ}-\eqref{e.HJ_F_0}; 

		\item \label{i.scaling}
			the solution satisfies $\rsig(t, tx, ty) = t\rsig(1,x,y)$ for all $t>0$ and $(x,y) \in \overline \H$;

		\item \label{i.Wulff}
			 the set $\cW = \{(x,y)\in \oH: \rsig(1,x,y)=0\}$ is star-shaped; \footnote{We will show that $\mathcal{W}$ is strictly convex after connecting with the control formulation; see \Cref{c.convex}.}

		\item \label{i.spreading}
		  for all $\eta > 0$,
			\begin{equation}\label{e.conv_to_0}
        			\lim_{t\to+\infty} \left[ 
				\sup_{
					\dist\left(\frac{1}{t}(x,y), \mathcal{W} \right) >\eta} V(t,x,y)  + 
        			\sup_{
					\dist\left(\frac{1}{t}(x,0), \mathcal{W} \right) >\eta} 
				U(t,x) \right] = 0
			\end{equation}
			and
       	\be\label{e.conv_to_1}
			\lim_{t\to+\infty}
				\left[
					\sup_{ 
						\dist\left(\frac{1}{t}(x,y), \H \setminus \mathcal{W} \right) >\eta}|V(t,x,y) -1|
					+ \sup_{ 
						\dist\left(\frac{1}{t}(x,0), \H \setminus \mathcal{W} \right) >\eta}\left|U(t,x) -\frac{\nu}{\mu}\right|
				\right]
				= 0.
		\ee
       
       \end{enumerate}
\end{theorem}
We note that $\cW$ is often referred to as a Wulff shape or asymptotic expansion shape associated to~\eqref{e.berestycki}.

\subsubsection{Comparison principle for $w$: strong viscosity solutions}\label{ss.flux}

One major difficulty in the proof of \Cref{t.uvw} is that, in order to use the half-relaxed limits approach to proving \Cref{t.uvw}.\ref{i.convergence}, the limiting equation must enjoy a comparison principle.  For some time it was an open question whether a comparison principle holds for weak viscosity solutions.  Recently, a breakthrough of Imbert and Monneau~\cite{Imbert2017flux,Imbert2017quasi}, 
established the comparison principle under certain hypotheses on $F_0$ by relating it to a notion of strong viscosity solutions with a suitable optimal control interpretation on the junction or boundary. Lions and Souganidis~\cite{Lions2017well} later provided a simpler proof by leveraging a deep connection with Kirchhoff-type junction conditions. Let us also mention~\cite{Barles2018flux,LionsSouganidis2016} for related work.

Moreover, while \Cref{t.uvw} reduces the front propagation problem of the road-field model~\eqref{e.berestycki} to understanding a limiting Hamilton-Jacobi equation~\eqref{e.HJ}-\eqref{e.HJ_F_0}, it does not give us access to a control formulation.  In particular, it is not immediately clear that $w$ is easier to understand than $(U,V)$.

Both of these issues lead us to search for the appropriate ``flux-limited'' boundary condition that is satisfied by  $w$ in the ``strong'' viscosity sense.  The definition of this is clarified in \Cref{s.strong_solutions}.  We follow \cite{Imbert2017quasi} to briefly derive it here.  

Fix a point $z=(t,x,0)$ on the ``road'' at which $w(z) > 0$, and let $(-\lambda, q,p) \in D^+_{\H}\rsig(z)$ be an element of the superdifferential. 
By interpreting~\eqref{e.HJ_F_0} in the relaxed sense, 
\begin{equation}\label{e.st.2}
 	\min\{\Hf(q,p),F_0(q,p)\} \leq \lambda.
\end{equation}
Being on the boundary, one can reduce the value of $p$ to find a critical slope $\overline p = \overline p(\lambda,q) \geq 0$ such that
\be\label{e.c052903}
	(-\lambda,q,p') \in D^+_{\H}\rsig(P)
	\quad\text{ if and only if }\quad
	p' \geq \overline p
\ee 
(see \Cref{d.subdifferential}). 
A key insight, due to Imbert and Monneau and summarized in \Cref{l.critsub}, is that
\begin{equation}\label{e.st.3}
    \Hf(q,\overline p) \leq \lambda.
\end{equation}
This is not obvious because~\eqref{e.st.2} involves a minimum of $H$ and $F_0$.  It is obtained by ``pushing'' the argument into the interior of $\H$, where $\rsig_t + H(\nabla \rsig) \leq 0$.
Let $\Hf^-$ be the nonincreasing-in-$p$ part of $\Hf$:
\be\label{e.Hf-}
	\Hf^-(q,p)
		:= \inf_{p' \leq p} \Hf(q,p')
		= \Hf(q,p_+)
		= p_+^2 + q^2 + 1.
\ee
Here $p_+ = \max\{p,0\}$ is the positive part of $p$.  
Since $p \geq \overline p$, then
\begin{equation}\label{e.st.4}
	\Hf^-(q,p)
		\leq \Hf(q, \bar p)
		\leq \lambda.
\end{equation}
Next, we define the {\em flux limiter} (see \cite{Imbert2017flux,Imbert2017quasi}): 
\be
	\tH(q) := \begin{cases}
	    \min \Hf(q,\cdot) = \Hf(q,0) & \text{ when }\Hf(q,0) \geq F_0(q,0),\\
     \sup_{p'>0} \min\{\Hf(q,p'), F_0(q,p')\} & \text{ when }\Hf(q,0) < F_0(q,0),
	\end{cases}
\ee
Let us mention three facts about $\tH$: (1) it is not immediately obvious that $\tH$ is convex, however, we show this in \Cref{lem.tH}; (2) we name $\tH$ in this way because it is the effective Hamiltonian
on the road in the control formulation (cf.~\eqref{e.Jcon}); (3) either $F_0(q,0) \leq \Hf(q,0)$  or there is $p_q > 0$ such that $\Hr(q) = F_0(q,p_q) = \Hf(q,p_q)$.  For this last point, we used that $\Hf(q,\cdot)$ is increasing in $p$ and $F_0(q,\cdot)$ is decreasing in $p$, for $p>0$.  Actually, this leads to the more explicit form~\eqref{e.tH} of $\tH$.

To motivate the definition of strong solutions, we claim, in addition to \eqref{e.st.4}, that
\be\label{e.Hr}
    H_r(q) \leq \lambda.
\ee
Indeed, if $F_0(q,0) \leq  \Hf(q,0)$ then $\tH(q)= \Hf(q,0) \leq \Hf(q,\bar p)$. We are then finished by applying~\eqref{e.st.3}.  Otherwise, $F_0(q,0) >  \Hf(q,0)$ and it follows from
 observation (3) above imply that there is $p_q > 0$ such that
\be\label{e.c050201}
	\tH(q) = \Hf(q,p_q) = F_0(q,p_q).
\ee
Since $p \mapsto \Hf(q,p)$ is increasing on $\R_+$, we deduce~\eqref{e.Hr} immediately from~\eqref{e.st.3} if $p_q \leq \bar p$.  If $p_q > \bar p$ then $(-\lambda, q, p_q) \in D^+_{\H} w(z)$, and again $H_r(q) \leq \lambda$,  thanks to~\eqref{e.st.2}-\eqref{e.c052903}. 
This concludes the justification of~\eqref{e.Hr}.

Combining \eqref{e.st.4} and \eqref{e.Hr}, we derive the strong  condition on the boundary:
\begin{equation}\label{e.F}
	F(q,p):= \max\{\Hf^-(q,p),\tH(q)\} \leq \lambda.
\end{equation}
whenever $(-\lambda,q,p) \in D^+_{\H}\rsig(P)$. 
This is to be compared with the relaxed condition \eqref{e.st.2}, where a minimum is involved.

Motivated by the above, we introduce the following flux-limited boundary condition:
\begin{equation}\label{e.HJ_F}
\min\{\rsig, \rsig_t + F(\rsig_x,\rsig_y)\} = 0 \quad \text{ in }(0,\infty) \times \R\times \{0\},
\end{equation}
By taking into account the explicit expressions of $H$ and $F_0$, we may compute the following:
\begin{equation}\label{e.tH}
    \tH(q)
        = q^2 + (p_q)^2 + 1
\end{equation}
with
\begin{equation}\label{e.pq}
    p_q = \begin{cases}
        0 \qquad &\text{ if } q^2 \leq \frac{1}{D-1},\\
        g^{-1}(q)
            \qquad &\text{ if } q^2 > \frac{1}{D-1}.
    \end{cases}
\end{equation}
Here $g:[0,\infty) \to  [\sfrac{1}{\sqrt{D-1}}, \infty)$ is the increasing function
\begin{equation}\label{eq:g}
    g(p) = \sqrt{ \frac{1}{D-1} \left[p^2 + 1 + \frac{\mu p}{\kappa\nu + p}\right]}.
\end{equation}

Making the above arguments more rigorous and precise, we obtain the following result.  Let us again note that the exact definition of a strong viscosity solution is given in \Cref{s.strong_solutions}. 
\begin{theorem}\label{t.strong}
    Under the assumptions of \Cref{t.uvw}, the limiting solution $\rsig$ is a strong viscosity solution to~\eqref{e.HJ}-\eqref{e.HJ_F}.
\end{theorem}

\subsubsection{The control formulation and $w$}

We now connect the solution $w$ to a control formulation. 
Define the value function 
\begin{equation}\label{e.Jcon}
    J(t,x,y)
    	= \min_{\gamma \in N(t,x,y)}
		\int_0^t
			\hat{L}(\gamma(s), \dot \gamma(s)) \,ds,
\end{equation}
the infimum is taken over the set
\be\label{e.N}
	N(t,x,y)
		= \left\{ \gamma \in H^1(0,t):
			~\gamma(0)=(0,0),
			~\gamma(t)=(x,y),
			\gamma(s) \in \overline \H \text{ for all } s \in[0,t]
			\right\},
\ee
and where the Lagrangian $\hat{L}(x,y,v_1,v_2)$ is given by
\begin{equation}\label{e.hatL}
    \hat{L}(x,y,v_1,v_2) = \Lf(v_1,v_2) \mathds{1}_{y>0} + \Lr(v_1)\mathds{1}_{y = 0},
\end{equation}
with the Lagrangians on the field ($\Lf$) and road ($\Lr$) given by
\be
	\begin{split}
	&\Lf(v) = \max_{(q,p)} \Big[v\cdot (q,p) -(q^2 + p^2 +1)\Big]= \frac{|v|^2}{4}-1 \quad \text{ and }
	\\&
	\Lr(v_1) = \max_{q} \Big[v_1 q - \tH(q)\Big].
	\end{split}
\ee
These are, respectively, the Legendre transforms of $\Hf(q,p) = q^2 + p^2 +1$ and the effective Hamiltonian $\tH(q)$. 
Let us note that $J$ has the following obvious scaling symmetry (see also \Cref{l.hr}.\ref{i.rho^*}):
\be\label{e.scaling}
	J(t,x,y) = t J(1, \sfrac{x}{t}, \sfrac{y}{t})
\ee
This allows one to focus on analyzing simply $J(1,\cdot)$.

It follows from \cite[Theorem 6.4]{Imbert2017flux} that $J$ is the unique solution satisfying, in the strong viscosity sense,
\begin{equation}\label{e.linear_HJ}
\begin{cases}
    J_t + \Hf(J_x,J_y)=0 &\text{ in }(0,\infty)\times \H,\\
    J_t + F(J_x,J_y) = 0 &\text{ in }(0,\infty)\times \mathbb{R},
\end{cases}
\end{equation}
with initial data
\be
\liminf_{(t',x',y') \to (0^+,x,y)}J(t',x',y') = \begin{cases}
    0  &\text{ when }(x,y) = (0,0),\\
    +\infty &\text{ otherwise}.
\end{cases}
\ee
Indeed, \eqref{e.linear_HJ} is the (flux-limited) Hamilton-Jacobi equation that arises from the linearization of~\eqref{e.berestycki} at the trivial solution.

\begin{theorem}\label{t.w=J+}
    Under the assumptions of \Cref{t.uvw}, we have
    \be\label{e.w=J+}
        \rsig(t,x,y) = \max\{0, J(t,x,y)\}.
    \ee
	As a consequence, the Wulff shape can also be written as follows
	\be\label{e.a0722.10}
		\mathcal{W} = \{(x,y) \in \H:~J(1,x,y) \leq 0\},
	\ee
	and, for $(t,x,y) \in (0,\infty) \times \H$,
	\be
		V^\eps(t,x,y) \to
			\begin{cases}
				1
					\qquad &\text{ if } (t,x,y) \in \{J < 0\}\\
				0
					\qquad &\text{ if } (t,x,y) \in \{J > 0\}
			\end{cases}
		\quad\text{ and }\quad
			U^\eps(t,x) \to
			\begin{cases}
				\frac{\nu}{\mu}
					\qquad &\text{ if } (t,x,0) \in \{J < 0\}\\
				0
					\qquad &\text{ if } (t,x,0) \in \{J > 0\}.
			\end{cases}
	\ee
\end{theorem}
See \Cref{fig:H2} for a computation of the Wulff-shape $\mathcal{W}$. 
Two easy corollaries of \Cref{t.w=J+} are the following:
\begin{corollary}
\label{c.convex}
    $\mathcal{W}$ is strictly convex for all $D>1$.
\end{corollary}

Surprisingly, among the class of all conical domains in $\mathbb{R}^2$, the case $\Omega = \H$ is the only case where the Wulff shape is convex for all $D>1$; see \Cref{p.non_convexity}. \Cref{c.convex} is a direct consequence of \eqref{e.a0722.10} and \Cref{l.W_convex} concerning the strict convexity of level set $\{J \leq 0\}$ (cf. \Cref{p.non_convexity}).

\begin{corollary}\label{c.angular_speed}
For each $\vt \in [-\sfrac\pi2,\sfrac\pi2]$, there exists a directional spreading speed $c_*(\vt)>0$ such that
\begin{equation}\label{e.directional}
    \begin{split}
    	&\lim_{t\to\infty} V(t,x + ct\sin\vt, y + ct\cos\vt) = 1 \quad \text{ if } 0\leq c < c_*(\vt), \qquad \text{ and}
 	   \\
    	&\lim_{t\to\infty} V(t,x + ct\sin\vt, y + ct\cos\vt) = 0 \quad \text{ if } c > c_*(\vt),
	\end{split}
\end{equation}
locally uniformly in $(x,y) \in \R^2$.    
\end{corollary}

To obtain \Cref{c.angular_speed},  simply notice that $J$ is strictly radially increasing (\Cref{l.monotonicity}.\ref{i.radially_increasing}), so that there is a unique value such that
\be
	J(1, c_*(\vt) \sin\vt, c_*(\vt)\cos\vt) = 0.
\ee
As such, we omit the proof as it follows directly from the tools we develop in the sequel.

In \Cref{s.control_corollaries}, we deduce several further results from \Cref{t.w=J+} by understanding the optimal paths $\gamma$ in~\eqref{e.Jcon}.  In particular:
\begin{itemize}

	\item (\Cref{p.straight_lines}) Optimal paths stick to the road $\{y=0\}$ until a time $\tau_0 \in [0,1]$ when they proceed through the field along straight lines.  This yields a Lax-Oleinik-type formula for $J$.

	\item (\Cref{p.roadspeed}) A characterization of the speed along the road $c_*(\sfrac\pi2)$ in terms of $\tH$ from which it is easy to deduce that $c_*(\sfrac\pi2) = 2$ for $D \leq 2$, $c_*(\sfrac\pi2) > 2$ for $D>2$, and $c_*(\sfrac\pi2) \sim \sqrt D$ as $D\to\infty$ (\Cref{c.D_near_2} and \Cref{c.large_D}).

	\item (\Cref{l.W_convex}) The Wulff shape $\cW$ is strictly convex.  Later, we see that optimal paths to the front do {\em not} follow speed $c_*(\sfrac\pi2)$ on the road for some time and speed $2$ in the field for the remainder of the time; instead, they move faster than speed $c_*(\sfrac\pi2)$ on the road and slower than speed $2$ in the field (\Cref{p.no_Huygens}). In particular, Huygen's principle does not hold.

	\item Further, a portion of the boundary of $\cW$ is a circle of radius two and, if $D>2$, a portion ``lifts off'' this circle.  More precisely, there is $\vt_*>0$ such that $c(\vt) = 2$ for $|\vt| \leq \vt_*$ and $c(\vt)$ is strictly increasing for $\vt> \vt_*$ (\Cref{prop:1.11}).

\end{itemize}
The results about the optimal paths are (necessarily) new, while those about the Wulff shape, or equivalently, the speed $c_*(\vt)$, recover known results of~\cite{BRR_JMB,Berestycki2016shape}.  Our approach, however, is quite different.

\subsection{Conical domains}\label{ss.conical_domains}

We discuss a novel and natural extension of our approach to the case when the half space $\H$ is replaced by a general conical domain:
for a fixed $a \in (0,\sfrac\pi2]$, let
\be\label{e.Omega_a}
\Omega_a = \{(r\sin \vt, r\cos\vt):~ r>0,~ \vt \in (\sfrac\pi2 -2a,\sfrac\pi2)\}.
\ee
In particular, $\Omega_{\sfrac\pi2}=\H$ and $\Omega_{\pi/4}$ is the first quadrant of $\mathbb{R}^2$.  Here, there are two portions of the road:
\be
    \Gamma_0 = (0,\infty)\times\{0\}
    \qquad\text{ and }\qquad
    \Gamma_a = \{(x,y): y>0, \sfrac{x}{y} = \tan(\sfrac\pi2-2a)\}.
\ee
Then we arrive at the system for $(V,U)$
\begin{equation}\label{e.eqn11}
    \begin{cases}
        V_t - \Delta V = V(1-V)
           \qquad&\text{ in }  (0,\infty) \times \Omega_a,\\
        U_t - D U_{xx} = \nu V|_{y=0} - \mu U
            \qquad&\text{ in } (0,\infty) \times (0,\infty),\\
        -V_y(t,x,0) = \kappa [\mu U(t,x) - \nu  V(t,x,0)]
            \qquad&\text{ in } (0,\infty) \times (0,\infty),
        \\
        \tilde U_t - \tilde D \tilde U_{xx} = \tilde \nu \tilde V|_{y=0} - \tilde \mu \tilde U
            \qquad&\text{ in } (0,\infty) \times (0,\infty),\\
        -\tilde V_y(t,x,0) = \tilde \kappa [\tilde \mu \tilde U(t,x) - \tilde \nu  \tilde V(t,x,0)]
            \qquad&\text{ in } (0,\infty) \times (0,\infty),
        \end{cases}
\end{equation}
where 
\be\label{e.Psi_tilde}
    \tilde{V}(t,x,y) = V(t,\Psi_a(x,y))
    \qquad\text{ and }\qquad
    \tilde{U}(t,x) = U(t, \Psi_a(x,0)),
\ee
with
\be\label{e.Psi_a}
    \Psi_a \text{ being the reflection that takes $\Gamma_0$ and $\Gamma_a$ to each another.}
\ee
Here $U$ is the population on the roads $\Gamma_0\cup \Gamma_a$ and $V$ is the population on the road $\Gamma_a$.  In general, we only work with $U|_{\Gamma_0}$ and $\tilde U|_{\Gamma_0}$ for convenience.  See \Cref{f.Omega_a}.

\begin{figure}
\centering
    \begin{tikzpicture}
        \def\a{25} 
        \def\length{3} 
    
        \draw[<->] (-1.1*\length,0) -- (1.1*\length,0) node[right] {$x$};
        \draw[->] (0,-.5) -- (0,1.1*\length) node[above] {$y$};
    
        \draw[violet] (0,0) -- ({\length*cos(2*\a)}, {\length*sin(2*\a)});
        \draw[violet] (0,0) -- (\length, 0);
    
        \draw[red, dashed]
            (0,0) -- ({\length*cos(\a)}, {\length*sin(\a)})
            node [above right] { $\Gamma_{\sfrac{a}2}$};
    
        \draw ({.95*(\length/3)*cos(\a)},{.95*(\length/3)*sin(\a)}) arc[start angle=\a, end angle={2*\a}, radius=.95*(\length/3)];
        
        \draw ({1.05*(\length/3)*cos(\a)},{1.05*(\length/3)*sin(\a)}) arc[start angle=\a, end angle={2*\a}, radius=1.05*(\length/3)]
            node[midway, above right]{\tiny $a$};
        
        \draw ({\length/3},0) arc[start angle=0, end angle=\a, radius=1]
            node[midway, right]{\tiny $a$};
        
        \draw[dotted] ({(\length/2)*cos(2*\a)},{(\length/2)*sin(2*\a)}) arc[start angle=2*\a, end angle=90, radius=\length/2]
            node[midway, above ]{\tiny $\frac\pi2-2a$};
    
        \fill[blue, opacity=0.1]
            (0,0) -- plot[domain=0:{2*\a}, samples=50] ({\length*cos(\x)}, {\length*sin(\x)}) -- 
            ({\length*cos(\a)}, {\length*sin(\a)}) -- ({\length*cos(2*\a)}, {\length*sin(2*\a)}) -- cycle;

        \node[blue] at ({\length*.85*cos(\a/2)},{\length*.85*sin(\a/2)}) {$\Omega_a$};

        \node[violet, below] at ({\length*.85},0) {$\Gamma_0$};
        \node[violet,left] at ({\length*.85*cos(2*\a)},{\length*.85*sin(2*\a)}) {$\Gamma_a$};
    \end{tikzpicture}
\qquad
    \begin{tikzpicture}
        \def\a{55} 
        \def\length{3} 
    
        \draw[<->] (-1.1*\length,0) -- (1.1*\length,0) node[right] {$x$};
        \draw[->] (0,-.5) -- (0,1.1*\length) node[above] {$y$};
    
        \draw[violet] (0,0) -- ({\length*cos(2*\a)}, {\length*sin(2*\a)});
        \draw[violet] (0,0) -- (\length, 0);
    
        \draw[red, dashed]
            (0,0) -- ({\length*cos(\a)}, {\length*sin(\a)})
            node [above right] { $\Gamma_{\sfrac{a}2}$};
    
        \draw ({.95*(\length/3)*cos(\a)},{.95*(\length/3)*sin(\a)}) arc[start angle=\a, end angle={2*\a}, radius=.95*(\length/3)];
        
        \draw ({1.05*(\length/3)*cos(\a)},{1.05*(\length/3)*sin(\a)}) arc[start angle=\a, end angle={2*\a}, radius=1.05*(\length/3)]
            node[midway, above right]{\tiny $a$};
        
        \draw ({\length/3},0) arc[start angle=0, end angle=\a, radius=1]
            node[midway, above right]{\tiny $a$};
    
        \fill[blue, opacity=0.1]
            (0,0) -- plot[domain=0:{2*\a}, samples=50] ({\length*cos(\x)}, {\length*sin(\x)}) -- 
            ({\length*cos(\a)}, {\length*sin(\a)}) -- ({\length*cos(2*\a)}, {\length*sin(2*\a)}) -- cycle;

        \node[blue] at ({\length*.85*cos(\a/2)},{\length*.85*sin(\a/2)}) {$\Omega_a$};

        \node[violet, below] at ({\length*.65},0) {$\Gamma_0$};
        \node[violet,left] at ({\length*.65*cos(2*\a)},{\length*.65*sin(2*\a)}) {$\Gamma_a$};
        
    \end{tikzpicture}

    \caption{Two examples of the domain $\Omega_a$, the first with angle $a < \sfrac\pi4$ and the second with $a> \sfrac\pi4$.}\label{f.Omega_a}
\end{figure}
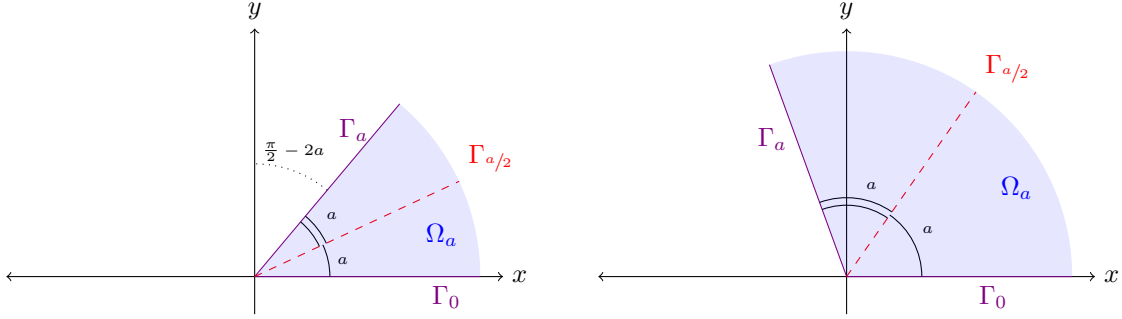

Let us note that the well-posedness of~\eqref{e.eqn11} is open due to the corner at $(0,0)$.  Existence is simple: by smoothing $\Omega_a$ to $\Omega_{a,\eps}$, one easily constructs an approximate solution $(V^\eps, U^\eps)$.  By compactness, we may take $\eps \to 0$ to obtain a solution on $\Omega_a$.  Instead, it is uniqueness that is not obvious.  As it is not our interest to settle that here, we avoid these technical details below.

Domains of this form were considered by Ducasse \cite{Ducasse2018influence} who, by generalizing the arguments of \cite{Berestycki2016shape}, showed that the speed of the road is not affected by the value of $a$ when the diffusion on both portions of the road is identical; that is, $D = \tilde D$. The spreading in the interior and the behavior if $D\neq \tilde D$ is a novel contribution of this work.

In \Cref{sec:ext}, we show that the study of the spreading of $(V,U,\tilde U)$ reduces to studying the solution $w_a$ of a Hamilton-Jacobi equation similar to~\eqref{e.HJ}-\eqref{e.HJ_F}.  Moreover, this equation has a control formulation:
\begin{equation}\label{e.Jcon2}
    J_a(t,x,y)
        = \inf_{\gamma \in N_a(t,x,y)}
            \int_0^t \left[
                \Lf(\dot\gamma(s)) \mathds{1}_{\{\gamma(s)\in\Omega_a\}}
                + \Lr(\dot\gamma(s))\mathds{1}_{\{\gamma(s) \in \Gamma_0\}}
                + \Lr(\dot{\tilde\gamma}(s))\mathds{1}_{\{\gamma(s) \in \Gamma_a\}}\right]\,ds,
\end{equation}
where the class of admissible paths $N_a$ is defined analogously as in~\eqref{e.N} and $\tilde \gamma = \Psi_a \gamma$. 
This leads to a characterization of the Wulff shape, along the lines of \Cref{t.w=J+}.  Let us roughly state this here in the simplest case when $\mu = \tilde \mu$, $\nu = \tilde \nu$, $\kappa = \tilde \kappa$ and $D = \tilde D$.  Further results when $D\neq \tilde D$ are discussed in \Cref{sec:ext}.

\begin{theorem}\label{t.conical}
Let $(V,U,\tilde U)$ be any solution of \eqref{e.eqn11} with initial data analogous to~\eqref{e.initial_data1}-\eqref{e.initial_data2} and $(v^\eps, u^\eps, \tilde u^\eps)$ are defined analogously as in~\eqref{e.thin_front_scaling}, then
    \be
        v^\eps, u^\eps, \tilde u^\eps \to w_a
    \ee
    locally uniformly, with
    \be
        w_a(t,x,y) = \max\{0, J_a(t,x,y)\}
    \ee
    and
    \be\label{e.c060701}
        J_a(t,x,y) = \min\{J(t,x,y), J(t,\Psi_a(x,y))\}.
    \ee
    Moreover, for every $(t,x,y) \in (0,\infty) \times \H$
    \be
		V^\eps(t,x,y) \to
			\begin{cases}
				1
					\qquad &\text{ if } (t,x,y) \in \{J_a < 0\}\\
				0
					\qquad &\text{ if } (t,x,y) \in \{J_a > 0\}
			\end{cases}
		\quad\text{ and }\quad
			U^\eps(t,x) \to
			\begin{cases}
				\frac{\nu}{\mu}
					\qquad &\text{ if } (t,x,0) \in \{J_a < 0\}\\
				0
					\qquad &\text{ if } (t,x,0) \in \{J_a > 0\}.
			\end{cases}
    \ee
    so that the Wulff shape is
    \be
        \cW_a =\{ (t,x,y)\in\overline \Omega_a: J_a(1,x,y) \leq 0\}.
    \ee
\end{theorem}

Let us note an important aspect of this result: it connects the problem on the conical domain $\Omega_a$ with the problem on the half-space $\H$ by relating $J_a$ and $J$.  Indeed, by symmetry, we see that $J_a(t,x,y) = J(t,x,y)$ if $(x,y)$ is closer to the road $\Gamma_0$ and is $J(t,\Psi_a(x,y))$ otherwise.  In this way, we can leverage the results on $\H$ to the more general case.

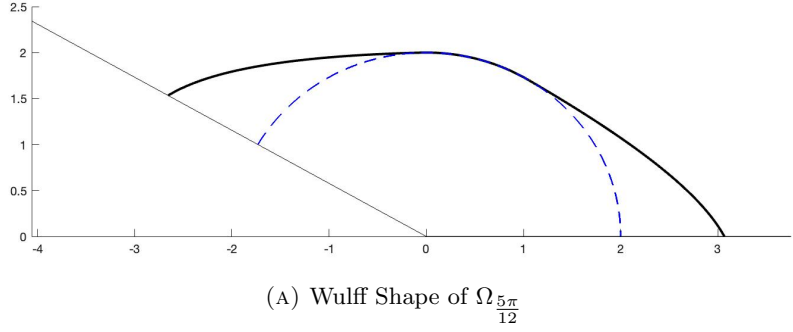
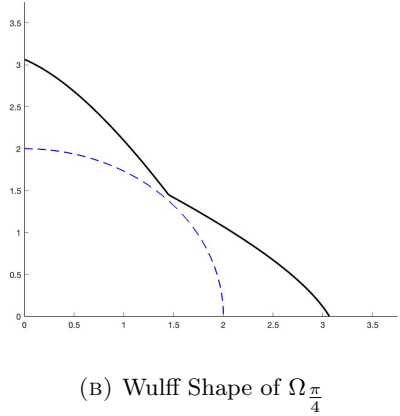
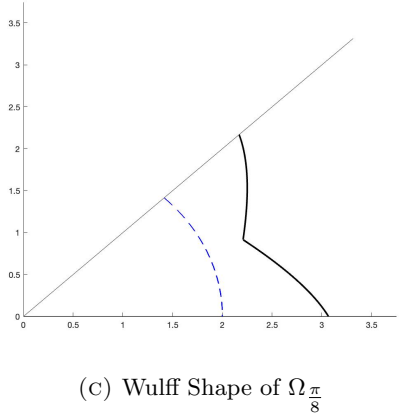
\begin{figure}
    \centering
     \begin{subfigure}[b]{\textwidth}
        \centering
        \resizebox{.8\textwidth}{!}{
        \begin{tikzpicture}
            \node[anchor=south west,inner sep=0] at (0,0) {\includegraphics[width=\textwidth]{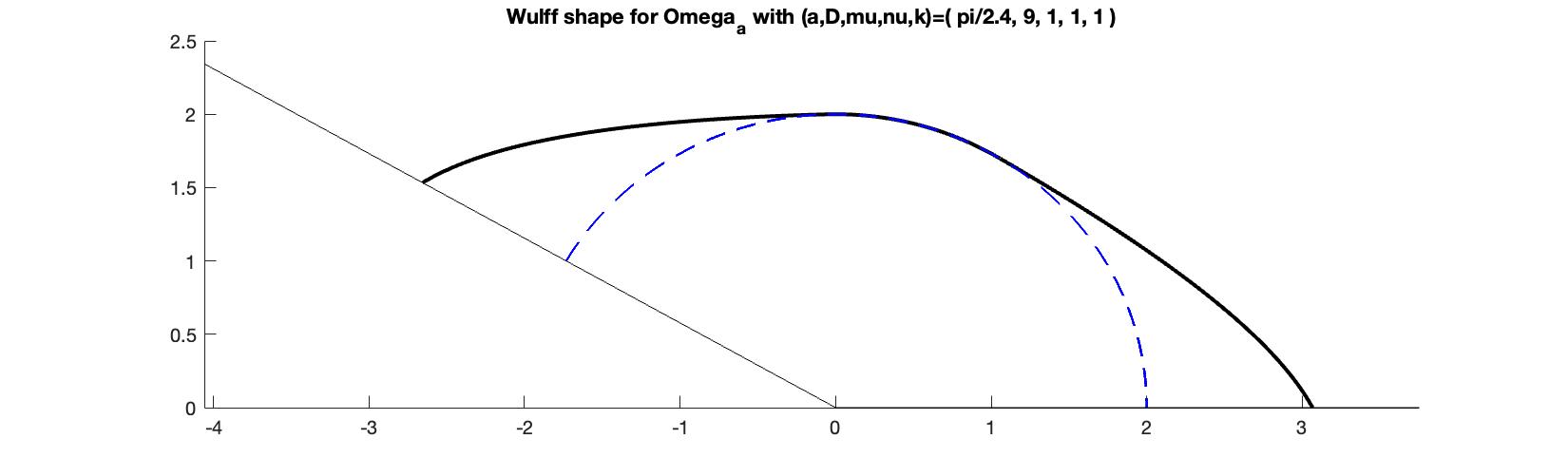}};
            \draw[white,fill=white] (5.2,4.4) rectangle (11.8,5);
    \end{tikzpicture}
    }
    \caption{Wulff Shape of $\Omega_{\tfrac{5\pi}{12}}$}
    \label{fig:obtuse}
    \end{subfigure}
    \hfill
 \begin{subfigure}[b]{.4\textwidth}
        \centering
        \resizebox{\textwidth}{!}{
        \begin{tikzpicture}
            \node[anchor=south west,inner sep=0] at (0,0) {\includegraphics[width=\textwidth]{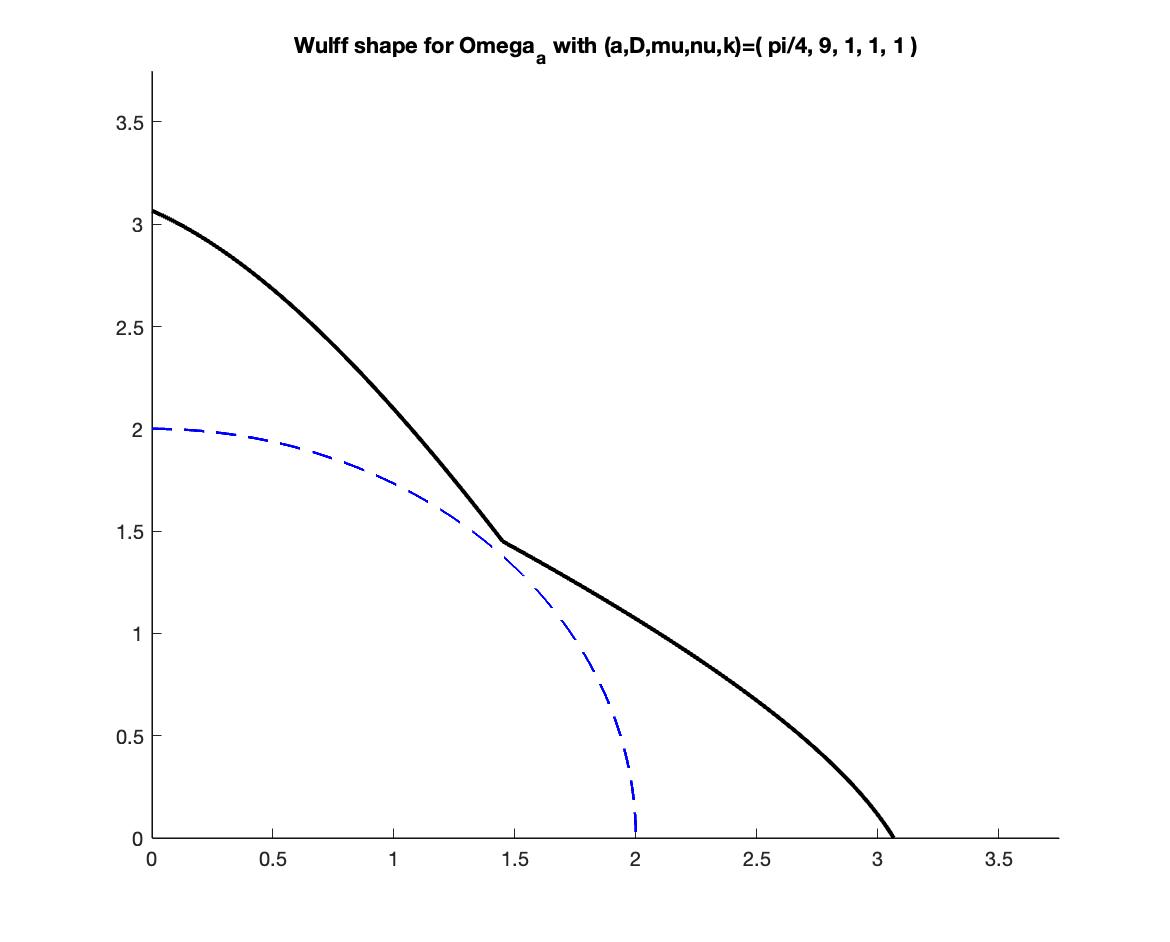}};
            \draw[white,fill=white] (1.5,4.8) rectangle (5.3,5.15);
    \end{tikzpicture}
    }
    \caption{Wulff Shape of $\Omega_{\tfrac{\pi}{4}}$}
    \label{fig:quad}
\end{subfigure}  
\begin{subfigure}[b]{.4\textwidth}
        \centering
        \resizebox{\textwidth}{!}{
        \begin{tikzpicture}
            \node[anchor=south west,inner sep=0] at (0,0) {\includegraphics[width=\textwidth]{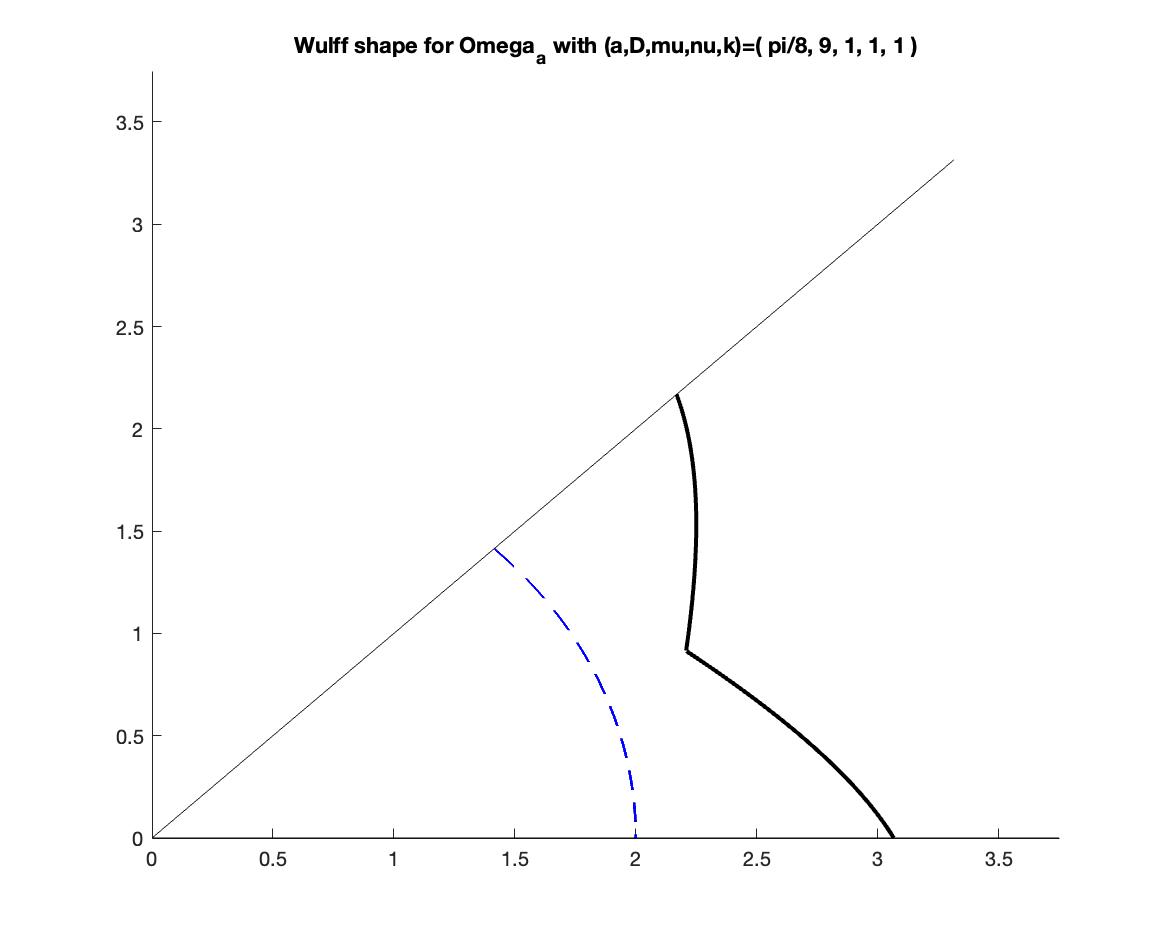}};
            \draw[white,fill=white] (1.5,4.8) rectangle (5.3,5.15);
    \end{tikzpicture}
    }
    \caption{Wulff Shape of $\Omega_{\tfrac{\pi}{8}}$}
    \label{fig:acute}
\end{subfigure}
        \caption{The boundary of $\mathcal{W}_a$, with different values of $a=\tfrac{5\pi}{12}, \tfrac{\pi}{4}, \tfrac{\pi}{8}$.  Here we take $\tilde D = D = 9$ and we set all other parameters to $1$. The dashed line represents the Wulff shape without enhancement by the road (the circle with radius $2$), and the solid line represents the Wulff shape with enhancement by the road. The road speed is given by $c_*(\sfrac\pi2)\approx 3.1243$.}
        \label{fig:cones}
\end{figure}

Using this, we derive further properties of the Wulff shape essentially ``for free'' from the results on $\H$:
\begin{itemize}
    \item We immediately recover from the discussion above that the speed on the roads $\Gamma_0$ and $\Gamma_a$ are $c_*(\sfrac\pi2)$.

    \item Further, we see that speeds in the field in the subdomain $\Omega_{\sfrac{a}{2}}$, which is the area under the dashed red line in \Cref{f.Omega_a}, are unchanged from the half-space case.  More succinctly,
    \be
        c_{*a}(\vt)
            = c_*(\vt)
            \qquad \text{ for $\vt > \sfrac\pi2 - a$}.
    \ee
    By symmetry, it follows that
    \be 
        c_{*a}(\vt)
            = c_{*a}(\pi - 2a - \vt)
            = c_{*}(\pi - 2a - \vt)
    \ee
    when $\vt \in [\sfrac\pi2-2a,\sfrac\pi2-a]$.  Equivalently,
    \be
        \cW_a
            = \left(\cW \cup \Psi_a(\cW) \right) \cap \overline{\Omega_a}.
    \ee
    See \Cref{p.J_a_J}.

    \item Interestingly, we see that the Wulff shape loses convexity for any $a\in (0,\sfrac\pi2)$ when $D$ is sufficiently large (\Cref{p.non_convexity}).  In particular, we obtain a lower bound $D \geq {4(2+\mu)^2} \csc^2{2a}$.  See \Cref{fig:acute,fig:quad} for numerical simulations illustrating this.  As we pointed out above, this implies that the interior behavior cannot be bounded by the types of planar supersolutions used in previous works.
    
\end{itemize}

We also discuss the behavior when $D \neq \tilde D$.  We prove that the speed on the slower road can be further enhanced by the effect on the faster road, while the speed on the faster road remains unchanged compared to the case on $\H$ (\Cref{p.tilde_D > D}).  Our results here are less precise because~\eqref{e.c060701} may not hold (cf.~\Cref{lemma:4.4}).

\subsubsection{Noncompactly supported initial data}

It is not difficult to see that our results easily generalize to initial data that is not compactly supported, although some care has to be taken as it is known that initial data with ``too slow'' of decay yield accelerating fronts~\cite{HamelRoques,HendersonFast}. 
If we assume that 
\be
	V_0(r\sin \vt,r\cos\vt)
		= e^{-(h(\vt)+o(1))r}, \quad U_0(r)
		= e^{-\left(h(\sfrac{\pi}{2})+o(1)\right)r}
			\quad \text{ as }r \to +\infty,
\ee
for some positive function $h\in C([-\sfrac\pi2,\sfrac\pi2])$ then our proofs go through nearly exactly, with only the step involving the comparison principle changed (and that step becomes easier because $w_0$ is continuous and finite) and with $J$ replaced by
\be
    J_h(t,x,y)
        = \inf_{\substack{\gamma \in H^1,\\ \gamma(t)= (x,y)}}\Bigg\{ \int_0^t \hat L(\gamma(s), \dot \gamma(s)) ds
        + w_0(\gamma(0))\Bigg\},
\ee
such that $w_0(x,y) = rh(\vt)$, in polar coordinates.  
We omit further discussion.

\subsection{Organization}
We organize the paper as follows.
We begin by demonstrating the strength of the control formulation by using it and the interpretation of the front as $\{J = 0\}$ to deduce fairly precise estimates on the Wulff shape in the half space case $\H$, as well as the general conical case $\Omega_a$.

The main work to establish the connection between the control formulation~\eqref{e.Jcon} and the original reaction-diffusion problem~\eqref{e.berestycki} in the half-space $\H$ occurs in \Cref{s.strong_solutions,s.half_relaxed,sec:w=J+}.  Specifically, \Cref{s.strong_solutions} gives the definition of weak and strong viscosity solutions, \Cref{s.half_relaxed} gives the main arguments to prove \Cref{t.uvw}, and \Cref{sec:w=J+} gives the proof of \Cref{t.w=J+}. 
Extensions of these results to conical domains $\Omega_a$ are discussed in \Cref{sec:ext}. All technical lemmas are contained in \Cref{s.technical}.

Finally, we include an appendix that contains a proof of the comparison principle for strong solutions of our Hamilton-Jacobi equations (\Cref{thm:scp}).  The main novelty here is that we allow for infinite initial data.

\subsection{Notation}

We use big-oh and little-oh notation throughout the manuscript.  In particular, when a limit is taken, say $\eps \to 0$, we use $b_\eps = o(a_\eps)$ and $\tilde b_\eps = O(a_\eps)$ to mean that
\be
	\lim_{\eps\to 0} \frac{b_\eps}{a_\eps} = 0
	\quad\text{ and }\quad
	\limsup_{\eps \to 0} \frac{|\tilde b_\eps|}{|a_\eps|} < \infty.
\ee
Alternatively, we write $b_\eps \sim a_\eps$ if $b_\eps = O(a_\eps)$ and $a_\eps = O(b_\eps)$.

We use weak solution and strong solution for, respectively, weak viscosity solution and strong viscosity solution (the precise definition of this is given in \Cref{s.strong_solutions}).  Aside from the equations relating to $u^\eps$ and $v^\eps$, all other equations are considered in the viscosity sense, so there is no risk of confusion.
Let us note that there is some ambiguity in the terminology for these solutions.  For example, they are sometimes also called flux-limited solutions.  We opt for simplicity here and follow the example of~\cite{Guerand2017effective}.

\subsubsection{Acknowledgements}

Both authors warmly acknowledge support of the Institut Henri Poincar\'e (UAR 839 CNRS-Sorbonne
Universit\'e), and LabEx CARMIN (ANR-10-LABX-59-01) during the thematic program ``‘Mathematical modeling of organization in living matter,'' during the course of which this project was begun.  CH was supported by NSF grants DMS-2204615 and DMS-233766. KYL is supported by NSF grant DMS-2325195.

\section{Consequences of the control formulation~\eqref{e.Jcon}}
\label{s.control_corollaries}

\Cref{s.control_corollaries}, together with part of \Cref{s.technical}, give a self-contained treatment of the value function $J$ of the optimization problem given in \eqref{e.Jcon}, and properties of its level sets. For simplicity, we will slightly abuse notation to denote \be\mathcal{W} = \{(x,y):~ J(1,x,y) \leq 0\}.\ee Later, after establishing \Cref{t.w=J+}, we see that this is consistent with the definition of $\cW$ given in \Cref{t.uvw}.\ref{i.Wulff}.

Before embarking on this, let us note that the fact that $\tH$ is convex and coercive is not obvious for $q > \sfrac{1}{\sqrt{D-1}}$.  We state this fact here and prove it by complicated, but elementary, calculus in \Cref{ss.control_lemmas}.
\begin{lemma}\label{lem.tH}
The Hamiltonian $\tH$ is strictly convex and coercive.
\end{lemma}

We start with a basic proposition that characterizes the optimal paths in~\eqref{e.Jcon}, as well as provides various qualitative properties of $J$.  The first conclusion in \Cref{p.straight_lines} has a nice ecological interpretation: individuals who wish to invade as quickly as possible will remain on\footnote{Actually, after scaling, being ``on'' the road represents, in the original variables, paths that stay near the road, hopping back and forth to take advantage of the faster propagation of the road and the ability to reproduce in the field.} the road for a portion of time $[0,\tau_0]$ and then follow a straight line path in the field for the remainder of time $[\tau_0,1]$.  This yields~\eqref{e.Lq3a}, which is a type of Lax-Oleinik formula.

\begin{proposition}\label{p.straight_lines}
Fix any $(x,y) \in \overline \H$ with $x\geq 0$. 
There is a unique minimizing path $\gamma$  in the control formulation~\eqref{e.Jcon}.  Further, there exists $\tau_0,z_0,p_0,q_0$ such that
		\be\label{e.Lq3b}
			\gamma(s)
				= \begin{cases}
				    \frac{s}{\tau_0}(z_0,0)
						&\text{ for }s\in[0,\tau_0],\\
					(z_0,0) + (s-\tau_0) \left( 2q_0, 2p_0\right) 
						&\text{ for }s \in [\tau_0,1].
				\end{cases}
		\ee
		In particular, recalling~\eqref{e.hatL}, we have
	  	\be\label{e.Lq3a}
			J(1,x,y)
				= \min_{0\leq z \leq x\atop 0 \leq \tau\leq t}
					\left[
						(1-\tau) \Lf\left(\frac{x-z}{1-\tau}, \frac{y}{1-\tau}\right)  
						+ \tau \Lr\left( \frac{z}{\tau}\right)
					\right],
		\ee
		with the convention that $\tau = 0$ implies $\tau \Lr\left(\tfrac{z}{\tau}\right) = 0$ (and similarly if $\tau = 1$).
		Furthermore, $\nabla J \in C_{\rm loc}((0,\infty)\times \H)$ and is given by
  \be\label{e.a0722.11}
\nabla J(t,x,y) = \frac{1}{2}\dot\gamma(t) = (q_0,p_0) \quad \text{ for } (t,x,y) \in (0,\infty)\times\overline \H.
  \ee
		Finally, if $y=0$, then the minimum occurs for $(z_0,\tau_0) = (x,t)$ and
		\begin{equation}\label{e.a.240719.1}
			J(t,x,0) = t\Lr\left(x/t\right) \quad \text{ for }(t,x) \in (0,\infty)\times \mathbb{R}.
		\end{equation}
\end{proposition}
It is somewhat standard that optimal paths have constant velocity for homogeneous, convex control problems.  While our setting is novel, due to the Lagrangians differing on the road and the field, the arguments are standard enough that we relegate its proof to \Cref{ss.control_lemmas}.  The main novelty of the proof is due to the presence of the road.

From this characterization and standard ideas in Hamilton-Jacobi equations, we can easily derive some monotonicity properties of $J$.

\begin{lemma}\label{l.monotonicity}
In the setting of \Cref{p.straight_lines}, with $\tau_0$ and $p_0$ given by~\eqref{e.Lq3b}, we have
	\begin{enumerate}[(i)]

		\item\label{i.grad_J}
		The value function is rotationally nondecreasing from the positive $x$-axis to the positive $y$-axes: for all $x,y> 0$,
			\be
				(-y,x)\cdot \nabla J(1,x,y) \geq 0.
			\ee
		Moreover, the inequality above is strict if $\tau_0>0$.

		\item\label{i.radially_increasing}
			The value function is strictly radially increasing for $t>0$ and $(x,y)\neq 0$:
			\be
				(x,y) \cdot \nabla J(1,x,y) > 0.
			\ee
			(Although, we note that it is not rotationally symmetric).	

	\end{enumerate}
\end{lemma}
\begin{proof}
By \Cref{p.straight_lines}, 
$\nabla J(t,x,y)$ is well-defined in $(0,\infty)\times \H$ and 
we have, when $\tau_0 < 1$,
\be
	(-y,x)\cdot \nabla J(1,x,y)
		= (-y,x) \cdot \frac{1}{2}
\left(\tfrac{x-z_0}{1-\tau_0}, \tfrac{y}{1-\tau_0}  \right)
		= \frac{y z_0}{2(1-\tau_0)}
		\geq 0,
\ee
where we used $x,y>0$ and $z_0 \geq 0$.  This is clearly strict except in the cases $z_0 = 0$ (which corresponds to $\tau_0 = 0$).

The proof of the assertion~\ref{i.radially_increasing} is essentially the same, so we omit it.  This completes the proof.
\end{proof}

Next, we define the quantity that corresponds to directional spreading speed. 
For each $\vt \in [-\sfrac\pi2, \sfrac\pi2]$, define
\be
c_*(\vt) = \sup\{ r \geq 0:~ J(1, r\sin\vt, r\cos\vt) \leq 0\}.
\ee
Notice that this is well-defined thanks to $J(1,0,0) = L_r(0) = -1$ and \Cref{l.monotonicity}.\ref{i.radially_increasing}.

\subsection{Invasion speed on the road}

Here, we discuss the speed $c_*(\sfrac\pi2)= c_*(-\sfrac\pi2)$ on the road.  In particular, we show its threshold behavior, that is, that $c_*(\sfrac\pi2) = 2$ for $D \leq 2$ and $c_*(\sfrac\pi2) > 2$ for $D > 2$, and we obtain rough asymptotic behavior as $D\to\infty$.

We begin by characterizing the speed along the road in two useful ways.
\begin{proposition}\label{p.roadspeed}
The speed along the road (recall~\eqref{e.directional}) is the unique positive velocity such that
\be\label{e.roadspeed1}
	0 = \Lr(c_*(\sfrac\pi2)).
\ee
Furthermore,
\begin{equation}\label{e.roadspeed2}
	c_*(\sfrac\pi2)
		= \min_{q>0} \frac{\tH(q)}{q}.
\end{equation}
Additionally, $c_*(\sfrac\pi2)$ is nondecreasing in $D$.
\end{proposition}

\begin{proof}
Let us note that the strict convexity and coercivity of $\tH$ in \Cref{lem.tH} guarantees the existence and uniqueness of the minimum in~\eqref{e.roadspeed2}.  It also guarantees the strict convexity of $\Lr$, which is the basis for the existence and uniqueness of the solution to~\eqref{e.roadspeed1}.  We omit further details.

The identity~\eqref{e.roadspeed1} follows directly from \eqref{e.a.240719.1} of \Cref{p.straight_lines}.  

For~\eqref{e.roadspeed2}, we begin by noticing that
\be
	0 
		= \Lr(c_*(\sfrac\pi2))
		= \max_q\left( c_*(\sfrac\pi2) q - \tH(q)\right).
\ee
Let $q_*$ be the maximizer above.  This implies that
\be
	c_*(\sfrac\pi2) = \frac{\tH(q_*)}{q_*}
		\qquad\text{ and }\qquad
	c_*(\sfrac\pi2) \leq \frac{\tH(q)}{q}
		\quad\text{ for all } q>0,
\ee
which is precisely~\eqref{e.roadspeed2}.

Finally, the fact that $c_*(\sfrac\pi2)$ is nondecreasing in $D$ follows directly from \eqref{e.roadspeed2} and the fact that $\tH$ is nondecreasing in $D$.  This concludes the proof.
\end{proof}

With this characterization in hand, it is quite easy to derive various properties about $c_*(\sfrac\pi2)$.  We begin by showing that speedup occurs only when $D>2$. 

\begin{corollary}\label{c.D_near_2}
We have the following:
\begin{enumerate}[(i)]

	\item\label{i.speedup}
		The speed $c_*(\sfrac\pi2) = 2$ when $D\leq 2$ and 
		\be
			c_*(\sfrac\pi2) >2 \qquad\text{ when } D>2.
		\ee

	\item\label{i.speed_bounds1}
		For any $\tau\in [0,1]$, $c\in [-c_*(\sfrac\pi2), c_*(\sfrac\pi2)]$, and $v \in \H$ such that $|v| \leq 2$, we have
		\be
			J(1,  c\tau + v_1 (1-\tau), v_2(1-\tau))
				\leq 0.
		\ee
		Roughly, this says that the endpoint of any path that moves at speed $|c| \leq c_*(\sfrac\pi2)$ on the road and speed at most $2$ in the field is inside the Wulff shape $\cW$.

	\item\label{i.speed_bounds2}
		Letting $\beta=(2 + \mu)^{-1}$, we have
		\be\label{e.a060702}
			2\sqrt D \beta
				\leq c_*(\sfrac\pi2)
				\leq \sfrac{D}{\sqrt{D-1}}.
		\ee
	
\end{enumerate}
\end{corollary}
Before proving \Cref{c.D_near_2}, we observe that the bound in~\ref{i.speed_bounds2} is sharp in the case $\mu \ll 1 \ll D$.  In general, though, it is not sharp.  We investigate this further below; however, it is interesting to easily deduce the result that $c_*(\sfrac\pi2) \sim \sqrt D$ as $D\to\infty$.
\begin{proof}[Proof of \ref{i.speedup}]
Let us begin with the case $D\leq 2$.  The lower bound in~\ref{i.speedup} when $D \leq 2$ follows from \Cref{p.roadspeed} and~\eqref{e.tH}:
\be
	c_*(\sfrac\pi2)
		= \min_{q>0} \frac{\tH(q)}{q}
		\geq \min_{q>0} \frac{q^2 + 1}{q}
		= 2.
\ee
For the upper bound, notice that $1 \in [0,\sfrac1{\sqrt{D-1}}]$, so $H_r(1) =  q^2 + 1\big|_{q=1}=2$, and
\be
	c^* \leq \min_{q>0} \frac{\tH(q)}{q}
		\leq \frac{\tH(1)}{1}
		= 2.
\ee

We now consider the case $D>2$. For the lower bound, fix an arbitrary
\be
	\bar q \in \left( \frac{1}{\sqrt{D-1}}, 1\right).
\ee
For $q \leq \bar q$,
\be\label{e.c052301}
	\frac{\tH(q)}{q}
		\geq q + \frac{1}{q}
		\geq \bar q + \frac{1}{\bar q}
		>2.
\ee
Here we used that $x + \sfrac1x$ is decreasing for $x< 1$ and that $\bar q < 1$.   For $q > \bar q$, we have
\be\label{e.c052302}
	\frac{\tH(q)}{q}
		= q + \frac{p_q^2 + 1}{q}
		\geq  q + \frac{p_{\bar q}^2 + 1}{q}
		\geq 2 \sqrt{ p_{\bar q}^2 + 1} 
		>2.
\ee
Here we used that $p_q$ is increasing in $q$ and that $\bar q > \sfrac{1}{\sqrt{D-1}}$ so that $p_{\bar q} > 0$.  The combination of~\eqref{e.c052301} and ~\eqref{e.c052302} finish the proof of \ref{i.speedup}.
\end{proof}

\begin{proof}[Proof of \ref{i.speed_bounds1}]
We use~\eqref{e.Lq3a} from \Cref{p.straight_lines} with the choice $x = c \tau + v_1(1-\tau)$, $y = v_1(1-\tau)$, and $z = c \tau$.  Then
\be
	\begin{split}
		J(1,x,y)
			&\leq (1-\tau) \Lf\left(\frac{v_1 (1-\tau)}{1-\tau}, \frac{v_2 (1-\tau)}{1-\tau}\right)
				+ \tau \Lr\left(\frac{\tau c}{\tau}\right)
			\\&
			= (1-\tau) \left( \frac{v_1^2 + v_2^2}{4} - 1\right)
				+ \tau \Lr\left(c\right)
			\\&
			\leq (1-\tau) \cdot 0  + \tau \Lr(c_*(\sfrac\pi2))
			= 0.
	\end{split}
\ee
In the last inequality, we used that $\Lr$ is increasing for in $|c|$, and $r, c_*(\sfrac\pi2) \geq 0$.
\end{proof}

\begin{proof}[Proof of \ref{i.speed_bounds2}]
The upper bound in~\ref{i.speed_bounds2} follows from the fact that, for $q>\sfrac{1}{\sqrt{D-1}}$, we have
\be
	c_*(\sfrac\pi2)
		\leq \frac{\tH(q)}{q}
		= \frac{ D q^2 - \frac{\mu p_q}{\kappa \nu + p_q}}{q}
		\leq D q.
\ee
Taking the limit $q\to \sfrac{1}{\sqrt{D-1}}$, we obtain the bound.

We now consider the lower bound.  Observe from the definition of $\tH$ that 
\be
	\tH(q) = q^2 + p_q^2 + 1
		= D q^2 - \frac{\mu p_q}{\kappa \nu + p_q}.
\ee
Hence, we may write, for $q\geq \sfrac{1}{\sqrt{D-1}}$,
\be
	\begin{split}
		\tH(q)
			= (1-\beta) (q^2 + p_q^2 + 1) + \beta\left( D q^2 - \frac{\mu p_q}{\kappa \nu + p_q}\right)
			\geq (1-\beta) + \beta\left( D q^2 - \mu\right)
			= \beta \left( D q^2 + 1\right),
	\end{split}
\ee
where we used that $\beta = \sfrac{1}{(2+\mu)}$.
We deduce that
\be
	\min_{q \geq \frac{1}{\sqrt{D-1}}} \frac{\tH(q)}{q}
		\geq 2\beta \sqrt D.
\ee
It is simple to check that
\be
	\min_{q \leq \frac{1}{\sqrt{D-1}}} \frac{\tH(q)}{q}
		= \frac{D}{\sqrt{D-1}}
		\geq 2\beta \sqrt D
\ee
because $2\beta < 1$. This completes the proof.
\end{proof}

These identities make various asymptotic computations quite easy.  In particular, let us consider the case $D\to \infty$.  For any $\theta>0$, define $\zeta_\theta$ to be zero if $\theta \leq 1$ and the unique positive solution of
	\be\label{e.zeta}
		\theta^2 = \zeta^2 + 1 + \frac{\mu \zeta}{\kappa \nu + \zeta}
	\ee
	 if $\theta >1$.  Then we recover \cite[Theorem 1.1 and eqn~(7.8)]{BRR_JMB}.
\begin{corollary}\label{c.large_D}
	With $\zeta_\theta$ defined above, we have
	 \be
	 	\lim_{D\to\infty} \frac{c_*(\sfrac\pi2)}{\sqrt D}
			= \min_{\theta>0} \frac{1+ \zeta_\theta^2}{\theta}
			> 0.
	 \ee
\end{corollary}
\begin{proof}
	This follows by letting $q_\theta = \sfrac{\theta}{\sqrt{D-1}}$ and noticing that $\zeta_\theta = p_{q_\theta}$.  Then
	\be
		\frac{c_*(\sfrac\pi2)}{\sqrt {D-1}}
			= \min_{q>0} \frac{\tH(q)}{q\sqrt {D-1}}
			= \min_{\theta>0} \frac{\tH(q_\theta)}{\theta}
			= \min_{\theta>0} \frac{1+ \zeta_\theta^2 + \frac{\theta^2}{D-1}}{\theta}.
	\ee
	Let us note that $\zeta_\theta \approx \theta$ for $\theta \gg1$, so the minimum above is well-defined.
\end{proof}

Finding $\zeta_\theta$ via~\eqref{e.zeta} amounts to solving a cubic equation.  While there are closed form solutions of these, they are somewhat complicated.  It is best to consider \Cref{c.large_D} numerically or in certain asymptotic regimes.  Using \Cref{c.large_D}, this is often quite simple.

We present briefly an example where the asymptotic speed can be computed.  When $\mu \to \infty$, we see that
\be
	\zeta_\theta \sim \frac{\theta^2 - 1}{\mu} \ll 1.
\ee
Hence,
\be
	c_*(\sfrac\pi2)
		= \min_{\theta>0}\frac{1+ \zeta_\theta^2}{\theta}
		\sim \min_{\theta>0}\frac{1 + \frac{(\theta^2-1)^2}{\mu}}{\theta}
		\sim \frac{4}{3^{\sfrac34} \mu^{\sfrac14}}.
\ee
The last equality follows from simple calculus computations after noting that the minimum occurs at $\theta_{\min} = O(\mu^{\sfrac{1}{4}})$.  The work above yields the following.
\begin{corollary}
	The asymptotic limit of the speed along the road is
	\be
		\lim_{\mu\to\infty} \lim_{D\to\infty} \frac{\mu^{\sfrac14} c_*(\sfrac\pi2)}{\sqrt D}
			= \frac{4}{ 3^{\sfrac34} }.
	\ee
\end{corollary}

\subsection{Convexity of the Wulff shape}

We now begin our investigation into how the presence of the road influences the invasion behavior in the field.  First, we examine the strict convexity of the Wulff shape.  In~\cite{Berestycki2016shape}, the authors were interested in comparing $\cW$ with $\underline \cW$, the shape containing all paths moving at speed $c_*(\sfrac\pi2)$ on the road (for time $\tau$) and speed $2$ in the field (for time $1-\tau$):
\be
	\underline \cW = \{(x,y) : \exists \tau \in [0,1], v\in \H \text{ s.t. } |v| = 2,  |x| \leq v_1(1-\tau) + c_*(\sfrac\pi2)\tau, y \leq v_2(1-\tau)\}.
\ee
See \Cref{fig:H2}.
\begin{figure}[h]
   \centering
    \includegraphics[width=.75\textwidth,height=0.2625\textwidth]{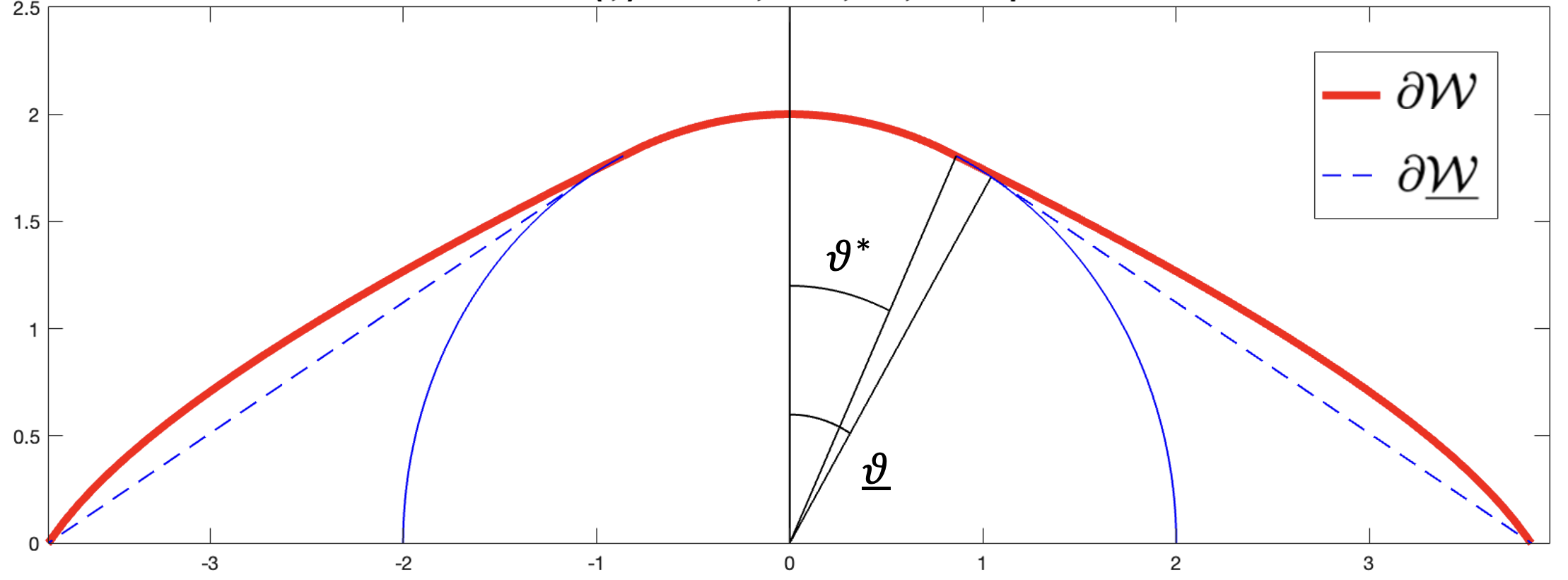}
	\caption{The boundary of $\mathcal{W}$ and $\underline{\mathcal{W}}$ in the case $D = 9$, and $d=\mu=\nu  =k=1$. In this case, the road speed is given by $c_*(\sfrac\pi2)\approx 3.1243$.  The angle $\vt_*$ is such that, for $|\vt| < \vt_*$ individuals (following optimal paths) do not use the road to travel and, for $|\vt|> \vt_*$, individuals use the road to travel for at least some portion of their journey; see \Cref{prop:1.11}. The constant $\underline \vt$ is defined in~\eqref{e.underline_vt}.}
 \label{fig:H2}
\end{figure}
Notice that $\underline{\mathcal{W}}$ is the convex hull of the union of the ball in $\H$ of radius $2$ centered at the origin and  the segment of the $x$-axis between $x= \pm c_*(\sfrac\pi2)$.  It is, thus, {\em not strictly convex}.  In fact, using some simple trigonometry, we find
\be
	\underline{\mathcal W}
			= \left\{r(\sin\vt,\cos\vt):|\vt| \leq \frac\pi2,~0 \leq r \leq 
				\begin{cases}
						2
							&\text{ for } |\vt| \leq \underline \vt\\
						2\sec (|\vt| - \underline\vt)
							&\text{ for } |\vt| \geq \underline \vt
					\end{cases}
				\right\},
\ee
where
\be\label{e.underline_vt}
	\underline\vt = \sin^{-1}\left(\frac{2}{c_*(\sfrac\pi2)}\right).
\ee
Recalling \Cref{c.D_near_2}.\ref{i.speed_bounds1}, we see that $\underline \cW \subset \cW$.
The following lemma shows that $\cW \neq \underline \cW$ when $D>2$ because the former is strictly convex and the latter is not. Hence, we recover that $\underline \cW \subsetneq\cW$, which was observed in \cite{Berestycki2016shape}. Further comparisons between $\underline \cW$ and $\cW$ are made in \Cref{p.c_angle>2,prop:1.18}.

\begin{lemma}\label{l.W_convex}
The set \,$\mathcal{W}$ is strictly convex.  As a consequence, $\underline \cW \subsetneq \cW$ when $D>2$.
\end{lemma}
\begin{proof}
It is enough to show that, if
\be
	J(1,x_i,y_i) = 0
		\qquad\text{ for } i =1,2,
\ee
then
\be
	J(1,\bar x,\bar y)<0
		\qquad\text{ where } \quad
	(\bar x,\bar y)
		=\left( \frac{x_1+x_2}{2}, \frac{y_1 + y_2}{2}\right).
\ee
Thanks to \Cref{p.straight_lines}, there exists $v_i, q_i, p_i>0$ and $\tau_i \in (0,1)$ such that 
the minimizing path $\gamma_i$ of $J(1,x_i,y_i)$ is given by
\be
\gamma_i(s) = \begin{cases}
s(v_i,0)                                            &\text{ for } s \in [0,\tau_i],\\
(s-\tau_i)(2q_i,2p_i) + \tau_i(v_i,0)     &\text{ for } s \in (\tau_i,1].
\end{cases}
\ee
i.e.
\be
0 = J(1,x_i,y_i) = \int_0^1 \hat{L}\gamma_i(s),\dot\gamma_i(s))\,ds = (1-\tau_i)L(2q_i,2p_i) + \tau_i \Lr(v_i).
\ee
Define a path to $(\bar x, \bar y)$ as follows: let
\be
	\begin{split}
		&\bar \tau = \frac{\tau_1 + \tau_2}{2}
		\qquad
		\bar v= \frac{\tau_1v_1 + \tau_2 v_2}{\tau_1 + \tau_2},
		\\&
		\bar q =\frac{(1-\tau_1)q_1+(1-\tau_2)q_2}{(1-\tau_1)+(1-\tau_2)},
		\quad\text{ and }\quad
		\bar p =\frac{(1-\tau_1)p_1+(1-\tau_2)p_2}{(1-\tau_1)+(1-\tau_2)},  				
	\end{split}
\ee
and, for all $s \in [0,1]$, let
\be\label{e.c052401}
	\gamma(s)
		= \begin{cases}
			s(\bar v,0)
				\qquad &\text{ for }s\in[0,\bar\tau],
			\\
			\tau(\bar v,0) + (s-\bar\tau)(2\bar q,2\bar p)
				\qquad &\text{ for }s\in[\bar\tau,1].
		\end{cases}
\ee
Clearly, $\gamma(1) = (\bar x, \bar y)$.  Computing $J$ directly, we find
\be
	\begin{split}
		J(1,\bar x, \bar y)
			&\leq \int_{\bar \tau}^1 \Lf(2 \bar q, 2\bar p)\,ds + \int_0^{\bar \tau} \Lr(\bar v)\,ds
			\\&= (1-\bar \tau) \Lf(2\bar q, 2\bar p) + \bar \tau \Lr(\bar v)
			\\&
			< (1-\bar \tau) \left[\frac{1-\tau_1}{2(1-\bar \tau) } \Lf(2q_1,2p_1)
				+ \frac{1-\tau_2}{2(1-\bar \tau) } \Lf(2q_2,2p_2) \right]
			\\&\qquad
				+ \bar \tau \left[ \frac{\tau_1}{2\bar\tau} \Lr(v_1)
				+ \frac{\tau_2}{2\bar \tau}\Lr(v_2) \right] = 0.
	\end{split}
\ee
where $\Lf(v_1,v_2) = \frac{|v|^2}{4}-1$, and we used the strict convexity of $\Lf$ and $\Lr$.
\end{proof}

\subsection{The spreading speed in different directions}\label{ss.directions}

We now understand how the shape of $\cW$ depends on the angle $\vt$ from the $y$-axis.  Our main goal is to show that there is a critical angle $\vt_*>0$ such that $c_*(\vt)$, defined in \Cref{c.angular_speed}, equals $2$ (that is, it is unaffected by the presence of the road) if and only if $|\vt| \leq \vt_*$.  Otherwise $c_*(\vt) > 2$; that is, the presence of the road speeds up the invasion.  We state the main result here:
\begin{proposition}\label{prop:1.11}
There exists $\vt_* \in (0,\sfrac\pi2]$ such that
\be\label{e.vt.bd1}
	c_*(\vt) = 2 \quad \text{ for }|\vt|\leq \vt_*
	\quad \text{ and }\quad
	c_*(\vt) >2 \quad \text{ for }\vt_*<|\vt|\leq \sfrac\pi2.
\ee
\begin{enumerate}[(i)]
	\item\label{i.D_leq_2}
	If $D \leq 2$, then $\vt_*  = \sfrac\pi2$.  In other words, $c_*(\vt) = 2$ for all $\vt$ and
		\be
			\mathcal{W}
				= \{(r\sin \vt, r \cos\vt):
					~ 0 \leq r \leq 2t,~ |\vt| \leq \sfrac\pi2\}
                = \overline{B_2(0) \cap \H}.
		\ee
	
	\item\label{i.D_>_2}
	If $D > 2$, then
    \be\label{e.c071905}
        \vt_* \in
			\left[
				\arcsin\Big(\frac{7}{16\sqrt{D-1}}\Big),
				\min\left\{
					\frac{\pi}{2},
					\arcsin\left(\frac{2+\mu}{\sqrt{D}}\right)
					\right\}
				\right).
    \ee
    Moreover,
  \be\label{e.a240722.13}
        \frac{d}{d\vt} c_*(\vt) > 0
        \qquad\text{ for } \vt_* < \vt \leq \sfrac\pi2.
    \ee
\end{enumerate}
\end{proposition}

Let us note that the result above was first proved in \cite{Berestycki2016shape} via fine construction of super/subsolutions for the cooperative system \eqref{e.unscaled_berestycki}.  Here, we give an alternative derivation of the result entirely via the control formulation~\eqref{e.Jcon} and \Cref{t.w=J+}.  We begin with a technical lemma whose proof we postpone until \Cref{s.technical}.
\begin{lemma}\label{lem:3.4}
Fix $(x,y) \in \H$ and let $\tau_0, p_0,q_0,z_0$ be given by~\eqref{e.Lq3b} for the point $(1,x,y)$.  Suppose that $\tau_0>0$.
	\begin{enumerate}[(i)]
    
		\item\label{i.tau_0>0a}
		We have
		\be
			q_0 = \Lr'\left( \frac{z_0}{\tau_0}\right),
            \quad
            \frac{z_0}{\tau_0} = \tH'(q_0),
            \quad\text{ and }\quad
            p_0 = p_{q_0}.
		\ee
        Moreover, $q_0 > \tfrac{1}{\sqrt{D-1}}$ if $y>0$.
		\item\label{i.tau_0>0b}
		The action takes the value
		\be
			J(s,\gamma(s))
				= (s-\tau_0)(\tH(q_0) -2) + \tau_0 \Lr\left( \tH'(q_0)\right) \quad \text{ for }s \in [\tau_0,t].
		\ee

	\end{enumerate}
\end{lemma}

Next, we state and prove two short propositions that let us distinguish between $\underline \cW$ and $\cW$ more precisely.  The first of these, showing that $\vt_*<\underline \vt$ (see \Cref{fig:H2} for an illustration), is used in the proof of \Cref{prop:1.11}.
\begin{proposition}\label{p.c_angle>2}
We have that $\vt_* < \underline \vt$.  Moreover:
\be\label{e.c072202}
	c_*(\vt)
		>\begin{cases}
			2 &\text{ for }|\vt| \in (\vt_*,\underline\vt),\\
			\frac{2}{\cos(\vt - \underline\vt)} &\text{ for }|\vt| \in [\underline \vt, \sfrac\pi2).
		\end{cases}
\ee
\end{proposition}
\begin{proof}
For the first claim, we need only show that $c*(\underline \vt) > 2$ by the strict convexity of $\cW$. To do so, consider the tangent line to $\underline\cW$ at $2(\sin\underline\vt, \cos\underline\vt)$. Suppose for contradiction that $c_*(\underline \vt) = 2$, then this line is also tangent to $\cW$ at $2(\sin\vt, \cos\vt)$ and it must not intersect $\cW$ at any other points. This contradicts the fact that this tangent line of $\underline\cW$  contains both the point $2(\sin\vt, \cos\vt)$ and $(c^*(\sfrac\pi2),0)$ (see \Cref{fig:H2}). 
This completes the proof that $c_*(\vt) > 2$ for $\vt \in (\vt_*,\underline \vt)$.

The second claim follows by an elementary geometry argument and is, thus, omitted.
\end{proof}

The second provides a complementary bound to~\eqref{e.c072202}.

\begin{proposition}\label{prop:1.18}
    For each $|\vt| \in [\vt_*, \sfrac\pi2]$, 
    \be\label{e.c053001}
	     c_*(\vt) \leq  \frac{2}{\cos(\vt - \vt_*)}
	     	\quad \text{ and }\quad
	     \lim_{D \to \infty} c_*(\vt) = \frac{2}{\cos \vt}.
    \ee
\end{proposition}
\begin{proof}
	The first inequality follows by trigonometry and using the fact that $\cW$ lies below the tangent line to $\partial \cW$ at $2(\sin \vt_*, \cos \vt_*)$; that is,
	\be
		\cW \subset \{(x,y) \in \overline \H : (x,y) \cdot (\sin \vt_*, \cos \vt_*) \leq 2\}.
	\ee
	To see the second inequality, we use that $0 \leq \vt_* \leq \underline \vt$ and, by~\eqref{e.underline_vt} and \Cref{c.large_D}, $\underline \vt(D) \to 0$ as $D\to\infty$.  It is then a direct consequence of the first inequality in~\eqref{e.c053001} and~\Cref{p.c_angle>2}.
\end{proof}

We are now in a position to prove \Cref{prop:1.11}.
\begin{proof}[Proof of \Cref{prop:1.11}.\ref{i.D_leq_2}]
Suppose $D\leq 2$.  Then the claim is finished if we show that
\begin{equation}\label{e.ss.61}
	J(1,x,y) \leq \frac{x^2+y^2}{4}-1
		\quad \text{ for }x^2 + y^2 \leq 4
	\quad \text{ and }\quad
	J(1,x,y) >0
		\quad \text{ for }x^2 + y^2 > 4.
\end{equation}
We assume without loss of generality that $y>0$.  Indeed, if $y=0$, it follows by the obvious continuity of $J$ that $J(1,x,0) \leq \sfrac{x^2}{4} - 1$ for $x\leq 2$ and $J(1,x,0) \geq 0$ for $x\geq 2$.  Using then \Cref{l.monotonicity}.\ref{i.radially_increasing}, we deduce that $J(1,x,0) >0$ for $x > 2$.  We assume without loss of generality that $x\geq 0$ because $J$ is clearly even in $x$.

We first consider the upper bound in~\eqref{e.ss.61}.  If we take the path $\gamma(s) = s(x,y),$ then
\be
	J(1,x,y)
		\leq \frac{x^2 + y^2}{4} -1.
\ee
Here we used that $y>0$ so that only the $\Lf$ appears in the computation of $J$.

We now consider a lower bound when $x^2 + y^2 > 4$.  If $\tau_0 = 0$, then we immediately deduce that $\gamma(s) = (\sfrac{s}{t})(x,y)$ from \Cref{p.straight_lines}, and it follows that
\be
	J(1,x,y)
		= \frac{x^2+y^2}{4} - 1
		> 0.
\ee
Let us, then, focus on the case
\be
	\tau_0 >0.
\ee

Before we begin, let us recall that, by~\eqref{e.tH}-\eqref{e.pq},
\begin{equation}\label{e.ss.62}
    \tH(q)
    	= q^2 + 1 \quad \text{ for } q \in \left[0,\frac{1}{\sqrt{D-1}}\right]
\quad \text{ and }\quad 
	\Lr(v)
		= \frac{v^2}{4}-1  \quad \text{ for }v \in \left[0,\frac{2}{\sqrt{D-1}}\right].
\end{equation}
Moreover, both $\tH$ and $\Lr$ are strictly convex (see~\Cref{lem.tH}) and increasing on $(0,\infty)$.

In view of \Cref{lem:3.4}.\ref{i.tau_0>0a}, the fact that $\tau_0>0$ implies that $p_0 = p_{q_0}$.  On the other hand, $p_0 > 0$ because $y>0$ (cf.~\eqref{e.Lq3b}, recalling that $\gamma(1) = y$).  It follows from~\eqref{e.pq} that $q_0 > \sfrac{1}{\sqrt{D-1}}$.  Then using 
the strict convexity of $\tH$ (\Cref{lem.tH}), we have
\be
	\tH'(q_0)
		\geq  2q_0
		> \frac{2}{\sqrt{D-1}}.
\ee
Now using the monotonicity of $\Lr$, we see that
\be
	\Lr(\tH'(q_0))
		> \Lr(\sfrac{2}{\sqrt{D-1}})
		= \frac{1}{D-1} - 1
		\geq 0,
\ee
where we used \eqref{e.ss.62}.
Hence, using \Cref{lem:3.4}.\ref{i.tau_0>0b} and also that $\tH$ is increasing,
\be
	\begin{split}
		J(1,x,y)
			&= (1-\tau_0)\left( \tH(q_0) - 2\right)
				+ \tau_0 \left( \Lr(\tH'(q_0))\right)
			> (1-\tau_0)\left( \tH(\sfrac{1}{\sqrt{D-1}}) - 2\right) + 0
			\\&
			= (1-\tau_0)\left( \frac{1}{D-1} - 1\right)
			\geq 0.
	\end{split}
\ee
This completes the proof.
\end{proof}

\begin{proof}[Proof of \Cref{prop:1.11}.\ref{i.D_>_2}]
First, by \Cref{p.straight_lines}, $\nabla J$ is continuous in $\H$.  

Next, without loss of generality, let $\vt \in [0,\sfrac{\pi}2)$ for the remainder of the proof. Observe that, in view of \Cref{l.monotonicity}.\ref{i.radially_increasing},
$c_*(\vartheta)$ is uniquely given by
\begin{equation}\label{e.a060401}
  J(1,c_*(\vartheta) \sin\vartheta, c_*(\vartheta) \cos \vartheta) = 0.  
\end{equation}
Further, by the continuity and nondegeneracy of $\nabla J$, we deduce from the implicit function theorem that $c_*(\vartheta)$ is differentiable in $\vartheta$.

Differentiating \eqref{e.a060401} with respect to $\vartheta$, we obtain, at the point $(x,y) = c_*(\vartheta)(\sin\vartheta, \cos\vartheta)$, 
\be
\frac{c'_*(\vartheta)}{c_*(\vartheta)}(x,y) \cdot \nabla J(1,x,y) = (-y,x) \nabla J(1,x,y),
\ee
and, hence, by \Cref{lem:3.4} again,
\be\label{e.c060501}
    c'_*(\vartheta)
        = c_*(\vartheta) \frac{(-y,x) \cdot\nabla J(1,x,y)}{(x,y) \cdot \nabla J(1,x,y)}
        \geq 0.
\ee
We note that the last inequality is strict whenever $\tau_0 > 0$ by \Cref{l.monotonicity}.\ref{i.grad_J}.

Given~\eqref{e.c060501}, we see that $c_*$ is increasing in $\vt$.  Hence, we define
\be
    \vt_* = \inf\{\vt \in [0,\sfrac\pi2] : c_*(\vt) >  2\}.
\ee
Since $D>2$, 
\Cref{c.D_near_2}.\ref{i.speedup} says that $c_*(\sfrac{\pi}{2}) >2$.  We deduce that $\vt_*$ is well-defined and satisfies $\vt_*<\sfrac{\pi}{2}$.

To deduce \eqref{e.a240722.13}, it remains to prove that  $\vt \in (\vt_*,\sfrac\pi2]$ implies $\tau_0 >0$. Indeed, suppose to the contrary that  for some $\vt \in (\vt_*,\sfrac\pi2]$ we have $\tau_0 = 0$, then $c_*(\vt) >2$ (since $c_*(\cdot)$ is monotone increasing). 
Moreover, for $(x,y) = c_*(\vt)(\sin\vt, \cos\vt$, we have $\gamma(s) = s(x,y)$. Hence,
\be
0=J(1,x,y) = \int_0^1 L_f(\dot\gamma(s))\,ds = \tfrac{1}{4}(x^2 + y^2) -1 = \tfrac{c_*(\vt)^2}{4}-1.
\ee
This is a contradiction. This proves \eqref{e.a240722.13}.

The improved upper bound of $\vt_*$ follows from the fact that, recalling \Cref{l.W_convex} (see, also, \Cref{fig:H2}),
\be\label{e.c072201}
	0 \leq \vt_*
		< \underline\vt
		= \arcsin \left(\sfrac2{c_*(\sfrac{\pi}2})\right)
\ee
and the lower bound~\eqref{e.a060702} of $c_*(\sfrac{\pi}2)$.

For the lower bound of $\vt_*$, let $(x,y) \in \partial\mathcal{W}$ be given such that
$$
(x,y) = c_*(\vt)(\sin \vt, \cos \vt) \quad \text{ for some }\vt \in \{\vt \in [0,\sfrac\pi2] : c_*(\vt) >  2\}.
$$
It suffices to show that 
\be\label{e.07120.1}
x \geq \frac{7}{8\sqrt{D-1}},
\ee
Indeed, were~\eqref{e.07120.1} established, then taking the infimum of all $\vt$ such that $c_*(\vt) >2$ yields 
\be
	2\sin \vt_*
		=c_*(\vt_*)\sin \vt_*
		\geq \frac{7}{8\sqrt{D-1}},
\ee
which implies \eqref{e.c071905}. 

Let us now prove~\eqref{e.07120.1}.  To aid us in this proof, associate to the point $(1,x,y)$ the quantities $\tau_0,p_0,q_0,z_0$  given as in \Cref{p.straight_lines}. By the assumptions on $(x,y)$, we have $\vt \in (\vt_*,\sfrac\pi2]$ and hence $\tau_0>0$, as we have argued above.

Next, by \Cref{prop:1.11}.\ref{i.tau_0>0a}, we deduce that $q_0 > \sfrac{1}{\sqrt{D-1}}$ and $\sfrac{z_0}{\tau_0} = H'_r(z_0)$. Since $\Hr$ is strictly convex, it follows that
\be
    \frac{z_0}{\tau_0}
        = H'_r(q_0)
        > H'_r\left(\tfrac{1}{\sqrt{D-1}}\right)
        = \tfrac{2}{\sqrt{D-1}},
\ee
where we used \eqref{e.ss.62}. Hence, by \eqref{e.Lq3b},
\be
    x
        = \tau_0 z_0 + (1-\tau_0)q_0
        \geq \tau_0^2 \cdot \tfrac{2}{\sqrt{D-1}}
                + (1-\tau_0)\cdot \tfrac{1}{\sqrt{D-1}}
        \geq \tfrac{7}{8\sqrt{D-1}}.
\ee
This is precisely~\eqref{e.07120.1}.  The proof is complete.
\end{proof}

\subsection{Further characterization of paths}

We now characterize the optimal paths a bit more precisely.  As a result, we see precisely why the strict inclusion $\underline \cW \subsetneq \cW$ holds.  Indeed, we see that, if $(x,y)$ is on the road-enhanced portion of the boundary in the field, it corresponds to a path that moves at speed $\crr > c_*(\sfrac\pi2)$ for time $\tau\in(0,1)$ and then moves at speed $\cf < 2$ in the field for remainder of the $(1-\tau)$ time.

\begin{proposition}\label{p.no_Huygens}
Fix $(x,y)$ on the front: $J(1,x,y) = 0$, and let $\gamma \in N(1,x,y)$ be the optimal path.  Then there are two alternatives:
\begin{enumerate}[(i)]
	\item (Non-road-enhanced or purely on road) If $x^2 + y^2 = 4$ or if $y = 0$, then $\gamma(s) = s(x,y)$;

	\item (Road-enhanced in the field) If $x^2 + y^2 > 4$ and $y>0$, then there are $\crr > c_*(\sfrac\pi2)$ and $\cf \in \H$ with $|\cf| < 2$ such that
		\be
			\gamma(s)
				= \begin{cases}
						s (\crr,0)
							\qquad & \text{ if } s \leq \tau_0\\
						\tau_0(\crr,0) + (s-\tau_0) \cf
							\qquad & \text{ if } s \geq \tau_0
					\end{cases}
		\ee
		for $\tau_0 \in (0,1)$ given by \Cref{p.straight_lines}.

\end{enumerate}
\end{proposition}

Before we prove this, we state two lemmas.  First, we state a lemma whose proof we delay until \Cref{ss.control_lemmas}. 
\begin{lemma}
\label{prop:D.5}
Fix $v_0$ and $q_0$ such that $q_0$ is in the subdifferential of $\Lr$ at $v_0$.  Then, if $\Lr(v_0) < 0$, then $\tH(q_0) <  2$.  Moreover, if $D>2$ and $\Lr(v_0) \leq 0$, then $\tH(q_0) < 2$.
\end{lemma}
Next, let us remind the reader of the dynamic programming principle.
 As this result is standard, we omit its proof.
\begin{lemma}[Dynamic programming principle]
    For $(t,x,y) \in (0,\infty)\times \overline \H$ and $\gamma \in N(t,x,y)$, suppose that
    \begin{equation}\label{e.bell.0}
	    J(t,x,y) = \int_0^t \hat L(\gamma(s),\dot\gamma(s))\,ds.
    \end{equation}
    with $\gamma(0) = 0$ and $\gamma(t) = (x,y)$. 
    Then for each $\tau \in (0,t)$, 
    \begin{equation}\label{e.bell.1}
        J(t,x,y) = \int_{\tau}^t \hat L(\gamma(s),\dot\gamma(s))\,ds + J(\tau,\gamma(\tau))
    \end{equation}
    and
    \begin{equation}\label{e.bell.2}
        J(\tau,\gamma(\tau)) = \int_0^\tau  \hat L(\gamma(s),\dot\gamma(s))\,ds.
    \end{equation}
\label{lem:bellman}
\end{lemma}

We now prove \Cref{p.no_Huygens}.
\begin{proof}[Proof of \Cref{p.no_Huygens}]
Using \Cref{p.straight_lines}, let us consider the cases $\tau_0 = 1$, $\tau_0 = 0$, and $\tau_0 \in (0,1)$. 
First consider when $\tau_0 = 1$.  This is equivalent to $y=0$ (under the standing assumption that $J(1,x,y) = 0$).  Then the conclusion follows directly from \Cref{p.straight_lines}.  

Next consider when $\tau_0 = 0$.  By \Cref{p.straight_lines}, this is equivalent to $\gamma(s) = s(x,y)$.  Additionally,
\be
	0
		= J(1,x,y)
		= \Lf(x,y),
\ee
so that we see that $x^2 + y^2 = 4$.  By the uniqueness of minimizing paths (\Cref{p.straight_lines}), it follows that $\tau_0 = 0$ is equivalent to $x^2 + y^2 = 4$ (under the standing assumption that $J(1,x,y) = 0$), which completes the proof of (i).

Finally, we consider the case $\tau_0 \in (0,1)$, which is the most difficult.  As we have seen above, this is equivalent to the assumption that $y>0$ and $x^2 + y^2 > 4$.  Additionally, by \Cref{c.D_near_2}, we must have $D>2$ in this case.

\Cref{lem:3.4}.\ref{i.tau_0>0b} yields
\begin{equation}\label{e.ss.51}
	J(s,\gamma(s))
		= (s-\tau_0)(\tH(q_0) -2) + \tau_0 \Lr\left(\frac{z_0}{\tau_0}\right) \quad \text{ for }s \in [\tau_0,t],
\end{equation}
where $q_0 = \Lr'\left( \frac{z_0}{\tau_0}\right)$.  Keeping this and \Cref{p.straight_lines} in mind, the optimal path $\gamma$ satisfies
\be
\dot\gamma(s) = \begin{cases}
    (c_{\rm r},0):= (\tfrac{z_0}{\tau_0},0) &\text{ for }s\in [0,\tau_0),\\
    c_{\rm f}:=(2q_0, 2p_{q_0}) &\text{ for }s \in (\tau_0,1]. 
\end{cases}
\ee
Hence, by the definition of $J$ and
by~\eqref{e.ss.51} evaluated at $s=1$, we have
\be
	\tau_0 \Lr(\crr) + (1-\tau_0) \Lf(\cf)
		= J(1,x,y)
		= (1-\tau_0)(\tH(q_0) -2) + \tau_0 \Lr\left(\frac{z_0}{\tau_0}\right),
\ee
whence
\be\label{e.c053103}
	\Lf(\cf) = \tH(q_0) -2.
\ee
Next, we show the non-strict inequality
\be\label{e.c053101}
	\Lr(\crr)
		\geq 0
	\quad\text{ and }\quad 
	\Lf(\cf) \leq 0.
\ee
Notice that, by~\eqref{e.ss.51}, evaluated at $s=1$, we have
\be\label{e.c053102}
	0 = J(1,x,y)
		= J(1,\gamma(1))
		= (1-\tau_0)(\tH(q_0) -2) + \tau_0 \Lr\left(\frac{z_0}{\tau_0}\right)
\ee
In view of~\eqref{e.c053103}, this implies that the two terms in~\eqref{e.c053101} have opposite signs. Suppose that~\eqref{e.c053101} were not true, then we must have 
$\Lr(\sfrac{z_0}{\tau_0})<0$. 
\Cref{prop:D.5} and the fact that $q_0 = \Lr'(\crr)$ imply  $\tH(q_0) < 2$.  This is impossible in view of \eqref{e.c053102}.  It follows that~\eqref{e.c053101} holds.

If $D>2$, then \Cref{prop:D.5} implies that $\tH(q_0) < 2$, whence~\eqref{e.c053102} implies that $\Lr(\crr) > 0$.  We deduce that $\crr > c_*(\sfrac\pi2)$ (recall that $\Lr$ is increasing and, by~\eqref{p.roadspeed}, $\Lr(c_*(\sfrac\pi2))=0$).  A similar argument using~\eqref{e.c053103} shows that $|\cf| < 2$.
This completes the proof.
\end{proof}

\subsection{Conical domains}

We now use our understanding of $J$ and $c_*$ developed in the previous subsections, as well as \Cref{t.conical}, to deduce several results for the problem posed on conical domains.  Recall the notation set in~\eqref{ss.conical_domains}.  Throughout this subsection, we suppose that $(D,\kappa,\mu,\nu) = (\tilde D,\tilde \kappa, \tilde \mu, \tilde \nu)$.

We also introduce the notation $c_{*a}$ for the speed on $\Omega_a$, with the understanding that $c_{*(\sfrac\pi2)} = c_*$.  For each $\vt \in [\sfrac\pi2-2a,\sfrac\pi2]$, this is well-defined by the identity~\eqref{e.c060701} and the fact that $J$ is radially increasing (\Cref{l.monotonicity}.\ref{i.radially_increasing}).

Recall that $J$ is clockwise rotationally increasing (\Cref{l.monotonicity}.\ref{i.grad_J}). Thus, using~\eqref{e.c060701}, we immediately see that the value of $J_a(t,x,y)$ is given by $J(t,x,y)$ if $(x,y)$ is closer to $\Gamma_0$ and by $J(t,\Psi_a(x,y))$ otherwise.  This yields the following result.
\begin{proposition}\label{p.J_a_J}
For any $(x,y) \in \overline{\Omega_a}$, we have that
\be
	J_a(t,x,y)
        = \begin{cases}
            J(t,x,y)
                \qquad &\text{ if } \frac{x}{y} \geq \tan(\sfrac\pi2-a),\\
            J(t,\Psi_a(x,y))
                \qquad &\text{ if } \frac{x}{y} \leq \tan(\sfrac\pi2-a).            
        \end{cases}
\ee
As a consequence, for all $\vt \in [\sfrac\pi2-2a,\sfrac\pi2]$,
    \be
        c_{*a}(\vt)
            = \begin{cases}
                c_*(\vt)
                    \qquad &\text{ if } \vt \in[\sfrac\pi2-a,\sfrac\pi2],
                \\
                c_*(\pi-2a-\vt)
                    \qquad &\text{ if } \vt \in [\sfrac\pi2-2a,\sfrac\pi2-a].
            \end{cases}
    \ee
\end{proposition}
Let us note that the above can be rephrased in the following way:
\be\label{e.c060702}
    \cW_a = \left(\cW \cup \Psi_a(\cW) \right) \cap \overline{\Omega_a}.
\ee
Here $\cW = \cW_{\sfrac\pi2}$ is the Wulff shape on the half space $\H$.

We can read off several results from this. 
\begin{proposition}\label{p.240719.1}
Let $\mathcal{W}_a$ be as above, then
\begin{equation}\label{e.240719.2}
    \mathcal{W}_a \subset \{(x,y) \in \overline{\Omega_a}:~ {\rm dist}((x,y),\Gamma_0 \cup \Gamma_a) \leq 2\}.
\end{equation}
In particular, $c_{*a}(\sfrac\pi2-a) \leq 2\csc a$ is bounded uniformly in $D >1$.
\end{proposition}
\begin{proof}
    The set inclusion in \eqref{e.240719.2} follows from \eqref{e.c060702} and $\mathcal{W} \subset \{(x,y) \in \overline\H:~ y \leq 2\}$. 
\end{proof}

\begin{corollary}\label{cor.cone.cstar.1}
Let $\vt_*$ be defined by \Cref{prop:1.11}.  Then:
\begin{enumerate}[(i)]

    \item If $D\leq 2$, then $c_{*a}(\vt)=2$ for all $\vt \in [\tfrac{\pi}{2}-2a,\tfrac{\pi}{2}]$.

    \item If $D>2$, and $a\in (0, \tfrac{\pi}{2} -\vt_*]$, then $c_{*a}(\vt)>2$ for all $\vt \in [\tfrac{\pi}{2}-2a,\tfrac{\pi}{2}]$.

    \item If $D>2$, and $a\in [\tfrac{\pi}{2} - \vt_*, \tfrac{\pi}{2})$, then 
    \be
    c_{*a}(\vt)>2 \quad \text{ for }\vt \in (\tfrac{\pi}{2}-\vt_*, \tfrac{\pi}{2}] \cup [\tfrac{\pi}{2} -2a, \tfrac{\pi}{2} -2a + \vt_*).
    \ee
    and $c_{*a}(\vt) = 2$ otherwise.
    
\end{enumerate}
\end{corollary}

From \Cref{p.J_a_J}, we can also easily deduce the nonconvexity of $\cW_a$.   Notice that this is a generic phenomenon, and the strict convexity (\Cref{l.W_convex}) of $\cW$ is an idiosyncracy of the case $a= \sfrac\pi2$. 
\begin{proposition}\label{p.non_convexity}
We have the following regaring the convexity of $\cW_a$:
\begin{enumerate}[(i)]

    \item\label{i.convex_iff}
    $\mathcal{W}_a$ is convex if and only if $a \geq \sfrac{\pi}{2} - \vt_*$, where $\vt_*$ is given in \Cref{prop:1.11}.

    \item\label{i.D_a_exists}
    For each $a \in (0,\sfrac\pi2)$, there exists $D_a \in [2,\infty)$ depending on $a$ such that $\mathcal{W}_a$ is convex for $D \in (1,D_a]$ and non-convex for any $D \in ( D_a,\infty)$. 

    \item\label{i.D_a_bound}
        $D_a \leq 4(2+\mu)^2\csc^2(2a)$.  In particular, $\mathcal{W}_a$ is nonconvex when $D \geq 4(2+\mu)^2\csc^2(2a)$.

\end{enumerate}
   
\end{proposition}
\begin{proof}
To prove \ref{i.convex_iff}, notice that $\mathcal{W}_a$ is convex if and only if there is a supporting hyperplane at $(\bar x, \bar y) = c_{*}(\sfrac{\pi}{2}-a)(\sin (\sfrac{\pi}{2}-a), \cos(\sfrac{\pi}{2}-a))$.  See \Cref{f.triangles} for an illustration, where $\hat u_a = (\sin(\sfrac\pi2-a), \cos(\sfrac\pi2-a))$.
This is the case if and only if $ \tfrac{d}{d\vt} c_*(\tfrac{\pi}{2}-a) \leq 0$, which is equivalent to $a \geq \sfrac\pi2 - \vt_*$, thanks to \Cref{prop:1.11}.\ref{i.D_>_2}. 

The assertion \ref{i.D_a_exists} follows from the fact that 
\be\label{e.c071908}
    c_{*a}(\sfrac\pi2-a) \leq 2\csc a
\ee
(from \Cref{p.240719.1}) and that, by \Cref{c.D_near_2,p.J_a_J}, $c_{*a}(\sfrac\pi2) = c_{*a}(\sfrac\pi2-2a) = c_*(\sfrac\pi2) \to \infty$ as $D \to \infty$. Thus there exists $D_a$ such that $\mathcal{W}_a$ is nonconvex for $D > D_a$.

We now prove assertion \ref{i.D_a_bound}. By \Cref{p.J_a_J}, $c_{*a}(\sfrac\pi2) = c_{*a}(\sfrac\pi2-2a) = c_*(\sfrac\pi2)$.  Hence, $\cW_a$ is nonconvex if
\be\label{e.c071909}
    c_*(\sfrac\pi2) \cos(a)
        > c_*(\sfrac\pi2-a).
\ee
See \Cref{f.triangles} for an illustration.  Using \Cref{c.D_near_2}.\ref{i.speed_bounds2} to bound the left hand side and~\eqref{e.c071908} to bound the right hand side, we see that a sufficient condition for~\eqref{e.c071909} (and, thus, $\cW_a$ is nonconvex) is given by  
\be
    \frac{2 \sqrt D}{2 + \mu} \cos(a)
        \geq 2 \csc a.
\ee
The conclusion follows by rearranging the above.  The proof is complete.  
\end{proof}
\begin{figure}
\centering
    \begin{tikzpicture}
        \def\a{29} 
        \def\length{5} 
        \def\r{.05} 
    
        \draw[<->] (-.5,0) -- (1.3*\length,0) node[right] {$x$};
        \draw[->] (0,-.5) -- (0,.9*\length) node[above] {$y$};
    
        \draw[violet]
            (0,0) -- ({\length*cos(2*\a)}, {\length*sin(2*\a)});
        \fill[violet]
            ({\length*cos(2*\a)}, {\length*sin(2*\a)})
                circle (\r)
            node[violet,above left] {\tiny$\Psi_a(c_*(\sfrac\pi2),0)$};
            
        \draw[violet] (0,0) -- (\length, 0);
        \fill[violet]
            (\length,0)
                circle (\r)
            node[violet,below] {\tiny$(c_*(\sfrac\pi2),0)$};

        \draw[violet] (\length,0) -- ({\length*cos(2*\a)},{\length*sin(2*\a)});
    
        \draw[red, dashed] (0,0) -- ({1.2*\length*cos(\a)}, {1.2*\length*sin(\a)});

        \fill[violet]
            ({\length*cos(\a)*cos(\a)}, {\length*sin(\a)*cos(\a)})
                circle (\r)
            node[violet, right]
                {\tiny$c_*(\sfrac\pi2)\cos(a) \hat u_a$};

        \draw[red,opacity=.3]
            ({.97*\length*cos(\a)*cos(\a)}, {.97*\length*sin(\a)*cos(\a)})
            -- ({.97*\length*cos(\a)*cos(\a)+.03*\length*cos(\a)*sin(\a)}, {.97*\length*sin(\a)*cos(\a)-.03*\length*cos(\a)*cos(\a)})
            -- ({\length*cos(\a)*cos(\a)+.03*\length*cos(\a)*sin(\a)}, {\length*sin(\a)*cos(\a)-.03*\length*cos(\a)*cos(\a)});
    
        \draw ({.95*(\length/3)*cos(\a)},{.95*(\length/3)*sin(\a)}) arc[start angle=\a, end angle={2*\a}, radius=.95*(\length/3)];
        
        \draw ({1.05*(\length/3)*cos(\a)},{1.05*(\length/3)*sin(\a)}) arc[start angle=\a, end angle={2*\a}, radius=1.05*(\length/3)]
            node[midway, above right]{\tiny $a$};
        
        \draw ({\length/3},0) arc[start angle=0, end angle=\a, radius=\length/3]
            node[midway, right]{\tiny $a$};
        
        \draw[dotted] ({(\length/2)*cos(2*\a)},{(\length/2)*sin(2*\a)}) arc[start angle=2*\a, end angle=90, radius=\length/2]
            node[midway, above ]{\tiny $\frac\pi2-2a$};
    
        \fill[blue, opacity=0.05]
            (0,0) -- plot[domain=0:{\a}, samples=50] ({\length*((1-.25*(\x/\a))*cos(\x))}, {\length*((1-.25*(\x/\a))*sin(\x))}) -- cycle;
        \fill[blue, opacity=0.05]
            (0,0) -- plot[domain={\a}:{2*\a}, samples=50] ({\length*((1-.25*((2*\a-\x)/\a))*cos(\x))}, {\length*((1-.25*((2*\a-\x)/\a))*sin(\x))}) -- cycle;

        \fill[blue]
            ({.75*\length*cos(\a)}, {.75*\length*sin(\a)})
                circle (\r)
            node[blue, above left]
                {\tiny $c_*(\sfrac\pi2-a)\hat u_a$};


    \end{tikzpicture}

    \caption{An cartoon example illustrating the nonconvex shape $\mathcal{W}_a$. 
    The violet, solid outline is the boundary of the region formed by the convex hull of the points $(c_*(\sfrac\pi2,0)$ and  $\Psi_a(c_*(\sfrac\pi2,0)$,
    while the blue shaded region is the Wulff shape.  Above, we denote the unit vector $\hat u_a = (\cos(a), \sin(a))$.}\label{f.triangles}
\end{figure}
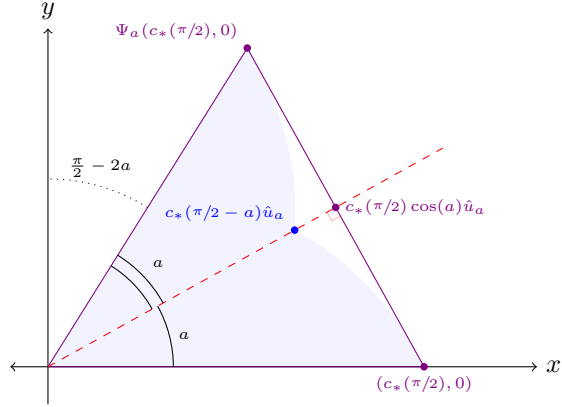

\section{Strong and weak viscosity solutions}\label{s.strong_solutions}

For the convenience of the reader, we provide the definitions of weak and strong viscosity solutions in our context.  We note that the notion of weak viscosity solution here corresponds to the standard notion of viscosity solution as in~\cite{UsersGuide}, while the notion of strong viscosity solution appeared more recently (see, e.g. \cite{Imbert2017flux,Imbert2017quasi,Guerand2017effective,Forcadel2023non} for the problem without obstacle).

Fix the domain $Q=(0,\infty)\times \H^2$ and any Hamiltonians $H, F:\mathbb{R}^2 \to \overline{\mathbb{R}}$ that correspond to the field and the road, respectively.  Consider the (variational) Hamilton-Jacobi equation:
\begin{equation}\label{e.c050801}
	\begin{cases}
		\min\{\rsig, \rsig_t + H(\rsig_x,\rsig_y)\} = 0 &\text{ in }(0,\infty) \times \H,\\
		\min\{\rsig, \rsig_t + F(\rsig_x,\rsig_y)\} = 0 &\text{ in }(0,\infty) \times \partial \H.
	\end{cases}
\end{equation}
We first define the notion of a weak viscosity solution.
\begin{definition}[Weak viscosity solutions]

    \begin{enumerate}[(i)]

        \item Let $\rsig : \bar{Q} \to \mathbb{R}$ be upper semicontinuous. We say that $\rsig$ is a weak subsolution to \eqref{e.c050801} if for any point $(t_0,x_0,y_0) \in Q$ such that $\rsig(t_0,x_0,y_0)>0$, and any $C^1$ function $\varphi$ touching $\rsig$ from above, then
    \begin{equation}\label{e.c050802}
        \begin{array}{lcll}
			y_0>0
				& \quad \Longrightarrow \quad
				&\varphi_t + H(\varphi_x,\varphi_y) \leq 0 \qquad
				& \text{ at }(t_0,x_0,y_0),\\
			y_0=0
				&\quad \Longrightarrow \quad
				&\text{either }\quad \varphi_t + H(\varphi_x,\varphi_y) \leq 0
				\quad \text{ or } \quad \varphi_t + F(\varphi_x,\varphi_y) \leq 0 
				& \text{ at }(t_0,x_0,y_0).
        \end{array}
    \end{equation}
    We note that the latter can be written as $\min\{\phi_t + H(\nabla \phi), \phi_t + F(\nabla \phi)\}\leq 0$.  This is how it is presented in~\cite{UsersGuide}.
	
	\item Let $\rsig : \bar{Q} \to [0,\infty)$ be lower semicontinuous. We say that $\rsig$ is a weak supersolution to \eqref{e.c050801} if $\rsig \geq 0$ in $Q$, and if
        for any point $(t_0,x_0,y_0) \in Q$ and any $C^1$ function $\varphi$ touching $\rsig$ from above, then 
        \begin{equation}\label{e.c050803}
	        \begin{array}{lcll}
				y_0>0
					& \quad \Longrightarrow \quad
					&\varphi_t + H(\varphi_x,\varphi_y) \geq 0 \qquad & \text{ at }(t_0,x_0,y_0),\\
				y_0=0
					&\quad \Longrightarrow \quad
					&\text{either }\quad \varphi_t + H(\varphi_x,\varphi_y) \geq 0 \quad
					\text{ or } \quad \varphi_t + F(\varphi_x,\varphi_y) \geq 0 
					& \text{ at }(t_0,x_0,y_0).
        \end{array}
    \end{equation}
        We note that the latter can be written as $\max\{\phi_t + H(\nabla \phi), \phi_t + F(\nabla \phi)\}\geq 0$.  This is how it is presented in~\cite{UsersGuide}.
	
	\item Let $\rsig \in C_{\rm loc}(\bar Q)$. We say that $\rsig$ is a weak solution to \eqref{e.c050801} if $\rsig$ is both a weak subsolution and a weak supersolution.
	
    \end{enumerate}
\end{definition}

We contrast this with the notion of a strong solution.  Notice that the alternatives in~\eqref{e.c050802} and~\eqref{e.c050803} do not appear in~\eqref{e.c050804} and~\eqref{e.c050805}.

\begin{definition}[Strong viscosity solutions]     
    \begin{enumerate}[(i)]
    
        \item Let $\rsig : \bar{Q} \to \mathbb{R}$ be upper semicontinuous. We say that $\rsig$ is a strong subsolution to \eqref{e.c050801} if for any point $(t_0,x_0,y_0) \in Q$ such that $\rsig(t_0,x_0,y_0)>0$, and any $C^1$ function $\varphi$ touching $\rsig$ from above, then
    \begin{equation}\label{e.c050804}
        \begin{array}{lcll}
			y_0>0
				& \quad \Longrightarrow \quad
				&\varphi_t + H(\varphi_x,\varphi_y) \leq 0 \qquad
				& \text{ at }(t_0,x_0,y_0),\\
		y_0=0
			&\quad \Longrightarrow \quad 
			& \varphi_t +  F(\varphi_x,\varphi_y) \leq 0  
			& \text{ at }(t_0,x_0,y_0).
        \end{array}
    \end{equation}

        \item Let $\rsig : \bar{Q} \to [0,\infty)$ be lower semicontinuous. We say that $\rsig$ is a strong supersolution to \eqref{e.c050801} if $\rsig \geq 0$ in $Q$, and if for any point $(t_0,x_0,y_0) \in Q$ and any $C^1$ function $\varphi$ touching $\rsig$ from above, then 
        
        \begin{equation}\label{e.c050805}
        \begin{array}{lcll}
			y_0>0
				& \quad \Longrightarrow \quad 
				&\varphi_t + H(\varphi_x,\varphi_y) \geq 0 \qquad 
				& \text{ at }(t_0,x_0,y_0),\\
			y_0=0
				&\quad \Longrightarrow \quad 
				& \varphi_t + F(\varphi_x,\varphi_y) \geq 0  
				& \text{ at } (t_0,x_0,y_0).
        \end{array}
    \end{equation}

        \item Let $\rsig \in C_{\rm loc}(\bar Q)$. We say that $\rsig$ is a strong solution to \eqref{e.c050801} if $\rsig$ is both a strong subsolution and a strong supersolution.

    \end{enumerate}
\end{definition}

Let us make a small note that the terminology used for strong and weak solutions is not consistent across the literature.  In place of using the, perhaps more common, terminology of strong (resp. weak) $F$-solution or $\tH$-flux-limited solution, we opt for the simpler terminology above.  This follows the terminology used in the classic text of Crandall, Ishii, and Lions; see \cite[Section~7A]{UsersGuide}.

Let us also note that, in practice, the Hamiltonian on the boundary is different for the weak and strong solutions.  More precisely, given a weak solution $\rsig$ to a Hamilton-Jacobi equation with Hamiltonian $H$ in the field and $F_0$ on the road, it is usually a strong solution to a Hamilton-Jacobi equation with Hamiltonian $H$ in the field and $F$ on the road, with $F \neq F_0$.  This is the case in our setting.  To compensate for this, we always reference the specific boundary conditions when we write ``strong solution'' or ``weak solution.''

\section{The half-relaxed limits and convergence of $(u^\eps,v^\eps$)}\label{s.half_relaxed}

In this section, we prove \Cref{t.uvw} and \Cref{t.strong} concerning the convergence of $v^\eps, u^\eps$ to $w$, which is the weak solution of~\eqref{e.HJ}-\eqref{e.HJ_F_0} and the strong solution of~\eqref{e.HJ}-\eqref{e.HJ_F}.  We do so by using the method of half-relaxed limits.  A main issue is to connect these with the notion of {\em strong} sub- and supersolutions so that the comparison principle may be applied. Note that this section does not depend on the results in \Cref{s.control_corollaries}.



\subsection{Definition and preliminary bounds}

Let $u^*, u_*$ (resp. $v^*, v_*$) be the half-relaxed limits of $u^\ep$ (resp. $v^\ep$):
\begin{equation}\label{e.half_relaxed}
    v^*(t,x,y) = \limsup_{(t',x',y') \to (t,x,y) } v(t',x',y')
    \quad\text{ and }\quad
    v_*(t,x,y) = \liminf_{(t',x',y') \to (t,x,y) } v(t',x',y'),
\end{equation}
with $u_*$ and $u^*$ are defined similarly.  Let
\begin{equation}\label{e.hr25}
    \rsig^*(t,x,y)
        = \begin{cases}
            v^*(t,x,y)
                \qquad&\text{ if } y > 0,\\
            \max\{u^*(t,x),v^*(t,x,0)\}
                \qquad&\text{ if } y = 0,
        \end{cases}
\end{equation}
and
\begin{equation}\label{e.hr26}
    \rsig_*(t,x,y)
        = \begin{cases}
            v_*(t,x,y)
                \qquad&\text{ if } y > 0,\\
            \min\{u_*(t,x),v_*(t,x,0)\}
                \qquad&\text{ if } y = 0.
        \end{cases}
\end{equation}
We immediately notice that, by construction,
\begin{equation}
    \rsig_* \leq \rsig^*.
\end{equation}
Hence, the locally uniform convergence in \Cref{t.uvw} follows if we show that
\be\label{e.c052402}
    \rsig_* \geq \rsig^*.
\ee
We do this by showing that these are, respectively, super- and subsolutions of the same equation and that comparison holds for that equation, despite the infinite, discontinuous initial data.  Let us note that, by construction, $w^*$ is lower semicontinuous and $w_*$ is upper semicontinuous.

One might be concerned about the finiteness of $\rsig_*$ and $\rsig^*$.  It follows immediately from~\eqref{e.initial_data2} and a straightforward comparison principle argument that
\be\label{e.c052403}
	\frac{\mu}{\nu}\|U\|_\infty, \|V\|_\infty
		\leq 1
\ee
(recall~\eqref{e.initial_data2}).  
We, thus, deduce that
\begin{equation}\label{e.nonnegative}
    \rsig^*
        \geq \rsig_*
        \geq \liminf_{\eps\to0} \eps \log \frac{1}{\min\left\{1, \sfrac{\nu}{\mu}\right\}}
        = 0.
\end{equation}
Next, we state the upper bound, the proof of which is slightly more involved, although not difficult. It is delayed until \Cref{ss.w_lemmas}.
\begin{lemma}\label{l.upper_bound}
    Under the assumptions of \Cref{t.uvw} and given $T\in[0,\infty)$, there is $A_T>0$ such that
    \begin{equation}
        u^\eps(t,x), v^\eps(t,x,y) \leq \frac{A_T}{t}\Big(1 + x^2 + y^2\Big)
            \qquad\text{ for all } t\in[0,T]
                \text{ and } \eps >0.
    \end{equation}
    Consequently, we have
    \begin{equation}\label{e.2.11}
        0 \leq \rsig_*(t,x,y)
            \leq \rsig^*(t,x,y)
            \leq \frac{A_T}{t}\Big(1 + x^2 + y^2\Big)
            \qquad\text{ for all } t\in[0,T].
    \end{equation}
\end{lemma}

Next, one complication is that, formally, we expect $w^*(0,\cdot,\cdot) = +\infty$ everywhere. This is one issue that prevents us from using the comparison principle to deduce~\eqref{e.c052402}.  To sidestep this issue, we construct suitable sub- and supersolutions to control the behavior of $w_*$ and $w^*$ at $(0,0)$ for positive times and at any $(x,y)\neq (0,0)$ at time zero.  We state these results here but again postpone the proofs until \Cref{ss.w_lemmas}.

\begin{lemma}\label{lem:2.4a}
$w_*(t,0,0) = w^*(t,0,0) = 0$ for $t>0$. 
\end{lemma}

\begin{lemma}\label{lem:2.4}
$w_*(0,x,y) = w^*(0,x,y) = +\infty$ for $(x,y)\neq (0,0).$ 
\end{lemma}

\subsection{The relaxed equation}

Our first main lemma is the following (recall the definition of $F_0$ in~\eqref{e.B}):

\begin{lemma}\label{l.hr}
    The half-relaxed limits have the following properties.
    \begin{enumerate}[(i)]
    
    \item\label{i.rho^*_relaxed} The upper relaxed limit $\rsig^*$ is a weak subsolution to \eqref{e.HJ}-\eqref{e.HJ_F_0},
and satisfies
        \begin{equation}
            \rsig^*(0,x,y) = +\infty
                \qquad\text{ for all } (x,y) \in \overline \H.
        \end{equation}

    \item\label{i.rho_*_relaxed} The lower relaxed limit $\rsig_*$ is a weak supersolution to \eqref{e.HJ}-\eqref{e.HJ_F_0},
        and satisfies
        \begin{equation}
            \rsig_*(0,x,y) = 
                \begin{cases}
                    +\infty
                        \qquad&\text{ for all } (x,y) \in \overline \H\setminus\{(0,0)\},\\
                    0
                        \qquad&\text{ if } (x,y) = (0,0).
                \end{cases}
        \end{equation}
        
    \item\label{i.rho^*} There is ${\rho^*}$ such that 
        \begin{equation}\label{e.l309182}
            \rsig^*(t,x,v) = t {\rho^*}(\sfrac{x}{t}, \sfrac{y}{t}).
        \end{equation}
        and that is upper semi-continuous is a weak subsolution to 
        \begin{equation}\label{e.sigma*}
            \begin{cases}
				\min\{{\rho^*} - (x,y) \cdot \nabla {\rho^*} + H(\nabla \rho^*), {\rho^*}\} \leq 0
                    \quad &\text{ in } \H,\\
				\min\{ \rho^* - (x,y)\cdot\nabla \rho^* + F_0(\nabla \rho^*), \rho^*\}
					\leq 0
						\quad &\text{ in } \R\times \{0\}.
            \end{cases}
        \end{equation}
        
    \end{enumerate}
\end{lemma}
Given \Cref{lem:2.4a} and \Cref{lem:2.4} regarding the initial data, 
we note that the proofs of \Cref{l.hr}.\ref{i.rho^*_relaxed} and \Cref{l.hr}.\ref{i.rho_*_relaxed} are essentially the same.  Hence, we omit the latter. 
\begin{proof}[Proof of \Cref{l.hr}.\ref{i.rho^*_relaxed}]
The lower semicontinuity of $u^*$ and $v^*$, which is then inherited by $\rsig^*$ follows directly from the definition of the half-relaxed limits~\eqref{e.half_relaxed}.  Additionally, the behavior of the initial data follows from  \Cref{lem:2.4}.  Hence, we focus only on showing that $\rsig^*$ is a subsolution to the constrained Hamilton-Jacobi equation.

Fix an test function $\phi$.  Suppose $\rsig^* - \phi$ has a strict local maximum point at $(t_0,x_0,y_0)$ with $t_0 > 0$ and $y_0 \geq 0$.  If $y_0 >0$, standard arguments arguments apply directly and we omit the argument.  We, thus, consider only the case $y_0 = 0$, i.e. $\rsig^* - \phi$ has a strict local maximum point at $(t_0,x_0,0)$.  If $w^*(t_0,x_0,0) = 0$, the conclusion is immediate.  Hence, we assume that $w^*(t_0,x_0,0)>0$.  We break the argument up into two cases.

{\bf Case one: $\phi_y(t_0,x_0,0) \leq -\kappa\nu$.} 
Fix a constant $A>0$ such that
\begin{equation}\label{eq:BB}
    \nu A >  -\phi_t(t_0,x_0,0) - D|\phi_x(t_0,x_0,0)|^2 +\mu.
\end{equation}
Since $(t_0,x_0,0)$ is a strict local maximum, there is $(t_\eps, x_\eps, y_\eps)$ tending to $(t_0,x_0,0)$ that is a local maximum of $\max\{v^\eps + \eps \log A,u^\ep\} - \phi$.  (Actually, this occurs along a subsequence of $\eps_n$ tending to zero, but we omit this notationally).

We claim that $y_\eps > 0$ for all $\eps$ sufficiently small.  Were this the case, it follows by standard arguments (that we omit) that
\begin{equation}\label{e.l220918a}
    \min\{\phi, \phi_t + |\nabla \phi|^2 + 1\} \leq 0
        \quad\text{ at } (t_0, x_0, 0).
\end{equation}
We show that $y_\eps>0$ by contradiction.  Suppose that $y_\eps = 0$. If, at $(t_\eps,x_\eps,0),$
\[
    v^\eps + \eps \log A > u^\eps,
\]
then $v_\ep + \ep\log A - \phi$ has a local maximum at $(t_\ep,x_\ep,0)$. By the third equation in~\eqref{e.scaled_eqn}, we find \begin{equation}\label{eq:step22}
    \begin{split}
        0 &\geq \partial_y(v^\ep - \phi) \big|_{(t_\ep,x_\ep,0)} 
        = k\left[\mu e^{\frac{v^\ep-u^\ep}{\ep}}-\nu\right] - \phi_y(t_0,x_0,0) + o(1)
        \\&
        \geq \frac{\kappa\mu}{A}-\kappa\nu - \phi_y(t_0,x_0,0) + o(1)
        \geq \frac{\kappa\mu}{A}   + o(1), 
    \end{split}
\end{equation}
where the last inequality follows by the assumption $\phi_y(t_0,x_0,0) \leq -\kappa\nu$ and that $o(1)$ is a quantity tending to zero as $\eps \to 0$.  This is a contradiction as $\kappa\mu/A >0$.

On the other hand, if, at $(t_\eps,x_\eps,0),$
\[
    v^\eps + \eps \log A \leq u^\eps,
\]
then $u_\ep - \phi$ has a local maximum at $(t_\ep,x_\ep,0)$.
The second equation in~\eqref{e.scaled_eqn}, implies that, at the point $(t_\eps, x_\eps, 0)$,
\begin{equation}\label{eq:BBB}
    \begin{split}
        0
        &\geq \phi_t -\ep D \phi_{xx} + D|\phi_x|^2 +  \nu e^{\frac{u^\ep - v^\ep}{\ep}}-\mu
        \\&
        \geq \phi_t -\ep D \phi_{xx} + D|\phi_x|^2 + \nu  A-\mu
        = \left(\nu A + \phi_t + D |\phi_x|^2 - \mu\right) + o(1)
        > 0.
    \end{split}
\end{equation}
In the last inequality, we used~\eqref{eq:BB} and took $\eps$ sufficiently small.  This is clearly a contradiction.  Since we have reached a contradiction in all cases, we conclude that $y_\eps > 0$ for $\eps$ sufficiently small.

{\bf Case two: $\phi_y(t_0,x_0,0) > -k \nu$.}   fix an arbitrary constant $\tilde{A}$ such that
\begin{equation}\label{e.c080901}
    0 < \tilde A < \frac{\kappa\mu}{\kappa\nu + \phi_y(t_0,x_0,0)}.
\end{equation}
Again, since $(t_0,x_0,0)$ is a strict local maximum, there is $(t_\eps, x_\eps, y_\eps)$ tending to $(t_0,x_0,0)$ that is a local maximum of $\max\{v^\eps + \eps \log \tilde A,u^\eps\} - \phi$.  (Actually, this again occurs along a subsequence of $\eps_n$ tending to zero, but we omit this notationally).

If $y_\eps > 0$ for $\eps$ sufficiently small, then we again conclude~\eqref{e.l220918a} by standard (and, thus, omitted) arguments.  Thus we suppose that $y_\eps = 0$ for all $\eps$ small.  Then, as in~\eqref{eq:step22}, if,
\begin{equation}\label{e.l221002a}
    v^\eps + \eps \log \tilde A \geq u^\eps
        \qquad\text{ at } (t_\eps, x_\eps,0),
\end{equation}
we find
\begin{equation}
    \begin{split}
        0 &\geq \partial_y(v^\ep - \phi) \big|_{(t_\ep,x_\ep,0)} 
        = k\left[\mu e^{\frac{v^\ep-u^\ep}{\ep}}-\nu\right] - \phi_y(t_0,x_0,0) + o(1)
        \\&
        \geq  k\left[\mu e^{- \log \tilde A}-\nu\right] - \phi_y(t_0,x_0,0) + o(1)
        >0
    \end{split}
\end{equation}
where the last inequality followed by the choice of $\tilde{A}$ in~\eqref{e.c080901} and taking $\eps$ small enough.

It follows that \eqref{e.l221002a} does not hold, and hence
\begin{equation}
    v^\eps + \eps \log \tilde A < u^\eps
        \qquad\text{ at } (t_\eps, x_\eps,0),
\end{equation}
when $\rsig^*(t_0,x_0,0)$ takes the limit value of $u^\eps$ at $(t_\eps, x_\eps, 0)$.  We, thus, use the second equation in~\eqref{e.scaled_eqn} and deduce, at $(t_\eps, x_\eps,0)$,
\begin{equation}
    0 \geq \phi_t - \eps D \phi_{xx} + D |\phi_x|^2 + \nu e^\frac{u^\eps - v^\eps}{\eps} - \mu
        > \phi_t - \eps D \phi_{xx} + D |\phi_x|^2 + \nu \tilde A - \mu.
\end{equation}
Taking $\eps$ to zero and then
\begin{equation}
    \nu \tilde A \to \frac{\kappa\nu \mu}{\kappa\nu + \phi_y(t_0,x_0,0)}
        =  \Bo(\phi_y(t_0,x_0,0)) + \mu, 
\end{equation}
we deduce, at $(t_0,x_0,0),$
\begin{equation}
    0 \geq \phi_t - \eps D \phi_{xx} + D |\phi_x|^2 + \Bo(\phi_y).
\end{equation}
This concludes the proof.
\end{proof}

\begin{proof}[Proof of \Cref{l.hr}.\ref{i.rho^*}]
Define, for any $(x,y) \in \overline \H$,
\begin{equation}\label{e.c081002}
    {\rho^*}(x,y) = \rsig^*(1,x,y).
\end{equation}
We claim that, for all $(t,x,y) \in (0,\infty) \times \H$ and $A>0$,
\begin{equation}\label{e.c081001}
    \rsig^*(At,Ax,Ay) = A \rsig^*(t,x,y).
\end{equation}
Notice that, if~\eqref{e.c081001}, holds, then, applying it with $A = \sfrac{1}{t}$ at $(1,\sfrac{x}{t},\sfrac{y}{t})$ and then using~\eqref{e.c081002}, we derive \eqref{e.l309182}.
Hence, we need only establish~\eqref{e.c081001}.  This is easy seen by using~\eqref{e.thin_front_scaling}:
\begin{equation}
    v^\eps(At,Ax,Ay)
        = A v^{\sfrac{\eps}{A}}(t, x, y)
    \qquad\text{ and }\qquad
    u^\eps(At,Ax)
        = A u^{\sfrac{\eps}{A}}(t, x).
\end{equation}
We deduce~\eqref{e.c081001} in the limit $\eps\to 0$.

In order to derive~\eqref{e.sigma*} from    \eqref{e.HJ} and weak $F_0$-condition on the boundary, one simply takes test functions $\psi: \overline \H \to \R$ for ${\rho^*}$ and transforms them to test functions $\tilde\psi(t,x,y)$ for $\rsig^*$ via
\begin{equation}
    \tilde \psi(t,x,y) = t \, \psi \left(\sfrac{x}{t}, \sfrac{y}{t}\right)
\end{equation}
and then argues in a straightforward manner.  We omit the details.
\end{proof}

In order to apply comparison principle arguments, we require some regularity of $\rsig^*$.  We do this by establishing a local Lipschitz bound $\rsig^*$ via $\rho^*$, which formally follows from the convexity of the Hamiltonian in  \eqref{e.sigma*}, for instance:
\be\label{e.c050807}
	\left|\nabla {\rho^*} - \frac{(x,y)}{2}\right|^2
		\leq - (x,y) \cdot \nabla {\rho^*} + |\nabla {\rho^*}|^2 + \frac{|(x,y)|^2}{4}
		\leq \frac{|(x,y)|^2}{4}
\ee
in $\{(x,y)\in \H:~ {\rho^*}(x,y) >0\}$.
\begin{lemma}\label{l.Lipschitz}
Let  ${\rho^*}$ be the function given in \Cref{l.hr}.\ref{i.rho^*}. Then the following  statements hold.
\begin{itemize}
    \item[(i)] There is $A_R>0$ such that for all $0\leq |x_i|\leq R$ and $0\leq y_i \leq R$, 
    \begin{equation} \label{eq:l.Lipschitz.1}
        |{\rho^*}(x_0,y_0) - {\rho^*}(x_1,y_1)|
            \leq \frac{A_R|(x_0 - x_1, y_0 - y_1)|}{\sqrt{\min\{y_0,y_1\}}}.
    \end{equation}
     In particular, ${\rho^*}$ is locally Lipshitz continuous in $\H$.
    \item[(ii)] ${\rho^*}$ is locally $C^{\sfrac12}$ on $\overline \H$.
\end{itemize}  
\end{lemma}
 The proof is postponed until \Cref{ss.w_lemmas}; however, it follows the rough idea given in the discussion around~\eqref{e.c050807}.  An important consequence of \Cref{l.Lipschitz} is that, since $w^*(t,x,y) = t\rho^*(x/t,y/t)$, 
 \be\label{e.c051304}
	w^*(t,x,y) \in C_{\rm loc}((0,\infty)\times \overline\H).
\ee

\subsection{The flux-limited equation}

The last step to applying the comparison principle is to show that $w_*$ and $w^*$ are, respectively, {\em strong} sub- and supersolutions of~\eqref{e.HJ}-\eqref{e.HJ_F}.  We show this now.

\subsubsection{Critical slopes}

For each viscosity subsolution or supersolution to \eqref{e.HJ} in the  interior of the domain $\H$, additional information can be obtained by considering ``critical slopes,''
which were introduced in \cite{Imbert2017flux}.

Let $\rsig: [0,\infty)\times\overline{\H} \to \mathbb{R}$ be given. Suppose $\varphi$ is a test function that touches $\rsig$ from below at a boundary point $(t_0,x_0,0)$, then one can define $\underline p$ to be the maximal number such that $\tilde\varphi(t,x,y) = \varphi(t,x,y) + \underline p y$ touches $\rsig$ from below at $(t_0,x_0,0)$. Furthermore, the test function $\tilde\varphi$ with the critical slope inherits the subsolution inequality provided $\rsig$ is a viscosity solution to the equation \eqref{e.HJ} in the interior of the upper half plane $\H$. We state this formally now.

\begin{lemma}
	[Critical slopes for super-solutions]
	\label{l.critsuper}
Let $\overline\rsig$ be a super-solution to~\eqref{e.HJ}, 
and let $\varphi$ be an arbitrary test function touching $\overline\rsig$ from below at any point $(t_0,x_0,0)$. Then the ``critical slope," defined as 
\be\begin{split}
	\underline{p}
		= \sup\{p \geq
		0:~\exists r>0,
    			~\varphi(t,x,y) + py \leq \overline\rsig (t,x,y)~\text{ for }
			~(t,x,y) \in B_r(t_0,x_0,0)\text{ with }y\geq 0\},
\end{split}\ee
satisfies that either $\underline p = +\infty$ or 
\be
	\varphi_t(t_0,x_0,0) + \Hf(\varphi_x(t_0,x_0,0),\varphi_y(t_0,x_0,0) + \underline{p}) \geq 0.
\ee
\end{lemma}
\begin{proof}
See \cite[Lemma 2.9]{Imbert2017flux} and \cite[Lemma A.9]{Imbert2017quasi}. 
\end{proof}

\begin{lemma}
	[Critical slopes for sub-solutions]
	\label{l.critsub}
Let $\underline\rsig$ be a subsolution to~\eqref{e.HJ}, 
and let $\varphi$ be an arbitrary test function touching $\underline\rsig$ from above at any point $(t_0,x_0,0)$. 
Assume that $\underline\rsig$ satisfies the weak continuity condition
\begin{equation}\label{eq:wcp}
    \underline\rsig(t_0,x_0,0) = \limsup_{(t,x,y) \to (t_0,x_0,0)} \underline\rsig(t,x,y).
\end{equation}
Then the ``critical slope," defined as
\be\label{e.bar_p}
    \overline{p}
    		= \inf\{p \leq
		0:
			~\exists r>0,
			~\varphi(t,x,y) + py \geq \underline\rsig(t,x,y)
				~\text{ for }~(t,x,y) \in B_r(t_0,x_0,0)
				\text{ with }y\geq 0\}
\ee
is finite and 
\begin{equation}
   \min\{\underline\rsig(t_0,x_0,0),\varphi_t(t_0,x_0,0) + H(\varphi_x(t_0,x_0,0),\varphi_y(t_0,x_0,0) + \overline{p}) \}\leq 0.
\end{equation}
\end{lemma}
\begin{proof}
See \cite[Lemma 2.10]{Imbert2017flux} \cite[Lemma A.10]{Imbert2017quasi}. 
\end{proof}

\subsubsection{Connecting weak and strong viscosity (sub/super) solutions}

We begin by observing 
that the definition~\eqref{e.Hr} of $\tH$ in \Cref{ss.flux} implies the following.
\begin{lemma}\label{lem.pq}
Fix $q \in \mathbb{R}$, then the following statements are equivalent.
\begin{enumerate}[(i)]
	\item $\tH(q) > \inf \Hf(q,\cdot)$;
	\item $\Hf(q,0) < F_0(q,0)$;
	\item there is a unique number $p_q> {\rm argmin}\,\Hf(q,\cdot)=0$   such that $\Hf(q,p_q) = F_0(q,p_q)$;
	\item $q^2 > \sfrac{1}{(D-1)}$
	\item the number $p_q$ given by~\eqref{e.pq} is positive.
\end{enumerate} 
If any of the above conditions hold, the numbers $p_q$ in (iii) and (v) are the same and $\tH(q) = \Hf(q,p_q)$.
\end{lemma}

\begin{proof}
The equivalence of (iii) and (v) is obvious.  We now show the equivalence of (i) - (iv).
For each fixed $q$, observe that ${\rm argmin}\,\Hf(q,\cdot) =0$,  $\Hf(q,\cdot)$ (resp. $F_0(q,\cdot)$) is strictly increasing (resp. strictly decreasing) in $(0,\infty)$, and that $\Hf(q,p) > F_0(q,p)$ for $p$ sufficiently large. Hence, the unique positive root $p_q>0$ exists if and only if $\Hf(q,0) < F_0(q,0)$. On the other hand,
\be
H(q,0) < F_0(q,0) \quad \Longleftrightarrow \quad  q^2 > \sfrac{1}{(D-1)} \quad \Longleftrightarrow \quad \tH(q) > \inf \Hf(q,\cdot),
\ee
thanks to the definition of $\tH$ in \eqref{e.tH}.
This completes the proof of the equivalence of all statements.

The fact that both $p_q$ are the same is a direct computation.  That $\tH(q) = \Hf(q,p_q)$ is also immediate from the fact that $\Hf(q,p)$ is increasing and $F_0(q,p)$ is decreasing for all $p\in [0,\infty)$.
\end{proof}

We now adapt the arguments of \cite[Proposition A.8]{Imbert2017quasi} in order to connect strong and weak solutions.  We remind the reader that $F = \max\{\tH, \Hf^-\}$, with $\tH$, and $\Hf^-$ defined by~\eqref{e.Hr} and~\eqref{e.Hf-}, respectively.

\begin{proposition}\label{thm:4.7}

   Let $F_0$ and $F$ be given respectively in \eqref{e.B} and \eqref{e.F}. Then every weak $F_0$-supersolution $\rsig$ (resp. weak $F_0$-subsolution additionally satisfying \eqref{eq:wcp}) to~\eqref{e.HJ}-\eqref{e.HJ_F_0} is a strong supersolution (resp. strong subsolution) of \eqref{e.HJ}-\eqref{e.HJ_F}.
\end{proposition}
\begin{proof}[Proof of \Cref{thm:4.7} for subsolutions]
Fix any test function $\varphi$ that touches $\rsig$ from above at $(t_0,x_0,y_0)$.  If $y_0 >0$ or $\rsig(t_0,x_0,0)=0$, there is nothing to show.  We, thus, assume that
\be
	y_0 = 0
	\quad\text{ and }\quad
	\rsig(t_0,x_0,0) > 0.
\ee 
Recall the definition of $p_{q_0}$ in \eqref{e.pq} and also the
the definition of $F_0$ in \eqref{e.B}.  Denote 
\begin{equation}\label{eq:notate}
	(q_0,p_0)
		= (\varphi_x(t_0,x_0,0),\varphi_y(t_0,x_0,0)),
		\quad \lambda = - \varphi_t(t_0,x_0,0),
		\quad \text{ and } \quad 
		\tH(q_0) = \Hf(q_0,p_{q_0}).
\end{equation}
By definition of weak solution to~\eqref{e.HJ}-\eqref{e.HJ_F_0}, we have
\begin{equation}\label{eq:relax1}
\min\{ F_0(q_0,p_0), \Hf(q_0,p_0)\} \leq \lambda.
\end{equation}
By \Cref{l.critsub}, there exists a critical slope $-\infty<\bar p \leq 0$ such that
\begin{equation}\label{eq:m1.2a}
\Hf(q_0,p_0 + \bar p) \leq \lambda.
\end{equation}

We claim that
\be\label{e.c050901}
	\tH(q_0) \leq \lambda.
\ee
Let us postpone its proof momentarily and show how to conclude the proof.  Since $\Hf^-$ is decreasing in $p$, we have
\be
	\Hf^-(q_0,p_0)
		\leq \Hf^-(q_0,p_0 + \bar{p})
        \leq \Hf(q_0,p_0+\bar p)
        \leq \lambda.
\ee
Combining this with~\eqref{e.c050901} yields
\be
	F(q_0,p_0)
		=\max\{\Hf^-(q_0,p_0),\tH(q_0)\} \leq \lambda,
\ee
which is precisely~\eqref{e.HJ_F}, finishing the proof.

We now establish~\eqref{e.c050901}.  If $\tH(q_0) \leq \Hf(q_0, p_0 + \bar p)$, then the conclusion is immediate from~\eqref{eq:m1.2a}.  Hence, we assume instead that
\be\label{e.c050903}
	\tH(q_0) > \Hf(q_0, p_0 + \bar p).
\ee 
By \Cref{lem.pq}, this implies that $F_0(q_0,\cdot)$ intersects the increasing part of $\Hf(q_0,\cdot)$ at $p_{q_0}>0$, i.e.
\begin{equation}\label{eq:m1.1}
  F_0(q_0,p_{q_0}) = \Hf(q_0,p_{q_0})= \tH({q_0}).
\end{equation}
Next, the combination of \eqref{e.c050903} and \eqref{eq:m1.1} implies that
\be
	\Hf(q_0,p_{q_0})> \Hf(q_0,p_0 + \bar p). 
\ee
Since $p_{q_0}>0$ and $\Hf(q_0,\cdot)$ is increasing on on $\R_+$, 
we see that
\be\label{e.c050902}
	p_0 + \bar p < p_{q_0}.
\ee
Recall that, due to the definition~\eqref{e.bar_p} of $\bar p$, we have
\be
	\rsig(t,x,y) \leq \phi(t,x,y) + \bar p y
		\qquad\text{ for $(t,x,y)$ sufficiently close to $(t_0,x_0,0)$}.
\ee
Since $\phi(t,x,y) \approx \phi(t,x,0) + p_0 y$, we deduce from~\eqref{e.c050902} that there is $r>0$ such that
\be
	\rsig(t,x,y)
		\leq \varphi(t,x,0) + p_{q_0} y
		\quad \text{ for }(t,x,y) \in B_r(t_0,x_0,0) \setminus\{(t_0,x_0,0)\}.
\ee
The above is an equality when $(t,x,y) = (t_0,x_0,0)$.  Hence, using $\phi(t,x,0) + p_{q_0}y$ as a test function of $\rsig$ in~\eqref{e.HJ}-\eqref{e.HJ_F_0}, we deduce
\be
	\tH(q_0)
    = \min\{F_0(q_0,p_{q_0}),\Hf(q_0,p_{q_0})\}
		\leq \lambda
\ee
where the first equality holds by~\eqref{eq:m1.1} and the second follows as in~\eqref{eq:relax1} by definition of weak subsolution. 
This proves \eqref{e.c050901}. Thus the claim is proved.
\end{proof}

\begin{proof}[Proof of \Cref{thm:4.7} for supersolutions]
Suppose that $\varphi$ touches $\rsig$ from below at $(t_0,x_0,0)$, and fix $\lambda, q_0, p_0,$ and $\tH(q_0)$ as in \eqref{eq:notate}.  Then $\rsig \geq 0$ everywhere, and
\be
\max\{H(q_0,p_0), F_0(q_0,p_0)\} \geq \lambda.
\ee
We only need to show 
\be\label{e.c050906}
	F(q_0,p_0)
		=\max\{ \Hf^-(q_0,p_0) , \tH(q_0)\}
		\geq \lambda.
\ee
Let $p_{q_0}\geq 0$ be given by~\eqref{e.pq} and  
let the critical slope $\bar p \in [0,+\infty]$ be as in \Cref{l.critsuper}. Let us first argue under the assumption that
\be\label{e.c051301}
	p_{q_0} < \bar p + p_0,
\ee
but we consider the opposite case afterwards. 

If \eqref{e.c051301} holds, then clearly, from the definition~\eqref{e.bar_p} of $\bar p$, that
\be\label{e.c050904}
	\varphi(t,x,0) + p_{q_0}y
	\quad\text{ touches $\rsig$ from below at $(t_0,0,0)$}.
\ee
By the definition of relaxed supersolution,~\eqref{e.c050904} implies that
\begin{equation}\label{eq:m1.3}
    \max\{F_0(q_0,p_{q_0}),\Hf(q_0,p_{q_0})\} \geq \lambda. 
\end{equation}

Consider first the case where $\tH(q_0) > \inf \Hf(q_0,\cdot)$.  Then $p_{q_0}$, defined above, is exactly the $p_{q_0}$ given in \Cref{lem.pq}.(iii).  Hence,
\be
	\tH(q_0)
		= \Hf(q_0,p_{q_0})
		= \max\{F_0(q_0,p_{q_0}),\Hf(q_0,p_{q_0})\}
		\geq \lambda,
\ee
which establishes~\eqref{e.c050906} in this case.

Now consider the opposite case, where 
$\tH(q_0) \leq  \inf \Hf(q_0,\cdot)$.  Then, \Cref{lem.pq}.(iv) and~\eqref{e.pq} implies that $p_{q_0} = 0$, combining with  \Cref{lem.pq}.(ii) and \eqref{eq:m1.3}  we get
\be\label{e.a240720.1}
	\Hf(q_0,0)
        = \max\{F_0(q_0,0), \Hf(q_0,0)\}
        \geq \lambda.
\ee 
Combining \eqref{e.a240720.1} with the obvious inequalities below (recall the definition~\eqref{e.Hf-})
\be
	\max\{\Hf^-(q_0,p_0),\tH(q_0)\}
		\geq \Hf^-(q_0,p_0)
		\geq \inf \Hf(q_0,\cdot)	
		= \Hf(q_0,0),
\ee
we deduce that
\be
	F(q_0,p_0)
		= \max\{\Hf^-(q_0,p_0),\tH(q_0)\}
		\geq \Hf(q_0,0)
		= \max\{F_0(q_0,0), \Hf(q_0,0)\}
		\geq \lambda.
\ee
This concludes the proof of \Cref{thm:4.7}, in the case that~\eqref{e.c051301} holds.

Let us now consider the case where~\eqref{e.c051301} does not hold, i.e. we assume
\be\label{e.c051303}
	p_{q_0} \geq \bar p + p_0,
\ee
In particular, we deduce that $0\leq \bar p < +\infty$, and, by \Cref{l.critsuper},
\be\label{e.c051302}
	\lambda
		\leq \Hf(q_0,p_0 + \bar p).
\ee
In case $p_0 + \bar p \leq 0$, then 
recalling that $\Hf^-$ is nonincreasing everywhere and is equal to $\Hf$ on $(-\infty,0]$, we find
\be
	\lambda \leq \Hf(q_0,p_0 + \bar p)
		= \Hf^-(q_0,p_0 + \bar p)
		\leq \Hf^-(q_0,p_0)
		< F(q_0, p_0)
\ee
which is the desired inequality~\eqref{e.c050906}.

In case $p_0 + \bar p >0$, then
\be\label{e.c062602}
    \lambda
        \leq \Hf(q_0,p_0 + \bar p)
        \leq \Hf(q_0, p_{q_0}) = \tH(q_0) \leq F(q_0,p_0),
\ee
which, again, is the desired inequality~\eqref{e.c050906}.
The first inequality in~\eqref{e.c062602} follows from \eqref{e.c051302} and the second follows from \eqref{e.c051303} and $p_0 + \bar p \geq 0$. 

Thus, in all cases,  \eqref{e.c050906} follows.  This concludes the proof.
\end{proof}

\subsection{The convergence of $w^\eps$: the proof of \Cref{t.uvw}.\ref{i.convergence}-\ref{i.Wulff} and \Cref{t.strong}}

\begin{proof}
Let us first note that we need only prove that $w_* = w^*$ on $(0,\infty)\times \overline\H$.  Indeed, denoting their common value by $w$, then the locally uniform convergence of $u^\eps, v^\eps$ to $w$ follows then directly from the definitions of $w_*$ and $w^*$, while the fact that $w$ is a weak solution to~\eqref{e.HJ}-\eqref{e.HJ_F_0} and a strong solution to~\eqref{e.HJ}-\eqref{e.HJ_F} follows from \Cref{l.hr} and \Cref{thm:4.7}, respectively, in view of the fact that $w^* = w_* = w$.

We now prove that $w_* = w^*$.  First, we claim that for each $\tau>0$, we have
\begin{equation}\label{e.s.scp1}
    w^*(\tau,x,y) \leq w_*(0,x,y) \quad \text{ for all }(x,y) \in \overline\H.
\end{equation}
By \Cref{lem:2.4a} and \Cref{lem:2.4}, 
\be
	w^*(\tau,0,0) =0 \leq w_*(0,0,0).
\ee
For $(x,y)\neq (0,0)$, \Cref{l.upper_bound} and \Cref{lem:2.4} imply that
\be
	w^*(\tau,x,y)
		< +\infty
		= w_*(0,x,y).
\ee
Hence~\eqref{e.s.scp1} holds.

We can, thus, apply the comparison principle, detailed in \Cref{thm:scp}, to $w^*(\tau+\cdot, \cdot,\cdot)$ and $w_*$ to conclude that
\begin{equation}\label{e.s.scp2}
	w^*(\tau + t,x,y)
		\leq w_*(t,x,y)
		\quad \text{ for }(t,x,y) \in (0,\infty)\times \overline\H.
\end{equation}
Let us note that \Cref{thm:scp} requires that we are working with {\em strong} sub- and supersolutions, which we have thanks to \Cref{thm:4.7}.

Recall from~\eqref{e.c051304} that $w^* \in C^{\sfrac12}_{\rm loc}((0,\infty)\times \overline \H)$.  Then, for fixed $(t,x,y) \in (0,\infty)\times \overline \H$, we can take $\tau \to 0$ in~\eqref{e.s.scp2}.  We, thus, find
\be
	w^*(t,x,y)
		\leq w_*(t,x,y) \quad \text{ for }(t,x,y) \in (0,\infty)\times \overline \H.
\ee
Since, by construction, we have $w^* \geq w_*$, it follows that $w^*=w_*$. 
	
The next claim, \Cref{t.uvw}.\ref{i.scaling}, is immediate from \Cref{t.uvw}.\ref{i.convergence} and \Cref{l.hr}.\ref{i.rho^*}.

Thanks to \Cref{t.uvw}.\ref{i.convergence}-\ref{i.scaling}, $\mathcal{W} = \{(x,y):~w(1,x,y) = 0\}$ is well defined and $\rsig(t,x,y) = t\rho(\sfrac{x}{t},\sfrac{y}{t})$.
Next, we observe that $\mathcal{W}$ is star-shaped. 
	It suffices to show that $\rho(x,y) \geq \rho(\beta x, \beta y)$ for all $0<\beta<1$. Indeed, one can show that $\rho(x,y)$ and $\rho(\beta x, \beta y)$ forms a pair of solution and subsolution to the equation \eqref{e.sigma*}. Hence one again have $\rho(x,y) \geq \rho(\beta x, \beta y)$. 
This completes the proof. 
\end{proof}

\subsection{The spreading properties of $(U,V)$: the proof of \Cref{t.uvw}.\ref{i.spreading}}

\Cref{t.uvw}.\ref{i.spreading} is made up of two equations~\eqref{e.conv_to_0} and~\eqref{e.conv_to_1}, which we prove one at a time.

\begin{proof}[Proof of~\eqref{e.conv_to_0}]
	We first show that the claim holds for all $(x,y)$ far from the origin: there exists $R>1$ such that
\begin{equation}\label{e.largeR}
    \lim_{t\to\infty} \sup_{ \frac{1}{t}(x,y) \notin B_{R+1}} V(t,x,y)
        =0
    \qquad\text{ and }\qquad
    \lim_{t\to\infty} \sup_{\frac{1}{t}(x,0) \notin B_{R+1}} U(t,x)
    		= 0.
\end{equation}
Indeed, for any fixed $\vt \in [-\sfrac\pi2, \sfrac\pi2]$, let
\be
	\overline{V}_\vt(t,x,y)
		= \exp\left\{ -(x,y)\cdot(\sin \vt, \cos \vt) +Rt\right\}
	\quad \text{ and } \quad
	\overline{U}_\vt(t,x)
		= \frac{\nu}{\mu} \exp\left\{ -x\sin \vt +Rt\right\}.
\ee
A straightforward computation shows that, for arbitrarily fixed $R \geq \max\{2, D\}$, $(\overline{U}_\vt,\overline{V}_\vt)$ is a supersolution to the unscaled problem~\eqref{e.berestycki}. Up to multiplying $(\overline{U}_\vt,\overline{V}_\vt)$ by a large constant $C>1$ so that $U \leq C\overline U$ and $V \leq C \overline V$ at $t=0$, one can then conclude by applying the comparison principle (see \cite[Proposition 4.3]{Berestycki2016shape}) and performing the scaling in $\eps$.    We omit the details as they are entirely straightforward.

We now consider the case when $(x,y)$ may be ``near'' the origin.  We claim that, for all $\eta>0$,
\begin{equation}\label{e.d.2}
        \lim_{t\to+\infty} \sup_{\frac{1}{t}(x,y) \in \overline B_{R+2}  \atop {\rm dist}(\frac{1}{t}(x,y), \mathcal{W} ) >\eta} V(t,x,y) = 0
        \qquad\text{ and }\qquad
        \lim_{t\to+\infty} \sup_{\frac{1}{t}(x,0) \in \overline B_{R+2}  \atop {\rm dist}(\frac{1}{t}(x,0), \mathcal{W} ) >\eta} U(t,x) = 0.
\end{equation}
Notice that the proof of~\eqref{e.conv_to_0} follows directly from the combination of~\eqref{e.largeR} and~\eqref{e.d.2}.

To this end, we first show how~\eqref{e.d.2} follows from the claim that
\be\label{e.c051305}
	\delta_\eta
		:=\lim_{\eps\to0}  \min\{\min_{(x,0) \in K_\eta}u^\eps(1,x),~ \min_{(x,y) \in K_\eta} v^\eps(1,x,y)\}
		> 0,
\ee
where
\be
	K_\eta =\{(1,x',y'):~ (x',y') \in \overline B_{R+2},  {\rm dist}((x',y'), \mathcal{W} ) \geq \eta\}.
\ee
Indeed, assuming~\eqref{e.c051305}, we find
\be
	\begin{split}
		\lim_{t\to+\infty} \sup_{\frac{1}{t}(x,y) \in \overline B_{R+2}  \atop {\rm dist}(\frac{1}{t}(x,y), \mathcal{W} ) >\eta} V(t,x,y)
			&= \lim_{\eps \to 0} \sup_{\eps (x,y) \in \overline B_{R+2}  \atop {\rm dist}(\eps(x,y), \mathcal{W} ) >\eta} V(\sfrac1\eps,x,y)
			\\&
			= \lim_{\eps \to 0} \sup_{(x',y') \in K_\eta} V(\sfrac1\eps,\sfrac{x'}{\eps},\sfrac{y'}{\eps})
			\\&
			= \lim_{\eps \to 0} \sup_{(x',y') \in K_\eta} e^{-\frac{1}{\eps} v^\eps(x',y')}
			\leq \liminf_{\eps \to 0} e^{-\frac{1}{\eps} \delta_\eta}
			= 0,
	\end{split}
\ee
and a similar statement for $U$ to hold,   
which is precisely~\eqref{e.d.2}.

We now prove~\eqref{e.c051305} to finish.  If it were not true, then there must be a sequence $\eps_n \to 0$ and $(x_n,y_n) \in K_\eta$ such that either $v^{\eps_n}(1,x_n,y_n) \to 0$ or $u^{\eps_n}(1,x_n) \to 0$.  By compactness, we may, up to passing to a subsequence whose renumbering we omit, assume that $(x_n,y_n)$ converges to a point $(x_\infty, y_\infty) \in K_\eta$.  By the locally uniform convergence given by \Cref{t.uvw}.\ref{i.convergence}, it follows that $w(1,x_\infty,y_\infty) = 0$.  This implies that $(x_\infty, y_\infty) \in \cW$, which contradicts the fact that $(x_\infty,y_\infty) \in K_\eta$.  Hence,~\eqref{e.c051305} must hold.  The proof is complete.
\end{proof}

The proof of the lower bound of $(V,U)$ in the interior of $\cW$ is significantly more involved.  The bound for $U$ follows from that of $V$, so we separate their proofs.

\begin{proof}[Proof of the bound on $V$ in~\eqref{e.conv_to_1}]
Fix any sequence $(\bar t_n, \bar x_n, \bar y_n)$ such that $\bar t_n \to \infty$, $\bar y_n > 0$, and
\be
	\dist\left(\frac1{\bar t_n} (\bar x_n, \bar y_n), \H\setminus \cW\right)
		\geq \eta.
\ee  
It is enough to show that 
\be
	V(\bar t_n, \bar x_n, \bar y_n) \to 1.
\ee
Let us note that, by~\eqref{e.c052403}, we need only show that
\be\label{e.V_geq_1}
	\liminf_{n\to\infty} V(\bar t_n, \bar x_n, \bar y_n) \geq 1.
\ee
Up to passing to a subsequence whose renumbering we suppress, there is $(\bar x, \bar y)$ such that $(\sfrac{\bar x_n}{\bar t_n}, \sfrac{\bar y_n}{\bar t_n}) \to (\bar x, \bar y)$ and
\be\label{e.c051401}
	\dist\left((\bar x, \bar y), \H\setminus \cW\right)
		\geq \eta.
\ee
Let us fix $n$ sufficiently large so that
\be
	|(\bar x, \bar y) - (\sfrac{\bar x_n}{\bar t_n}, \sfrac{\bar y_n}{\bar t_n})|
		\leq \frac{\eta}{2}.
\ee
There are two cases to consider depending on $\bar y$.

\medskip
\noindent
{\bf \# Case one: $\bar y>0$.}  Define the test functions
\be\label{e.c051501}
	\phi_n(x,y) = |t-1|^2 + |x-\sfrac{\bar x_n}{\bar t_n}|^2 + |y - \sfrac{\bar y_n}{\bar t_n}|^2.
\ee
Using~\eqref{e.c051401}, $(\bar x, \bar u) \in {\rm Int}\, \mathcal{W}$. By the definition of $\cW$, it follows that $w - \phi_n$ attains a maximum value of $0$ at  the point $(1,\sfrac{\bar x_n}{\bar t_n}, \sfrac{\bar y_n}{\bar t_n})$ on $B_{\tilde \eta}(1,\bar x, \bar y)$ for some $\tilde \eta$ that is uniform in $n$ sufficiently large.  From \Cref{t.uvw}.\ref{i.convergence}, it follows that, for $n$ sufficiently large,
\be\label{e.c051404}
	v^{\sfrac1{\bar t_n}} - \phi_n  \qquad (\text{where } \quad v^{\sfrac1{\bar t_n}} = v^{\ep} \big|_{\ep = \sfrac1{\bar t_n}})
\ee
has an interior maximum in $B_{\tilde \eta}(1, \bar x, \bar y)$ at some point $(t_n,x_n,y_n)$ such that
\be\label{e.c051402}
	(t_n, x_n,y_n) \to (1,\bar x, \bar y)
		\qquad\text{ as } n\to\infty.
\ee
Clearly, for $n$ sufficiently large, $y_n > 0$.  Using~\eqref{e.scaled_eqn}, we find
\be\label{e.c051403}
	\partial_t \phi_n - \frac{1}{\bar t_n} \Delta \phi_n + |\nabla \phi|^2 \leq V(\cdot\ \bar t_n, \cdot\ \bar t_n, \cdot\ \bar t_n) - 1
		\qquad\text{ at $(t_n,x_n,y_n)$.}
\ee
Let us note that, by explicit computation and the convergence~\eqref{e.c051402}, the left hand side of~\eqref{e.c051403} tends to zero as $n\to\infty$.  Hence,
\be\label{e.c051406}
	1
		\leq \liminf_{n \to \infty} V(t_n \bar t_n, x_n \bar t_n, y_n \bar t_n).
\ee

On the other hand, we have
\be\label{e.c051405}
	\begin{split}
		V(t_n \bar t_n, x_n \bar t_n , y_n \bar t_n)
			&= \exp\left\{- \bar t_n v^{\sfrac{1}{\bar t_n}}(t_n, x_n, y_n)\right\}
			\\&
			\leq \exp\left\{- \bar t_n
					\left(
						v^{\sfrac{1}{\bar t_n}}
						(1, \sfrac{\bar x_n}{\bar t_n}, \sfrac{\bar y_n}{\bar t_n})
						- \phi_n (1, \sfrac{\bar x_n}{\bar t_n}, \sfrac{\bar y_n}{\bar t_n})
						+ \phi_n (t_n, x_n, y_n)
					\right)
					\right\}
			\\&
			= V(\bar t_n, \bar x_n, \bar y_n) \exp\left\{\bar t_n
					\left(
						- \phi_n (t_n, x_n, y_n)
					\right)
					\right\}
			\leq V(\bar t_n, \bar x_n, \bar y_n),
	\end{split}
\ee
since $(x_n,y_n)$ is the location of the maximum of~\eqref{e.c051404}.  Here we also used the form of $\phi_n$ to deduce that $\phi_n (1, \sfrac{\bar x_n}{\bar t_n}, \sfrac{\bar y_n}{\bar t_n}) = 0$.

Putting together~\eqref{e.c051405} and~\eqref{e.c051403}, we deduce that
\be
	1
		\leq \liminf_{n\to\infty} V(\bar t_n, \bar x_n, \bar y_n).
\ee
This is precisely~\eqref{e.V_geq_1}, which completes the proof in this case.

\medskip
\noindent
{\bf \# Case two: $\bar y=0$.}
Fix any (small) $\delta>0$, and define the test function
\be
    \Phi_\delta(y)=1+ \frac{\delta y}{1+y}.
\ee
It is easy to observe that
\begin{equation}\label{e.c051505}
	\frac{\Phi_\delta'(0)}{\Phi_\delta(0)} = \delta,
	\qquad 
	\Phi_\delta(0)=1,
	\qquad \text{ and }\qquad
	-\Phi_\delta'' \leq 2\Phi_\delta \quad \text{ for }y \geq 0.
\end{equation}
As in the previous case, we may find a sequence $(t_n,x_n,y_n) \in B_{\tilde \eta}(1,\sfrac{\bar x_n}{\bar t_n}, \sfrac{\bar y_n}{\bar t_n})$ that are locations of maxima of
\be\label{e.c051502}
	\max\left\{
		v^{\sfrac{1}{\bar t_n}} + \frac{1}{\bar t_n} \log \Phi_\delta\left( \bar t_n y\right),
		u^{\sfrac{1}{\bar t_n}} + \frac{1}{\bar t_n} \log \left( \frac{\nu}{\mu} - \frac{\delta}{2\kappa\mu}\right)
		\right\}
		- \phi_n
\ee
Here, $\phi_n$ is as in~\eqref{e.c051501}.  Using the form of $\phi_n$ and the fact that $w(1,\cdot,\cdot) \equiv 0$ on $B_\eta(\bar x, \bar y)$, it follows that
\be\label{e.c051503}
	(t_n,x_n,y_n) \to (1,\bar x,\bar y).
\ee
If $\bar y >0$, then repeat the proof of case one. Henceforth suppose $\bar y = 0$ in \eqref{e.c051503}, i.e. $y_n \to 0$.
There are exactly three remaining subcases to consider here.

{\bf Subcase one: $y_n=0$ for infinitely many $n$ and the maximum is attained by the $u$ term in~\eqref{e.c051502}.}
Momentarily, let
\be
	\tilde \phi_n(t,x,y)
		= \phi_n(t,x,y) - \frac{1}{\bar t_n} \log\Phi_\delta(\bar t_n y).
\ee
It is clear that all $t$ and $x$ derivatives of $\tilde \phi_n$ tend to zero at $(t_n,x_n,y_n)$ as $n\to\infty$ due to the convergence~\eqref{e.c051503}.   Thus, we use~\eqref{e.scaled_eqn} to find, at $(t_n,x_n,y_n)$,
\be
	\begin{split}
		0
			&\geq \partial_t \tilde \phi_n
				- \frac{1}{\bar t_n} D \partial_{xx} \tilde \phi_n
				+ D |\partial_x \tilde \phi_n|^2
				+ \nu e^{\bar t_n \left(u^{\sfrac{1}{\bar t_n}} - v^{\sfrac{1}{\bar t_n}}\right)} - \mu
			\\&
			\geq 
				o(1)
				+ \nu e^{\bar t_n \left(u^{\sfrac{1}{\bar t_n}} - v^{\sfrac{1}{\bar t_n}}\right)} - \mu
			\\&
			\geq o(1)
				+ \nu \left(\frac{\mu}{\nu - \sfrac{\delta}{2\kappa}}\right) - \mu
			= o(1)
				+ \mu \left(\frac{\nu}{\nu - \sfrac{\delta}{2\kappa}} - 1\right)
			> 0,
	\end{split}
\ee
where the third inequality follows from the fact that, by assumption, at $(t_n, x_n, 0)$,
\be
	v^{\sfrac{1}{\bar t_n}} + \frac{1}{\bar t_n} \underbrace{\log \Phi_\delta\left(0\right)}_{=0}
		\leq 
		u^{\sfrac{1}{\bar t_n}} + \frac{1}{\bar t_n} \log \left( \frac{\nu}{\mu} - \frac{\delta}{2\kappa \mu}\right)
\ee
and the last inequality follows by taking $n$ sufficiently large. This is a contradiction.  Hence, this subcase cannot occur.

{\bf Subcase two: $y_n=0$ for infinitely many $n$ and the maximum is attained by the $v$ term in~\eqref{e.c051502}.}
We argue using the equation for the boundary condition of $v$ in~\eqref{e.scaled_eqn}.  
Indeed, we find, at $(t_n,x_n,0)$,
\be
	\begin{split}
		-\frac{2 \bar y_n}{\bar t_n} - \delta
			= \partial_y \tilde\phi_n
			&\geq  \kappa \left(\mu e^{\bar t_n \left(v^{\sfrac{1}{\bar t_n}}-u^{\sfrac{1}{\bar t_n}}\right)} - \nu\right)
			\\&
			\geq \kappa \left( \nu - \frac{\delta}{2\kappa} - \nu\right)
			= -\frac{\delta}{2},
	\end{split}
\ee
where, in the second inequality, we used that, by assumption
\be
	v^{\sfrac{1}{\bar t_n}} + \underbrace{\frac{1}{\bar t_n} \log \Phi_\delta\left(0\right)}_{=0}
		\geq 	u^{\sfrac{1}{\bar t_n}}
				+ \frac{1}{\bar t_n} \log \left( \frac{\nu}{\mu} - \frac{\delta}{2\kappa \mu}\right).
\ee
This is clearly a contradiction when $n$ is sufficiently large.  Hence, this subcase cannot occur.

{\bf Subcase three: $y_n>0$ for all $n$ sufficiently large.}
We argue similarly as in the first case, except using the equation for the boundary condition of $v$ in~\eqref{e.scaled_eqn} in place of the equation for $u$.  It is clear that all derivatives of $\phi_n$ tend to zero at $(t_n,x_n,y_n)$ as $n\to\infty$ due to the convergence~\eqref{e.c051503}.

Thus, we find, at $(t_n,x_n,y_n)$,
\be
	\begin{split}
		0
			&\geq \partial_t \tilde \phi_n 
				- \frac{1}{\bar t_n} \Delta \tilde \phi_n
				+ |\nabla \tilde \phi_n|^2
				+ 1 - V(\cdot \ \bar t_n,\cdot \ \bar t_n,\cdot \ \bar t_n)
			\\&
			\geq o(1)
				- \left( - \frac{\Phi_\delta''(\bar t_n y_n)}{\Phi(\bar t_n y_n)} + \frac{\Phi_\delta'(\bar t_n y_n)^2}{\Phi_\delta(\bar t_n y_n)^2}\right)
				+ \left|o(1) + \frac{\Phi'_\delta(\bar t_n y_n)}{\Phi_\delta(\bar t_n y_n)}\right|^2
				+ 1 - V(\cdot \ \bar t_n,\cdot \ \bar t_n,\cdot \ \bar t_n)
			\\&
			\geq o(1) - 2\delta + 1 - V(\cdot \ \bar t_n,\cdot \ \bar t_n,\cdot \ \bar t_n).
	\end{split}
\ee
In the last inequality, we used~\eqref{e.c051505}.  Thus, we have
\be\label{e.c051506}
	\liminf_{n\to\infty} V(t_n \bar t_n, x_n \bar t_n, y_n \bar t_n)
		\geq 1 - 2\delta.
\ee
Next, we use that $(t_n,x_n,y_n)$ is the location of a maximum of~\eqref{e.c051502}.  Using this and then the positivity of $\log \Phi_\delta$, we find
\be
	\begin{split}
		v^{\sfrac{1}{\bar t_n}}(t_n,x_n,y_n)
			&\geq v^{\sfrac{1}{\bar t_n}}(1, \sfrac{\bar x_n}{\bar t_n}, \sfrac{\bar y_n}{\bar t_n})
				- \frac{1}{\bar t_n} \log \Phi_\delta(\bar y_n)
				+ \frac{1}{\bar t_n} \log \Phi_\delta(\bar t_n y_n)
			\\&
			\geq v^{\sfrac{1}{\bar t_n}}(1, \sfrac{\bar x_n}{\bar t_n}, \sfrac{\bar y_n}{\bar t_n})
				- \frac{1}{\bar t_n} \log \Phi_\delta(\bar y_n),
	\end{split}
\ee
which, after undoing the Hopf-Cole transform, yields
\be
	V(t_n \bar t_n, x_n \bar t_n, y_n \bar t_n)
		\leq V(\bar t_n, \bar x_n, \bar y_n) \left(1+ \frac{\delta \bar y_n}{1+ \bar y_n} \right)
		\leq V(\bar t_n, \bar x_n, \bar y_n) \left(1+ \delta\right).
\ee
Combining this with~\eqref{e.c051506} and using that $\bar y_n \to 0$, by assumption of the case two, we deduce
\be
	\liminf_{n\to\infty} V(\bar t_n, \bar x_n, \bar y_n)
		\geq \frac{1 - 2\delta}{1+\delta}.
\ee
Letting $\delta \to 0$, this is precisely~\eqref{e.V_geq_1}, finishing the proof of case two and, thus, the lower bound of $V$.
\end{proof}

We now use the lower bound on $V$ to deduce the lower bound on $U$.
\begin{proof}[Proof of the bound on $U$ in~\eqref{e.conv_to_1}]
We argue by contradiction.  Suppose there is a sequence $(\bar t_n, \bar x_n)$ such that $\bar t_n \to \infty$, and
\be
	\dist\left(\frac1{\bar t_n} (\bar x_n, 0), \H\setminus \cW\right)
		\geq \eta,
\ee
but
\be\label{e.c051507}
	\lim_{n\to\infty} U(\bar t_n, \bar x_n, \bar y_n) < \frac{\mu}{\nu}.
\ee
Let
\be
	U_n(t,x) = U(t + \bar t_n, x + \bar x_n)
	\quad\text{ and }\quad
	V_n(t,x) =V(t + \bar t_n, x+ \bar x_n, y).
\ee
By parabolic regularity theory and using~\eqref{e.c051507}, we can pass to a sequence so that 
\begin{equation}\label{e.sssu1}
	\lim_{j \to \infty} U_j(0,0)
		= \lim_{j \to \infty} U^{\ep_j}(t_j,x_j)
		< \tfrac{\nu}{\mu},
	\quad \text{ and }\quad
	(U_j,V_j) \to (U_\infty,1)
		\quad \text{ in } C^{1,2}_{\rm loc},
\end{equation}
where $U_\infty$ is a solution to 
\begin{equation}\label{e.sus}
	\partial_t U_\infty - D \partial_{xx} U_\infty
		= \nu - \mu U_\infty
		\quad \text{ in } \mathbb{R}^2
	\qquad\text{ and } \qquad
	0 \leq U_\infty  \leq \tfrac{\nu}{\mu}.
\end{equation}
In~\eqref{e.sus}, we used that, for all fixed $(t,x)$ and large enough $n$,
\be
	\dist\left(\frac{1}{t + \bar t_n}(\bar x_n, 0), \H \setminus \cW\right)
		\geq \frac{\eta}{2},
\ee
so that the previous proof implies that
\be
	V_n(t,x)
		= V(t+ \bar t_n, x + \bar x_n, 0)
		\to 1.
\ee
Noticing that, for any $t_0>0$,
\be
	\underline{U} = \tfrac{\nu}{\mu} (1- e^{-\mu(t+t_0)})
\ee
is a subsolution to~\eqref{e.sus}. Since $\underline{U}(t_0,\cdot) =0 \leq U_\infty(t_0,\cdot)$, we may apply the comparison principle to deduce that
\be
	\frac{\mu}{\nu}
		> U_\infty (0,0)
		\geq \underline U(0,0)
		= \frac{\nu}{\mu} \left(1- e^{-\mu t_0}\right).
\ee
Taking $t_0 \to \infty$, we obtain a contradiction.  This concludes the proof.
\end{proof}

\section{Using the control problem to compute $w$: $w = \max\{J,0\}$} \label{sec:w=J+}

In this section, we prove \Cref{t.w=J+}, namely, $w = \max\{J,0\}.$  Our approach is to use Freidlin's condition. It first appeared in~\cite[Section~2]{Freidlin(N)} as ``condition (N).''  Roughly, it says that the optimal paths in~\eqref{e.Jcon} that lead to the front at time $t$ remain at or beyond the front at all intermediate times $s \leq t$.  Let us note that Freidlin's context was a bit different from ours being boundary-less; nonetheless, the proof may be easily adapted.

First, we make some simplifications.  By the scaling invariance~\eqref{e.scaling}, 
we need only consider $t=1$ and $\gamma \in H^1(0,1)$.  Next, we define some notation.  Let the ``action'' of the Lagrangian $\hat L$ be
\be\label{e.c052901}
	\cA(\gamma) := \int_0^1 \hat L(\gamma(s), \dot \gamma(s)) ds.
\ee
Let us now define Freidlin's condition.

\begin{definition}[Freidlin's condition]
Let $J$ be given in \eqref{e.Jcon} and define
\be\label{e.P}
	\cP = \{(t,x,y) \in (0,\infty)\times \overline\H:~ J(t,x,y)>0\}.
\ee
We say that $J$ satisfies (F) if, for each $(t,x,y) \in \partial P$, 
\be\tag{\bf F}
	\begin{split}
		J(t,x,y)
			= \inf \left\{ \cA(\gamma) : 
			         \gamma \in N(t,x,y) \text{ and } (s,\gamma(s)) \in P~\text{ for }s\in(0,t)\right\}. 
	\end{split}
\ee
\end{definition}

\begin{lemma}\label{l.(F)}
The value function $J$ satisfies condition (F).
\end{lemma}
\begin{proof}
Without loss of generality, we take $t=1$.  Additionally, we note that it is enough to show that the minimizer $\gamma$ satisfies $(s,\gamma(s)) \in \bar {\cP}$ for all $s \in [0,1]$.  Indeed, by replacing any (nontrivial) $\gamma$ by
\be
	\gamma_\eps(s)
		= \begin{cases}
			\gamma\left( s(1+\eps) \right)
				\qquad&\text{ if } s \leq \frac{1}{1+\eps},
			\\
			\gamma(1)
				\qquad&\text{ if } s \geq \frac{1}{1+\eps},
		\end{cases}
\ee
we see that
\be
	\lim_{\eps\searrow 0} \cA(\gamma_\eps)
		= \cA(\gamma)
	\qquad\text{ and, for all $\eps>0$, }\quad
	\cA(\gamma) < \cA(\gamma_\eps).
\ee
We consider two cases along the lines of \Cref{p.no_Huygens}: (i) $y=0$ or $x^2 + y^2 = 4$; (ii) $x^2 + y^2 > 4$ and $y>0$.

In the case (i), $\gamma(s) = s(x,y)$ and $\dot\gamma(s) = (x,y)$. Then $J(1,x,y) = 0$ and the form of $\hat{L}$ in \eqref{e.hatL} implies $\hat{L}(\gamma(s),\dot\gamma(s))\equiv 0$.
A direct computation using \Cref{lem:bellman} gives
\be
	J(s,\gamma(s))
		= s J(1,x,y)
		= 0.
\ee
Hence, $(s,\gamma(s)) \in \bar {\cP}$.  This completes the proof of this case.

In case (ii), we apply \Cref{p.no_Huygens} again to find
\be
	\gamma(s)
		= \begin{cases}
			(\crr,0) s
				\qquad&\text{ if } s \leq \tau_0,\\
			(\crr,0) \tau_0 + \cf(s-\tau_0)
				\qquad&\text{ if } s \geq \tau_0
		\end{cases}
\ee
with $\crr > c_*(\sfrac\pi2)$ and $|\cf| < 2$.  It follows easily that when $s\leq \tau_0$,
\be
	J(s,\gamma(s))
		= s \Lr(\crr)
		> 0.
\ee
When $s \geq \tau_0$,
\be
	J(s,\gamma(s))
		= (s- \tau_0) \Lf(\cf)
			+ \tau_0 \Lr(\crr),
\ee
and we see that this is strictly decreasing in $s \in [\tau_0,1]$ because $|\cf| < 2$.  Since $J(1,\gamma(1)) = 0$, it follows that, when $s \in [\tau_0, 1)$,
\be
	J(s,\gamma(s)) > 0.
\ee
We conclude that $(s,\gamma(s)) \in \cP$ for all $s\in (0,1)$. 
This completes the proof.
\end{proof}

\subsection{The proof of \Cref{t.w=J+}}

We now use Freidlin's condition to  show that $w = \max\{J,0\}$.

\begin{proof}[Proof of \Cref{t.w=J+}]
Let $I(t,x,y) = \max\{J(t,x,y),0\}$. It is easy to see that $I$ satisfies the continuity requirements of the uniqueness result \Cref{cor:scp}.  As a result, it suffices to verify that $I$ is a strong solution to \eqref{e.HJ}-\eqref{e.HJ_F}. First, we verify that it is a strong subsolution. Indeed, in $\{I \leq 0\}$, there is nothing to prove. In $\{I=J\}=\cP = \{J>0\}$,  $J$ satisfies
\be
	\begin{cases}
		J_t + \Hf(\nabla J) = 0
			\qquad &\text{ in } \cP\cap \{y>0\},
		\\
		J_t + F(\nabla J) = 0
			\qquad &\text{ on } \cP \cap \{y = 0\}.
	\end{cases}
\ee
See, e.g. \cite[Theorem 2.9]{Barles2018flux} for the corresponding properties of $J$. 
It follows easily that $I$ is a strong subsolution, as claimed.

Next, we verify that $I$ is a strong supersolution to \eqref{e.HJ}-\eqref{e.HJ_F}. Obviously, $I\geq 0$. Suppose $I - \varphi$ has a local minimum at $(t_0,x_0,y_0)$. If $(t_0,x_0,y_0) \in \cP$, then $I = J$ in a neighborhood of $(t_0,x_0,y_0)$ and this again follows from \cite[Theorem 6.4]{Imbert2017flux}. If $(t_0,x_0,y_0) \in {\rm Int}\,\{(t,x,y): J \leq 0\}$, then $I \equiv 0$ in a neighborhood of $(t_0,x_0,y_0)$, so that the supersolution property follows from the fact that $H(0,0) >0$ and $F(0,0) >0$.

It remains to consider the case $(t_0,x_0,y_0) \in \partial \cP$. Using the dynamic programming principle (\Cref{lem:bellman}), we find
\be
	J(t_0,x_0,y_0)
		= \inf_{\gamma \in N(t_0,x_0,y_0)}
			\left\{ 
				\int_{\tau}^{t_0}
				\hat{L}(\gamma(s),\dot\gamma(s))\,ds
				+ J(\tau,\gamma(\tau))\right\},
\ee
where we recall the definition~\eqref{e.N} of $N$. 
Since $J$ satisfies condition (F), it follows that at the minimizing path $\gamma$, $I(s,\gamma(s)) = J(s,\gamma(s))$ for $s \in [0,t_0]$. Based on this observation, and that $I \geq J$, it follows that
\be
	I(t_0,x_0,y_0)
		= \inf_{\gamma \in N(t_0,x_0,y_0)}\left\{ \int_{\tau}^{t_0} \hat{L}(\gamma(s),\dot\gamma(s))\,ds + I(\tau,\gamma(\tau))\right\}.
\ee
This equality implies by standard arguments (see \cite[Theorem 6.4]{Imbert2017flux}) that at the point $(t_0,x_0,y_0)$,
\be
	\begin{cases}
		\varphi_t + H(\nabla \varphi) \geq 0
				\qquad &\text{ if }y_0 >0,
		\\
		\varphi_t + F(\nabla \varphi) \geq 0
			\qquad &\text{ if }y_0 =0.
	\end{cases}
\ee
Hence, $I=\max\{J,0\}$ is a strong solution to \eqref{e.HJ}-\eqref{e.HJ_F}.  This completes the proof.
\end{proof}



\section{Conical domains}\label{sec:ext}

In this section, we discuss the extension of our result to conical domains and non-compactly supported initial data.  Nearly all results follow analogously, so we only briefly outline the main steps of the proof where changes are necessary.  For simplicity, we only consider the case
\be
    \mu = \tilde \mu,
    \quad
    \nu = \tilde \nu,
    \quad\text{and}\quad
    \kappa = \tilde \kappa.
\ee

\subsection{The case when diffusion is the same on both roads}

In \Cref{t.conical}, there are two main claims.  First, it states that $v^\eps, u^\eps$ converge  to a common limit $w_a$ in $C_{loc}$.  This follows exactly as in the nonconical case $a=\sfrac\pi2$.  Further, in the process, we can deduce that
\be\label{e.HJa}
    \min\{\partial_t w_a + \Hf(\nabla w_a),w_a)\} = 0
        \qquad\text{ in } \R_+\times \Omega_a
\ee
and that the boundary conditions
on $\R_+ \times \{0\} \times \R_+$
\begin{align}
    \min\{w_a, \partial_t w_a + F(\nabla w_a)\} = 0,
        \label{e.HJ_F_0a1}\\
    \min\{\tilde w_a,\partial_t \tilde w_a + F(\nabla \tilde w_a)\} = 0,
        \label{e.HJ_F_0a2}
\end{align}
are satisfied in the strong sense. 
(Note that the boundary condition at $(t,x,y) = (t,0,0)$ is not necessary thanks to \Cref{lem:2.4a}).  Recall, from~\eqref{e.Psi_tilde}, that $\tilde w_a(t,x,y) := w_a(t,\Psi(x,y))$.  We omit the proof of these facts. 

Second, it states that $w_a$ is given in terms of $J_a$.
 It is contained in the following:
\begin{proposition}\label{p.w_aDD}
Let $a \in (0,\sfrac\pi2]$, and suppose $(D,\kappa, \mu,\nu) = (\tilde D, \tilde \kappa, \tilde \mu, \tilde \nu)$. Then
\begin{equation}\label{e.waa}
    w_a(t,x,y)
        = \max\{J_a(t,x,y),0\} \quad \text{ for }t>0,~ (x,y)\in \Omega_a,
\end{equation}
where $J_a$ is defined by~\eqref{e.Jcon2}, and, moreover,
\be\label{e.c062603}
    J_a(t,x,y)
        = \min\{J(t,x,y), J(t,\Psi_a(x,y))\},
\ee
where $J: [0,\infty)\times \H$ is given by \eqref{e.Jcon}.
\end{proposition}
\begin{proof}
The equality~\eqref{e.c062603} is simple to check using \Cref{l.monotonicity}.\ref{i.grad_J} or simply from the variational problem \eqref{e.Jcon2}, so we omit the details.  We focus instead on~\eqref{e.waa}.  
Since the uniqueness of $w_a$ can be established in an analogous way as in \Cref{sec:A}, it remains to show that $\max\{J_a, 0\}$ is a weak solution to \eqref{e.HJa}-\eqref{e.HJ_F_0a1}-\eqref{e.HJ_F_0a2}.

First, notice that $J$ by the monotonicity of $\vt\mapsto J(t,\sin \vt, \cos vt)$ (\Cref{l.monotonicity}.\ref{i.grad_J}), it follows that $J_a = J$ and $\tilde J_a = \tilde J$ on the subdomain $\Omega_{\sfrac{a}2}$.  Hence, $J_a$ satisfies in viscosity sense
\be\label{e.HJ_Ja}
    \partial_t J_a + \Hf(\nabla J_a) = 0
        \qquad\text{ in } \R_+\times \Omega_a
\ee
away from the line $\Gamma_{\sfrac{a}2}$, 
and satisfies in the strong sense
\be\label{e.HJ_F_0_Ja}
    \begin{cases}
        \partial_t J_a + F(\nabla J_a) = 0\\
        \partial_t \tilde J_a + F(\nabla \tilde J_a) = 0
    \end{cases}
        \qquad\text{ on } \R_+\times\{0\}\times\R_+.
\ee
Additionally, it is Lipschitz continuous, so it satisfies \eqref{e.HJ_Ja} in the classical sense almost everywhere. One can then argue via a approximation argument and stability of subsolution that $J_a$ satisfies \eqref{e.HJ_Ja} in the viscosity sense everywhere in $\mathbb{R}_+\times \Omega_a$ (see, e.g.,~\cite[Chapter II, Proposition~5.1]{Bardi1997optimal} in a related context). Thus, $J_a$ is a strong subsolution of~\eqref{e.HJ_Ja}-\eqref{e.HJ_F_0_Ja}.

On the other hand, since $J_a$ is the minimum of two solutions, it is a strong supersolution to~\eqref{e.HJ_Ja}-\eqref{e.HJ_F_0_Ja}.  We deduce that $J_a$ is a strong solution to~\eqref{e.HJ_Ja}-\eqref{e.HJ_F_0_Ja}.  Finally, it follows by the arguments in \Cref{sec:w=J+} that $\max\{J_a, 0\}$ is a strong solution to~\eqref{e.HJa}-\eqref{e.HJ_F_0a1}-\eqref{e.HJ_F_0a2} in $\Omega_a$.  This completes the proof.
\end{proof}

\subsection{The case with unequal diffusion on the two roads}

Consider the case when the diffusion rate on the roads are different in~\eqref{e.eqn11}. Precisely, we assume that
\be 
    \tilde{D} > D > 2.
\ee 
because, in the case where $D \leq 2$, the road $\Gamma_0$ does not play a role in the propagation.  Because we wish to compare objects with different diffusivities, let us write, for any $\tilde D > D >2$,
    \be w_a^{(D,\tilde D)}\ee
to be the solution to~\eqref{e.HJa} and, on $\R_+ \times\{0\}\times \R_+$,
\begin{align}
    \min\{w_a^{(D,\tilde D)}, \partial_t w_a^{(D,\tilde D)} + F^{(D)}(\nabla w_a^{(D,\tilde D)})\} = 0,
        \label{e.HJ_Fb1}\\
    \min\{\tilde w_a^{(D,\tilde D)},\partial_t \tilde w_a^{(D,\tilde D)} + F^{(\tilde D)}(\nabla \tilde w_a^{(D,\tilde D)})\} = 0.
        \label{e.HJ_Fb2}
\end{align}
We have added the additional superscript to denote the dependence on $D$ and $\tilde D$, and $\tilde{w}_a^{(D,\tilde D)} = w_a^{(D,\tilde D)}(t,\Psi_a(x,y))$. The convergence of $v^\eps, u^\eps$ to $w_a^{(D,\tilde D)}$ follows analogously as in the previous case.  Finally, we update our notation similarly for
\be
    J_a^{(D,\tilde D)}.
\ee

The main complication here is that the analogue of~\eqref{e.waa}-\eqref{e.c062603} may not hold.  Indeed,~\eqref{e.waa}-\eqref{e.c062603} is  true because the optimal paths in $J_a$, defined as in~\eqref{e.Jcon2} with the two $\Lr$ terms there taking into account the respective diffusivity $D$ or $\tilde D$, will clearly only interact with the road closest to the endpoint.  When $1 \ll D \ll \tilde D$, it is not immediate to rule out that optimal paths near $\Gamma_0$ (the road associated to the smaller diffusivity $D$) will use the ``very fast'' road $\Gamma_a$ (the road associated to the larger diffusivity $\tilde D$) for some time before passing passing through the field to use the ``fast, but not as fast'' road $\Gamma_0$.  This is an interesting question worthy of future investigation.  Our goal, however, is to deduce some immediate results from our developed theory, so we do not pursue it further here.

We first immediately find a lower bound on $w_a$ via the comparison principle.
\begin{lemma}\label{lemma:4.4}
Let $(\kappa,\mu,\nu) = (\tilde\kappa,\tilde\mu,\tilde\nu)$ and let $D < \tilde D$.
    Let $a \in (0,\sfrac\pi2)$. Then
    \be
        w_a^{(D,\tilde D)}(t,x,y)
            \geq w_a^{(\tilde D, \tilde D)}
            = \max\{0,
                \min\{
                    J_a^{(\tilde D, \tilde D)}(t,x,y),
                    J_a^{(\tilde D, \tilde D)}(t, \Psi_a(x,y))
                    \}
                \}.
    \ee
\end{lemma}
\begin{proof}
The first inequality follows directly from the fact that all Hamiltonians are increasing in $\hat D$, meaning that $w^{(\tilde D, \tilde D)}$ is a subsolution to the equation satisfied by $w^{(D,\tilde D)}$.  The second equality is simply \Cref{p.w_aDD} with our updated notation.
\end{proof}

Next, we deduce that the speed of the faster road remains the same as in the half plane class, while the speed on the slower road can sometime be enhanced if the angle $a$ is small and $\tilde{D}$ is large.  Let us introduce the following notation: for $\hat D$, let \be
    c_*(\vt, \hat D)
\ee
be the directional spreading speed associated with $(D,\kappa,\mu,\nu)$ on $\H$, as in \eqref{e.directional}, and, for $\hat D_0$, $\hat D_a$, we let
\be
    c_{*a}(\vt,\hat D_0, \hat D_a)
\ee
be the directional spreading speed on $\Omega_a$ associated to diffusivity $\hat D_0$ on $\Gamma_0$ and diffusivity $\hat D_a$ on $\Gamma_a$.

\begin{proposition}\label{p.tilde_D > D}
Fix $a \in (0,\sfrac\pi4)$.  The spreading speed on the fast road $\Gamma_a$ is unchanged by the slow road:
\begin{equation}\label{e.c071801}
    c_{*a}(\sfrac\pi2-2a, D, \tilde D) = c_*(\sfrac\pi2, \tilde D).
\end{equation}
Furthermore, if $2\csc(2a) > c_*(\sfrac\pi2,D)$, there exists $\tilde{D}_{\min}>2$ such that 
\be\label{e.c071802}
    c_{*a}(\sfrac\pi2) > c_*(\sfrac\pi2, D)
        \qquad \text{ whenever }\tilde{D} > \tilde{D}_{min}.
\ee
\end{proposition}

Before beginning the proof, we make a note of the case $a\geq \sfrac\pi4$.  It is clear that~\eqref{e.c071801} will hold for $a\geq \sfrac\pi4$.  Intuitively this is because the slow road is even farther from the fast road $\Gamma_a$.  Our proof does not directly carry over for technical reasons, although we do not anticipate serious issues developing a new proof.  It appears that one should, roughly, define the supersolution not purely by $J^{(\tilde D)}$, but also using $\sfrac{(x^2 + y^2)}{4t} - t$ on a certain part of the domain.  We opt for simplicity and omit this case.

On the other hand, we believe that~\eqref{e.c071802} should not hold for any choice of $\tilde D$ and $D$ when $a\geq \sfrac\pi4$.  The reason for this is, roughly, that $\Gamma_0$ and $\Gamma_a$ are anti-aligned (the former runs horizontally to the right and the latter runs up and to the left; see \Cref{f.Omega_a}). Roughly speaking the presence of the fast road influence the speed on the slower road if and only if the angle $2a$ is strictly less than $\sfrac\pi2$ and $\tilde D$ is large enough.

\begin{proof}
Define
\begin{equation}
    \overline{w}
        =\max\{0,
            J^{(\tilde D)}(t,\Psi_a(x,y))
            \}.
\end{equation}
It suffices to prove that $\overline w$ is a strong supersolution to \eqref{e.HJa}-\eqref{e.HJ_Fb1}-\eqref{e.HJ_Fb2}.  Postponing this momentarily, we show how to conclude the proof with it assumed.  By comparison,
\be
    w^{(D,\tilde D)}_a \leq \overline{w}.
\ee
From this it follows that, for all $\vt \in [\sfrac\pi2,\sfrac\pi2-2a]$,
\be\label{e.c071803}
    c_{*a}(\vt, D, \tilde D) \geq c_*(\pi-2a-\vt, \tilde D).
\ee
From \Cref{lemma:4.4}, we have that
\be\label{e.c071804}
    c_{*a}(\vt,D,\tilde D)
        \leq c_{a*}(\vt,\tilde D, \tilde D)
        = \max\{c_*(\vt,\tilde D), c_*(\pi -2a-\vt, \tilde D)\}.
\ee
Specializing~\eqref{e.c071803}-\eqref{e.c071804} to the case $\vt = \sfrac\pi2-2a$, we deduce~\eqref{e.c071801}.

To understand~\eqref{e.c071802}, we take $\vt = \sfrac\pi2$ in \eqref{e.c071803} 
and use \Cref{prop:1.18}
to see that
\be
    \lim_{\tilde{D} \to +\infty} c_{*a}(\sfrac\pi2, D, \tilde D)
        \geq \lim_{\tilde{D} \to +\infty} c_*(\sfrac\pi2 - 2a, \tilde D)
        \geq 2\sec(\sfrac\pi2 - 2a)
        = 2\csc(2a) > c_*(\sfrac\pi2, D). 
\ee
This concludes the proof up to establishing that $\overline w$ is a supersolution, as claimed above.

It remains to prove that $\overline w$ is a strong supersolution to~\eqref{e.HJa}-\eqref{e.HJ_Fb1}-\eqref{e.HJ_Fb2}. There is nothing to check in the $\H$ and on $\Gamma_a$ since, there,
\be
    \overline w(t,x,y) = w^{(\tilde D)}(t, \Psi_a(x,y)),
\ee
which satisfies~\eqref{e.HJa}-\eqref{e.HJ_Fb2}.
Hence, it suffices to check the boundary condition on $\Gamma_0$. 

Fix any test function $\varphi$ and suppose that $\overline w - \varphi$ has a minimum at $(t,x,0)$ with $t>0$. 
There are two cases: (a) $\overline{w}(t,x,0)=0$; (b) $\overline{w}(t,x,0)>0$.   Case (a) is trivial.  Indeed, in this case, because $\overline w\geq 0$, we have that $\varphi$ has a minimum at $(t,x,0)$.  Thus, $\varphi_t(t,x,0) =\varphi_x(t,x,0)=0$. We deduce that
\be\label{e.c071805}
    \varphi_t + H_f^-(\nabla \varphi)
        \geq 1 >0.
\ee

In case (b), we use that $J^{(\tilde D)}\circ \Psi_a$ solves~\eqref{e.HJa} at the point $(t,\Psi_a(x,0))$ since $\Psi_a(x,0) \in \H$.  Hence,
\be
    \varphi_t + H_f(\nabla \varphi)
        \geq \hat{J}_t  + H_f(\nabla \hat J)
        =0.
\ee
However, $a\in (0,\sfrac\pi4)$ and \Cref{l.monotonicity}.\ref{i.grad_J} imply that $J^{(\tilde D)}(t,\Psi_a(x,y))$ is decreasing in $y$ at $y=0$.  It must be that $\varphi_y \leq 0$.  We deduce that, at $(t,x,0)$,
\be\label{e.c071806}
    \varphi_t + H_f^-(\nabla \varphi)
        =\varphi_t + H_f(\nabla \varphi)
        \geq 0.
\ee

Putting together~\eqref{e.c071805}-\eqref{e.c071806} and recalling that $F(p,q) = \max\{H_f^-(p,q), H_r(q)\}$, we have
\be
    \varphi_t + F(\nabla \varphi) \geq 0 \quad \text{ at }(t,x,0).
\ee
Thus, the proof is complete.
\end{proof}

\section{Proofs of technical lemmas}\label{s.technical}

\subsection{Lemmas related to the control formulation}\label{ss.control_lemmas}

\subsubsection{Convexity of $\tH$: \Cref{lem.tH}}

\begin{proof}
    For assertion (iii), we begin by noting that it is enough to establish that $p_q^2 = (g^{-1}(q))^2$ is convex as a function of $q$ when $q \geq \sqrt{\sfrac{1}{(D-1)}}$.
To this end, we notice that
\[
    \begin{split}
        \frac{d^2}{dx^2} (g^{-1}(q))^2
            &= 2\frac{d}{dx} \frac{g^{-1}(q)}{g'(g^{-1})(q))}
            = 2\left(
                    \frac{1}{g'(g^{-1})(q))^2}
                    - \frac{ g^{-1}(q) g''(g^{-1}(q))}{g'(g^{-1}(q)^3}
                \right)
            \\&
            = \frac{2}{g'(g^{-1}(q)^3}\left(1 - \frac{g^{-1}(q) g''(g^{-1}(q))}{g'(g^{-1}(q)}\right).
    \end{split}
\]
Above we used multiple times that
\[
    (g^{-1})'(q)
        = \frac{1}{g'(g^{-1}(q))}.
\]
We claim that
\begin{equation}
    \frac{g^{-1}(q) g''(g^{-1}(q))}{g'(g^{-1}(q)}
        \leq 1,
\end{equation}
which is enough to conclude the proof.  This is equivalent to establishing that $p g''(p) \leq g'(p)$ for $p \geq 0$.  By homogeneity, it suffices to prove this under the assumption that $D-1 = 1$, which simplifies the computations somewhat. 
 Hence, from~\eqref{eq:g}, we find
\[
    g'(p) = \frac{2p + \frac{\mu k \nu}{(\kappa\nu + p)^2}}{2 g}
    \quad\text{ and } \quad
    g''(p) = \frac{2 - 2\frac{\mu k \nu}{(\kappa\nu + p)^3}}{2 g}
        - \frac{\left(2p + \frac{\mu k \nu}{(\kappa\nu + p)^2}\right)^2}{4 g^3}.
\]
We immediately see that
\[
    p g''(p)
        \leq \frac{2p}{2g}
        \leq \frac{2p + \frac{\mu k \nu}{(\kappa\nu + p)^2}}{2 g}
        = g'(p),
\]
which concludes the proof.
\end{proof}

\subsubsection{The unique minimizing path follows straight lines: \Cref{p.straight_lines}}

\begin{proof}[Proof of \Cref{p.straight_lines}]
Observe that $\Hf \leq \tH$ by~\eqref{e.tH}.  Hence,
\be
	\Lr(v_1) \leq \Lf(v)
		\qquad\text{ for all } v.
\ee
Observe also that $\Lf$ and $\Lr$ are strictly convex because $\Hf$ and $\tH$ are strictly convex.  This is obvious for $\Hf$ and is a consequence of \Cref{lem.tH} for $\tH$.

Let us first consider the case where $y>0$.   We claim that any path $\gamma$ connecting $(0,0)$ and $(x,y)$ not of the form~\eqref{e.Lq3b} can be replaced by a path $\tilde \gamma$ of the form~\eqref{e.Lq3b} with strictly lower cost; that is,
\be\label{e.c052304}
	\cA(\tilde \gamma) < \cA(\gamma),
\ee
where $\cA$ is defined in~\eqref{e.c052901}.
Let us note that the compactness inherent in paths of the form~\eqref{e.Lq3b}, which, by~\eqref{e.c052304} are the optimal ``class'' of paths, yields the existence of a minimizer.

Because $y>0$, we can define
\be
	\tau_0
		:= \max\left\{s\in [0,1]: \gamma_2(s)=0\right\}
		< 1.
\ee
First, we claim that
\be
	\tilde \gamma(s)
		:=\begin{cases}
			(\gamma_1(s), 0)
				\quad &\text{ if } s \leq \tau_0,\\
			\gamma(s)
				\quad&\text{ if } s \geq \tau_0
		\end{cases}
\ee
is equal to $\gamma$ or it
has strictly lower cost than $\gamma$; that is,~\eqref{e.c052304} holds. 
We argue assuming that $\tilde \gamma \neq \gamma$.

Then there is (recalling the continuity of $\gamma$) an interval $(s_0, s_1) \subset (0,\tau_0)$ such that
	\be
		\gamma_2(s_0) = \gamma_2(s_1) = 0
		\qquad\text{ and }\qquad
		\gamma_2(s) > 0 \quad\text{ for all } s \in (s_0,s_1).
	\ee
We immediately see that
	\be
		\cA(\gamma) - \cA(\tilde \gamma)
			\geq \int_{s_0}^{s_1}
				\frac{\dot\gamma_2(s)^2}{4} ds
			> 0.
	\ee
	Hence,~\eqref{e.c052304} is established.
	
	Either $\tilde \gamma$ is of the form~\eqref{e.Lq3a} or we may let $\dbtilde \gamma$ be the modification of it to be of that form; that is
	\be
		\dbtilde \gamma(s)
			= \begin{cases}
				\frac{s}{\tau_0} \tilde \gamma(\tau_0)
					\qquad&\text{ if } s \leq \tau_0,\\
				\tilde \gamma(\tau_0) + \frac{s-\tau_0}{1-\tau_0}(\tilde \gamma(1) - \tilde\gamma(\tau_0))
					\qquad&\text{ if } s\geq \tau_0.
			\end{cases}
	\ee
	Using the strict convexity of $\Lr$ and $\Lf$, we see that
	\be 
		\begin{split}
			\cA(\tilde \gamma)
				& = \int_0^{\tau_0} \Lr\left(\dot{\tilde{\gamma}}\right) ds
					+ \int_{\tau_0}^1 \Lr\left(\dot{\tilde{\gamma}}\right) ds
				\\&
				\geq \tau_0 \Lr\Big(\frac{1}{\tau_0}\int_0^{\tau_0}\dot{\tilde{\gamma}}\Big) ds
					+ (1-\tau_0) \Lr\Big(\frac{1}{1-\tau_0} \int_{\tau_0}^1\dot{\tilde{\gamma}}\Big) ds
				= \cA(\dbtilde \gamma(s))
		\end{split}
	\ee
	and the above inequality is strict if $\tilde \gamma \neq \dbtilde \gamma$.  This completes the proof of our claim.
	
	Since any path may be replaced by one of the form~\eqref{e.Lq3b}, we see that the equality in~\eqref{e.Lq3a} holds.  Moreover, such paths are parametrized by two parameters, $z$ and $\tau$, living in a compact set.  The existence of a minimizer follows easily.  The uniqueness proof follows exactly along the lines of the proof of \Cref{l.W_convex}; that is, given two distinct minimizers, the suitable ``average'' of them as in~\eqref{e.c052401} will have lower cost, a contradiction.  This is standard, so we omit its proof.

Next, we analyze $\nabla J$. First, it follows from \eqref{e.Lq3a} that $J$ is locally Lipschitz and hence differentiable almost everywhere in $(0,\infty)\times \H$. Next, note that whenever $J$ is differentiable, it satisfies $J_t + H(\nabla J) = 0$ and that
 the dynamic programming principle (\Cref{lem:bellman}) implies that 
 \be
L_f(\dot\gamma(s)) = \frac{d}{ds} J(s,\gamma(s)) = J_t(s,\gamma(s)) +\dot\gamma(s) \cdot \nabla J(s,\gamma(s)) = 0 \quad \text{ for a.e. } s \in [0,t].
 \ee
 It follows that, at differentiable points $(t,x,y) = (t,\gamma(t))$,
 \be
H(\nabla J) = -J_t = \dot\gamma(t) \cdot \nabla J - L_f(\dot\gamma(t)).
 \ee
 It follows that $\nabla J$ and $\dot\gamma(t)$ are conjugate, i.e. 
 \be\label{e.a240722.12}
 \nabla J(t,x,y) = DL_f(\dot\gamma(t)) = \tfrac{1}{2}\dot\gamma(t).
 \ee
 Since this holds almost everywhere, and that $(x,y)\mapsto \dot\gamma(t)$  is continuous (since it is unique), it follows that $\nabla J$ is continuous and \eqref{e.a240722.12} is true everywhere. This proves \eqref{e.a0722.11}.

In the case where $y=0$, the above work clearly shows that $\tau_0 = 1$.  The remainder of the steps follow analogously.  The proof is complete. 
\end{proof}

\subsubsection{Fine properties of $J$: \Cref{lem:3.4}}
\begin{proof}
Suppose that the expression in the square bracket in \eqref{e.Lq3a} attains a global minimum at some $\tau_0 \in (0,t)$ and $z_0 \in (0,x)$. 
By differentiating the quantity in square brackets in~\eqref{e.Lq3a} in $z$, we deduce the optimality condition
\be
	0
		= - \partial_x\Lf\left(\frac{(x-z_0,y)}{1-\tau_0}\right) + \Lr'\left(\frac{z_0}{\tau_0}\right),
\ee
which yields
\be\label{e.c051603}
	q_0 = \frac{x-z_0}{2(1-\tau_0)} = \Lr'\left(\frac{z_0}{\tau_0}\right).
\ee
Similarly, differentiating in $\tau$ at the minimizer $\tau_0$, we find
\be\label{e.c051602}
	\begin{split}
		0
			&= - \Lf\left(\frac{(x-z_0,y)}{1-\tau_0}\right)
				+ \frac{(x-z_0,y)}{1-\tau_0} \cdot \nabla \Lf\left(\frac{(x-z_0,y)}{1-\tau_0}\right)
				+ \Lr\left(\frac{z_0}{\tau_0}\right)
				- \frac{z_0}{\tau_0} \Lr'\left( \frac{z_0}{\tau_0}\right)
			\\&
			= q_0^2 + p_0^2 + 1
				+ \Lr\left(\frac{z_0}{\tau_0}\right)
				- \frac{z_0}{\tau_0} \Lr'\left( \frac{z_0}{\tau_0}\right).
	\end{split}
\ee
Let us recall that, since $\Lr$ is the Legendre transform of $\tH$, i.e. $\Lr(v) = \max_{q} vq - \tH(q)$, we have
\be\label{e.c051604}
	\tH(q) + \Lr(v)
		= vq \quad \text{ provided that }
		q = \Lr'(v)
		~\text{ or }~
		v = \tH'(q).
\ee
Applying this in~\eqref{e.c051602} and recalling~\eqref{e.c051603} yields
\be\label{e.c051605}
	q_0^2 + p_0^2 + 1
		= \frac{z_0}{\tau_0} \Lr'\left( \frac{z_0}{\tau_0}\right)
				- \Lr\left(\frac{z_0}{\tau_0}\right)
		= \tH \left( \Lr'\left(\frac{z_0}{\tau_0}\right)\right)
		= \tH(q_0).
\ee
Using again~\eqref{e.c051604}, we find that $\sfrac{z_0}{\tau_0} = \tH'(q_0)$.  Then~\eqref{e.c051605} yields that $p_0 = p_{q_0}$ due to~\eqref{e.tH}-\eqref{e.pq}.  

Next, note that $y = (1-\tau_0) p_{q_0} >0$, so by the definition of $p_{q_0}$ in \eqref{e.pq}, we have $q_0 > \tfrac{1}{\sqrt{D-1}}.$

Finally, the form of $J$ in~\ref{i.tau_0>0b} follows from a direct computation using the above identities.  This completes the proof of~\ref{i.tau_0>0a} and \ref{i.tau_0>0b}. 
\end{proof}

\subsubsection{Relating $\Lr(v_0)$ and $\tH(q_0)$ when $q_0= \Lr'(v_0)$: \Cref{prop:D.5}}
\begin{proof}
Since both quantities are strictly convex and have zero derivative at $0$, they are strictly increasing away from the origin.  Hence, it is enough to check the boundary case $\Lr(v_0) = 0$.

Additionally, to avoid discussions of regularity, we investigate only the cases $D<2$ and $D>2$.  Indeed, in these cases, the critical $v_0$ such that $\Lr(v_0) = 2$ occurs in the interior of the two cases in the definition~\eqref{e.Hr} of $\tH$ where all quantities are smooth.  The case $D=2$ holds by continuity.

Let us first consider the case when $D < 2$, where we need to show $L(v_0) < 0$ implies $H(q_0) <2$.  
To this end, we claim that $\Lr(2) = 0$ and $\tH(\Lr'(2)) = 2$.  Let $v_0 = 2$ and $q_0 = 1$.  Indeed, in this case
\be
    q_0 < \frac{1}{\sqrt{D-1}}
\ee
so that $\tH(q) = q^2 - 1$ for $q$ near $q_0$.  Hence,
\be
     \tH'(q_0) = 2q_0 = v_0,
\ee
which implies that
\be
	\Lr(v_0) = v_0q_0 - \tH(q_0) = 0.
\ee
Additionally, $\tH(q_0) = 2$.  This concludes the proof in this case.

Next, consider the case $D>2$, where we need to show that $L(v_0)\leq 0$ implies $H(q_0) <2$.  In this case, the additional $g$ term in $\tH$ plays a role.  Fix $v_0$ such that
\be
	\Lr(v_0) = 0,
\ee
and let $q_0 = \Lr'(v_0)$.  Recall that, also, $v_0 = \tH'(q_0)$.  If 
\be
	q_0 \leq \frac{1}{\sqrt{D-1}}.
\ee
then
\be
	\tH(q)
		= q^2 + 1
		\leq \frac{1}{D-1} + 1
		< 2.
\ee
Hence, we consider the case where
\be\label{e.c052101}
	q_0 > \frac{1}{\sqrt{D-1}}.
\ee
We have
\be\label{e.c052102}
	\tH(q_0)
		= \tH(q_0) + \Lr(v_0)
		= q_0 v_0
		= q_0 \tH'(q_0).
\ee
Let us compute the right hand side and show that it is less than $2$.  Recall that $\tH(q) = q^2 + p_q^2 + 1$, where $p_{q_0} = g^{-1}(q)$ due to~\eqref{e.c052101}.  Then:
\be\label{e.c052103}
	\begin{split}
		q_0\tH'(q_0)
			&= q_0[2q_0 + 2g^{-1}(q_0) (g^{-1})'(q_0))]
			= 2\left[ q_0^2 + q_0g^{-1}(q_0) (g^{-1})'(q_0)) \right]
			\\&
			= 2\left[ \tH(q_0) - 1-p_{q_0}^2 + \frac{p_{q_0} g(p_{q_0})}{g'(p_{q_0})} \right]
			\\&
			= 2\left[ \tH(q_0) - 1+\frac{p_{q_0}}{g'(p_{q_0})}\Big(-p_{q_0}g'(p_{q_0}) + g(p_{q_0})\Big) \right].
	\end{split}
\ee
We claim that
\be\label{e.c052104}
	-p_{q_0}g'(p_{q_0}) + g(p_{q_0}) > 0.
\ee
If that were true, then~\eqref{e.c052102}-\eqref{e.c052103} implies that
\be
	\tH(q_0)
		> 2 \tH(q_0) - 2,
\ee
which, after rearrangement, yields the claim.

We now prove~\eqref{e.c052104}.  We do this by direct computation using that $g(p_{q_0}) = q_0 > 0$ and $p_{q_0}>0$ by~\eqref{e.c052101}:
\be
	\begin{split}
		g(p_{q_0}) \Big(g(p_{q_0}) - p_{q_0} g'(p_{q_0})\Big)
			&= \left(
					p_{q_0}^2 + 1 + \frac{\mu p}{\kappa\nu + p_{q_0}}
				\right)
				- \left(
					p_{q_0}^2
					+ \frac{1}{2} \frac{\mu p}{\kappa\nu + p_{q_0}} 
						\frac{\kappa\nu}{\kappa\nu + p_{q_0}}
				\right)
			>0.
	\end{split}
\ee
This concludes the proof.
\end{proof}

\subsection{Proofs of lemmas relating to $w_*$ and $w^*$}
\label{ss.w_lemmas}

\subsubsection{The upper bound on $u^\eps, v^\eps$: \Cref{l.upper_bound}}

\begin{proof}[Proof of \Cref{l.upper_bound}]
    Following~\cite[Lemma~2.1]{Evans1989pde}, it is easy to see that the desired bounds hold for $v^\eps$ away from the boundary.  Hence, we focus on obtaining a bound for $v^\eps$ and $u^\eps$ near the boundary.  To this end, we simply construct a supersolution and appeal to the (parabolic) comparison principle.

    First we obtain a bound on for $|x|, |y| \leq \sfrac{\eps}{2C_0}$.  Fix $\alpha, \beta > 0$ to be chosen.  Let
    \begin{equation}\label{e.c081003}
        \begin{split}
            &\bar v(t,x,y) = \alpha t - \eps\beta \log\Big( 1 - \frac{C_0^2 y^2}{\eps^2} \Big)
            - \eps\beta \log\Big( 1 - \frac{C_0^2 x^2}{\eps^2}  \Big)
            + \eps \gamma
            \qquad\text{ and}
            \\&
            \bar u(t,x) = \alpha t - \eps \beta \log\Big(1 - \frac{C_0^2 x^2}{\eps^2}  \Big).
        \end{split}
    \end{equation}
    Clearly $(u^\eps, v^\eps) < (\bar u, \bar v)$ at $t=0$ due to~\eqref{e.initial_data1} and~\eqref{e.c081003}.      Hence, we need only show that $(\bar u, \bar v)$ is a supersolution to~\eqref{e.scaled_eqn}.  Note that we use here the finiteness of $T$.
    
    One immediately sees that, for $|x|, |y| < \sfrac{C_0}{\eps}$,
    \begin{equation}
        \begin{split}
            \bar v_t - &\eps \Delta \bar v + |\nabla \bar v|^2 + (1 - e^{-\bar v/\eps})
                \\&
                \geq \alpha - \eps \left(\frac{\frac{2\beta C_0^2}{\eps}}{1 - \frac{C_0^2 y^2}{\eps^2}}
                +\frac{\frac{4\beta C_0^4 y^2}{\eps^3}}{(1 - \frac{C_0^2 y^2}{\eps^2})^2}
                + \frac{\frac{2\beta C_0^2 }{\eps}}{1 - \frac{C_0^2 x^2}{\eps^2}}
                +\frac{\frac{4\beta C_0^4 x^2}{\eps^3}}{(1 - \frac{C_0^2 x^2}{\eps^2})^2}
                \right)  
                \\&\qquad
                +\left|\frac{\frac{2\beta C_0^2 y}{\eps}}{1 - \frac{C_0^2 y^2}{\eps^2}}\right|^2
                + \left|\frac{\frac{2\beta C_0^2 x}{\eps}}{1 - \frac{C_0^2 x^2}{\eps^2}}\right|^2
                \\&
                = \alpha + \frac{4\beta(\beta-1) C_0^4 y^2}{\eps^2 (1 - \frac{C_0^2 y^2}{\eps^2})^2} 
                    + \frac{4\beta(\beta - 1) C_0^4 x^2}{\eps^2 (1 - \frac{C_0^2 x^2}{\eps^2})^2} 
                    - \frac{2 \beta C_0^2}{1 - \frac{C_0^2 y^2}{\eps^2}}
                    - \frac{2 \beta C_0^2}{1 - \frac{C_0^2 x^2}{\eps^2}}
        \end{split}
    \end{equation}
    When $x\geq \sfrac{\eps}{2C_0}$ the second and third terms dominate after increasing $\beta$, and the above is positive.  When $x \leq \sfrac{\eps}{2C_0}$, the first term dominates after increasing $\alpha$, and the above is positive.  Hence
    \begin{equation}\label{e.c081004}
        \bar v_t - \eps \Delta \bar v + |\nabla \bar v|^2 + (1 - e^{-\bar v/\eps})
            \geq 0
                \qquad\text{ in } (0,\infty) \times (\sfrac{-\eps}{C_0}, \sfrac{\eps}{C_0})\times (0, \sfrac{\eps}{C_0}).
    \end{equation}
    
    A similar computation shows that
    \begin{equation}\label{e.c081005}
        \bar u_t - \eps D \bar u_{xx}+ D|\bar u_x|^2 +  \nu e^{\frac{\bar u - \bar v}{\ep}} - \mu  = 0
            \qquad\text{ in } (0,\infty) \times (\sfrac{-\eps}{C_0}, \sfrac{\eps}{C_0}).
    \end{equation}
    after further increasing $\alpha$ and $\beta$, if necessary, depending on $D$ and $\mu$.

    Finally, notice that, when $y=0$,
    \begin{equation}\label{e.c081006}
        \bar v_y
            = 0
            < \kappa( \mu e^{\gamma} - \nu)
            = \kappa( \mu e^\frac{\bar v - \bar u}{\eps} - \nu),
    \end{equation}
    after increasing $\gamma$.  The combination of~\eqref{e.c081004},~\eqref{e.c081005}, and~\eqref{e.c081006}, along with the comparison principle, shows that $(u^\eps, v^\eps) < (\bar u, \bar v)$ for all $t>0$.  This concludes the proof on the set $(0,\infty) \times [\sfrac{-\eps}{2C_0}, \sfrac{\eps}{2C_0}]\times[0,\sfrac{\eps}{2C_0}]$.

    We now consider the complement of the above set; that is, $\max\{|x|,|y|\}\geq \sfrac{\eps}{2C_0}$.  For any time interval $[0,T]$, we again define supersolutions: for $\alpha, \beta, \gamma>0$ to be determined, let
    \begin{equation}\label{e.c081008}
        \begin{split}
            &
            \bar v(t,x,y)
                = \alpha t + \beta\frac{\sqrt{x^2 + (\sfrac{\eps}  {2C_0})^2} + y + x^2 + y^2}{t} + \eps \gamma
            \qquad\text{ and }\qquad
            \\&
            \bar u(t,x)
                = \alpha t + \beta\frac{|x| + x^2}{t} + \eps\gamma.
        \end{split}
    \end{equation}
    Up to increasing $\alpha, \beta, \gamma$, clearly $(u^\eps, v^\eps) < (\bar u, \bar v)$ on the parabolic boundary of our set; that is, at $t=0$ and $\max\{|x|,|y|\}> \sfrac{\eps}{2C_0}$ or when $t\in(0,T]$ and $\max\{|x|,|y|\}= \sfrac{\eps}{2C_0}$.  For the former, this is obvious because $\bar u, \bar v = +\infty$, and for the latter, this is clear from the supersolution~\eqref{e.c081003} evaluated at $\max\{|x|,|y|\}= \sfrac{\eps}{2C_0}$.

    Hence, we need only show that $(\bar u, \bar v)$ is a supersolution to~\eqref{e.scaled_eqn}.
    We check only that~\eqref{e.c081004} holds, as the other inequalities follow by similar methods as~\eqref{e.c081004} and along the lines of the work above.  Indeed, for $|x|, |y|$ such that $\max\{|x|,|y|\}\geq \sfrac{\eps}{2C_0},$ we find
    \begin{equation}
        \begin{split}
            \bar v_t &- \eps \Delta \bar v + |\nabla \bar v|^2 + (1 - e^{-\bar v/\eps})
                \\&
                \geq \alpha - \beta \frac{\sqrt{x^2 + (\sfrac{\eps}{2C_0})^2} + y + x^2+y^2}{t^2}
                    - \frac{\beta \eps}{t}
                        \left( 
                        \frac{\sfrac{2C_0}{\eps}}{((\sfrac{2C_0x}{\eps})^2 + 1)^{\sfrac32}}
                        + 2
                        \right)
                    \\&\qquad
                    + \frac{\beta^2}{t^2}\Big( x^2\Big(\frac{1}{\sqrt{x^2 + (\sfrac{\eps}{2C_0})^2}} + 2\Big)^2 + (1+y)^2\Big)
                \\&
                \geq \alpha - \beta \frac{\sqrt{x^2 + (\sfrac{\eps}{2C_0})^2} + y + x^2+y^2}{t^2}
                    - \frac{\beta}{t}
                        \left( 
                        2C_0
                        + 2\eps
                        \right)
                    \\&\qquad
                    + \frac{\beta^2}{t^2}\Big( x^2\Big(\frac{1}{\sqrt{x^2 + (\sfrac{\eps}{2C_0})^2}} + 2\Big)^2 + (1+y)^2\Big).
        \end{split}
    \end{equation}
    It is clear that the third term after the last inequality can be absorbed in the first and last terms (depending on whether $t$ is large or small), up to increasing $\alpha$ and $\beta$ so that $\alpha \geq 2C_0 \beta$ and $\beta \geq 2C_0$.  Additionally, the second term is nonnegative after increasing $\beta$.  Indeed, the $y$-component of the last term clearly dominate the $y$-component of the second term.  When $x$ is at least $\sfrac{\eps}{2C_0}$, it is easy to see that the fourth term is bounded below by
    \be
        \frac{\beta^2}{t^2} ( 1 + 2x)^2
    \ee
    which is clearly larger than the $x$-component of the second term.  When $x < \sfrac{\eps}{2C_0}$, we ``borrow'' from the $y$-component:
    \be
        \frac{\beta^2}{t^2}\Big( x^2\Big(\frac{1}{\sqrt{x^2 + (\sfrac{\eps}{2C_0})^2}} + 2\Big)^2 + (1+y)^2\Big)
            \geq \frac{\beta^2}{t^2}(1 + 2y + y^2)
    \ee
    while the second term satisfies:
    \be
        - \beta \frac{\sqrt{x^2 + (\sfrac{\eps}{2C_0})^2} + y + x^2+y^2}{t^2}
            \geq - \frac{\beta}{t^2} \Big(\frac{\eps}{\sqrt{2}C_0} + y + \frac{\eps^2}{4C_0^2} + y^2\Big).
    \ee
    It is clear that, up to increasing $\beta$, the sum of these two is non-negative.  We deduce that
    \begin{equation}
        \bar v_t - \eps \Delta \bar v + |\nabla \bar v|^2 + (1 - e^{-\bar v/\eps})
            \geq 0.
    \end{equation}
    As noted above, the rest of the argument that $(\bar u, \bar v)$ is a supersolution to~\eqref{e.scaled_eqn} follows similarly.  We deduce that $(u^\eps, v^\eps) \leq (\bar u, \bar v)$, which finishes the proof.
\end{proof}

\subsubsection{The lower bound of $(U,V)$: \Cref{lem:2.4a}}\label{subsec:b.2}

\begin{proof}[Proof of \Cref{lem:2.4a}]
Let $B_R(0)$ be the open ball in $\mathbb{R}^2$ with radius $R$ centered at the origin and let $\lambda_R$, $\phi_R(x,y)$ be, respectively, the principal eigenvalue and positive eigenfunction of 
\begin{equation}
\begin{cases}
    -\Delta \phi_R = \lambda_R \phi_R \quad \text{ in }B_R(0),\quad \phi_R = 0 \quad \text{ on }\partial B_R(0).\\
    \sup \phi_R = 1.
    \end{cases}
\end{equation}
Next, we fix any $R,c>0$ such that $\lambda_1(R) = \frac{1}{3}$,
which is possible since $\lambda_1(R) = R^{-2}\lambda_1(0)>0$. Next, we define, for any $c>0$ and ${\bf e} \in \mathbb{S}^1$
\be
\psi_{c,{\bf e}}(t,x,y)= \begin{cases}
    \phi_R((x,y+R) - tc{\bf e})) &\text{ if }(x,y+R) - tc{\bf e}) \in B_R(0),\\
    0 &\text{ otherwise},
\end{cases} 
\ee
Observe that for any $c,\eta>0$ and ${\bf e}= (e_1,e_2)\in\mathbb{S}^1$ satisfying
\begin{equation}
\frac{c^2}{4}\leq \frac{1}{3},\quad  0 < \eta<\frac{1}{3},\quad e_2\geq 0, 
\end{equation}
it can be directly verified that
\begin{equation}
\begin{split}
    (\psi_{c,{\bf e}})_t - \Delta \psi_{c,{\bf e}}
    	\leq \psi_{c,{\bf e}}(1-\psi_{c,{\bf e}}) &\text{ in }(0,\infty)\times \H,
	\\
    -(\psi_{c,{\bf e}})_y
    	\leq 0 = \psi_{c,{\bf e}} &\text{ on }(0,\infty) \times \partial \H.
\end{split}
\end{equation}
Hence, the pair $(\psi_{c,{\bf e}},0)$ forms a pair of subsolution to \eqref{e.berestycki}. By the notion of generalized subsolution (See \cite[Definition 4.2]{Berestycki2016shape} or \cite[Definition 1.1.1]{Lam2022introduction}), it follows that $(\underline\psi,0)$ is a generalized subsolution, where
\be
\underline\psi(t,x,y) = \eta \sup\{\psi_{c,{\bf e}}(t,x,y):~ 0<c\leq 1,~ {\bf e}\in \mathbb{S}^1,~e_2\geq 0\}.
\ee
By the strong maximum principle, $V(1,x,y)>0$ in 
 the closure of $\H$. Hence, we may 
 choose $0<\eta<\frac{1}{3}$ small enough so that $\psi(1,x,y) \leq V(t,x,y)$ in the (compact) support of $(x,y) \mapsto \psi(1,x,y)$. A comparison principle thus yields that
$V(t,x,y) >\underline\psi(t,x,y)$ for all $t \geq 1$ and all $(x,y) \in \H$, in particular, 
\be
V(t,x,y) \geq \eta_0 \quad \text{ if }t \geq 1,~ y \geq R,~\text{ and }~ |(x,y)|<t/2,
\ee
where $\delta_0$ is independent of the $(t,x,y)$ in the prescribed ranged. 

Next, we claim
\begin{equation}\label{e.bccc}
    V(t,x,y) \geq \eta_0 
        \quad \text{ if }t \geq 1
        ~\text{ and }~ |(x,y)|<t/4.
\end{equation}
If not, then there exists $t_k \to \infty$, $0\leq y_k \leq R$, $|x_k|\leq t/4$, such that $V_k(t,x,y):= V(t_k + t, x_k + x, y)$ satisfies $V_k(0,0,0) = 0$. Letting $k \to \infty$, then (up to a subsequence) $y_k \to y_\infty$ and $V_k \to \tilde V$ in compact subsets of $\mathbb{R} \times \overline{\H}^2$, where $\tilde V$ satisfies
\begin{equation}\label{e.bcc}
\begin{cases}
    \tilde{V}_t = \Delta \tilde{V} + \tilde V(1-\tilde V) &\text{ in }\mathbb{R}\times \H,\\
    -\tilde{V}_y(t,x,0) \geq -\nu \tilde{V}(t,x,0),\quad ~\tilde{V}(t,x,R)>0~ &\text{ for }(t,x) \in \mathbb{R}^2,\\
    \tilde{V}(0,0,y_\infty) = 0. &
\end{cases}
\end{equation}
Now, the strong maximum principle implies $\tilde{V}(t,x,y) >0$ for $y>0$. Hence, 
$y_\infty = 0$ and $\tilde{V}_y(0,0,0)>0$ by the Hopf's lemma. This contradicts the boundary condition in \eqref{e.bcc}. This proves \eqref{e.bccc}. Similarly, we can show that there exists $\eta'_0>0$ such that 
\begin{equation}\label{e.bcccc}
U(t,x) \geq \eta'_0\quad \text{ for }t \geq 1,~\text{ and }~|x|\leq  t/8.\end{equation}
Finally, \eqref{e.bccc} and \eqref{e.bcccc} implies $w^*(t,x,y) = 0$ for $|(x,y)| < t/8$. This proves the lemma. 
\end{proof}

\begin{proof}[Proof of \Cref{lem:2.4}]
Fix $A>0$, and define
\be
\underline{u}(t,x)=A x - Q t- \ep(Q + \log\tfrac{\nu}{\mu}),\quad \text{ and }\quad \underline{v}(t,x,y)=A(x+y) - Q(t+\ep),
\ee
where $Q=Q_A$ is chosen large enough so that 
\be
Q e^{-A x}\geq U(0,x) \quad \text{ and }\quad \frac{\nu Q}{\mu} e^{-A(x+y)} \geq V(0,x,y),\quad Q \geq \max\{2A^2+1,DA^2\}.
\ee
Then we can verify that $(\underline{u}(t,x),\underline{v})$ is a subsolution to \eqref{e.scaled_eqn}, and satisfies
\be
\underline{u}(0,x) \leq u^\ep(0,x)\quad \text{ and }\quad \underline{v}(0,x,y) \leq v^\ep(0,x,y).
\ee
By comparison, we have
\be
u^\ep(t,x) \geq Ax - Q_At - \ep(Q_A + \log \tfrac{\nu}{\mu}),\quad v^\ep(t,x,y) \geq A(x+y) - Q_A(t+\ep).
\ee
Taking the half-relaxed limit (see \eqref{e.half_relaxed}~-~\eqref{e.hr26}) as $\ep \to 0$, we have
\be
w_*(t,x,y)\geq A(x+y) - Q_At.
\ee
The desired conclusion follows by setting $t=0$ and letting $A \to +\infty$.
\end{proof}

\subsubsection{Continuity of $\rsig^*$ and ${\rho^*}$: \Cref{l.Lipschitz}}

\begin{proof}[Proof of \Cref{l.Lipschitz}]
    First we prove that it is enough to establish the following estimate: 
    for every $(x_0,y_0)$ with $y_0>0$, there is a constant $\alpha$ such that, for every $(x,y) \in \overline \H$ satisfying $|x-x_0| + |y-y_0| \leq 1$, we have
    \begin{equation}\label{e.c081101}
        {\rho^*}(x,y)
            \leq {\rho^*}(x_0,y_0) + \alpha \left(1 +y_0^{-\sfrac12}\right) (|x-x_0| + |y-y_0|),
    \end{equation}
    where  $\alpha$ can be chosen uniformly for all $(x_0,y_0)$ in bounded subsets of $\overline\H$.

    By reversing the role of $(x,y),(x_0,y_0)$ in \eqref{e.c081101}, we immediate deduce \eqref{eq:l.Lipschitz.1} and that ${\rho^*}$ is locally Lipshitz continuous on $\H$. This proves assertion (i). To prove assertion (ii), we need to understand the behavior at the boundary.  
    Let us point out that it is enough to establish a $C^{\sfrac12}$ bound on $\H$ and show that ${\rho^*}$ is continuous up to the boundary in $y$.

    The easier piece is continuity up the boundary in $y$, so we begin there.     
    Since $\rsig^*$ is upper semicontinuous, so is ${\rho^*}$.  We deduce that
    \begin{equation}\label{e.c091502}
        \limsup_{(\tilde x, \tilde y)\to(x,0)} {\rho^*}(\tilde x,\tilde y)
            \leq {\rho^*}(x,0).
    \end{equation}
    We seek the reverse inequality.  We obtain this by applying~\eqref{e.c081101} with $y =0 $ and taking a limit as $y_0 \searrow 0$.  Indeed,
    \begin{equation}\label{e.c091501}
        {\rho^*}(x_0,0)
            \leq {\rho^*}(x_0,y_0)
                + \alpha (\sqrt{y_0} + y_0).
    \end{equation}
    We deduce from~\eqref{e.c091502} and~\eqref{e.c091501} that, for any $x_0$,
    \begin{equation}
        {\rho^*}(x_0,0) = \lim_{y\searrow0} \rho^*(x_0,y).
    \end{equation}
    
    Next, we obtain the $C^{\sfrac12}$ bound up to the boundary by applying~\eqref{e.c081101} a large, but finite, number of times.  Fix $(x_0,y_0), (x_1, y_1) \in \overline \H$ such that
    \begin{equation}
		y_0 +|x_0 - x_1|+|y_0 - y_1| \leq 1,
    \end{equation}
and denote
    \begin{equation}
        \delta = \max\{|y_0-y_1|, |x_0-x_1|\}
        \quad\text{ and }
        \quad
        \eps = y_0 + \delta - y_1 \geq 0.
    \end{equation}
    Notice that $\eps \leq 2\delta$ and, hence, it suffices to show that
    \begin{equation}\label{e.c091503}
        {\rho^*}(x_0,y_0) \leq {\rho^*}(x_1,y_1) + C (\sqrt \delta + \sqrt \eps).
    \end{equation}
    We now establish~\eqref{e.c091503}. 
    For any natural number $N$
    we have
    \begin{equation}
        \begin{split}
            {\rho^*}(x_0,y_0)
                &= {\rho^*}(x_0,y_0) - {\rho^*}(x_1,y_0+ \ep)
                    \\&\qquad
                    + \sum_{i=1}^N \left({\rho^*}(x_1, y_1+2^{-i+1}\eps)
                        - {\rho^*}(x_1, y_1+2^{-i} \eps)\right)
                    \\&\qquad + {\rho^*}(x_1, y_1 + 2^{-N} \eps)
                    - {\rho^*}(x_1,y_1)
                    + {\rho^*}(x_1,y_1)
                \\&\leq 
                    2\alpha \frac{ \delta + |x_1-x_0|}{ \sqrt{y_0 + \delta}}
                    + \sum_{i=1}^N 2\frac{\alpha 2^{-i} \eps}{\sqrt{y_1 + 2^{-i}\eps}}
                    + 2\frac{\alpha 2^{-N} \eps}{\sqrt{y_1}}
                    + {\rho^*}(x_1, y_1)
                \\&\leq
                    4\alpha \sqrt \delta 
                    + 2\sum_{i=1}^\infty \alpha 2^{-i/2} \sqrt \eps
                    + \frac{\alpha 2^{-N+1} \eps}{\sqrt{y_1}}
                    + {\rho^*}(x_1, y_1)
                \\&\leq
                    4\alpha \sqrt \delta 
                    + C \alpha \sqrt \eps
                    + \frac{\alpha 2^{-N+1} \eps}{\sqrt{y_1}}
                    + {\rho^*}(x_1, y_1).
        \end{split}
    \end{equation}
    Taking $N\to\infty$, we deduce~\eqref{e.c091503}, which completes the proof of the $C^{\sfrac12}$ bound of ${\rho^*}$. This proves assertion (ii).

    It remains to establish~\eqref{e.c081101}. For this purpose, fix $(x_0,y_0) \in \mathbb{H}^2$ and let $R=|x_0| + y_0$. Let $\alpha$ be a constant to be chosen later, 
    and define, for any $\beta$,
    \begin{equation}
        \phi_\beta(x,y) = \beta + \alpha(1+y_0^{-\sfrac12})\left(|(x-x_0,y-y_0)| + |(x-x_0,y-y_0)|^4 \right).
    \end{equation}
    Define
    \begin{equation}
        \beta_0 := \inf\{\beta : \phi_\beta > {\rho^*} \text{ on } \overline \H\}.
    \end{equation}
    By \Cref{l.upper_bound}, for any $\alpha>0$ and $\beta \in \mathbb{R}$, we have $\phi_\beta > {\rho^*}$ for $|x|+y$ large enough.  Thus, $\beta_0$ is well-defined and there is a touching point $(x_t,y_t)$ such that
    \begin{equation}\label{e.c081102-1}
        \phi_{\beta_0}(x_t,y_t) = {\rho^*}(x_t,y_t),
    \end{equation}
    but $\phi_{\beta_0} \geq {\rho^*}$ elsewhere.  If $(x_t,y_t) = (x_0,y_0)$, we deduce~\eqref{e.c081101} immediately.  Hence, we consider the case where $(x_t,y_t) \neq (x_0,y_0)$.

    First, notice that ${\rho^*}(x_t,y_t)>0$ in this case. Indeed, by the definition of $\beta_0$ and \eqref{e.nonnegative}, 
  \begin{equation}\label{e.c082901-1}
        \beta_0
            = \phi_{\beta_0}(x_0, y_0)
            \geq {\rho^*}(x_0,y_0)
            \geq 0,
    \end{equation}
and hence,     
        \begin{equation}\label{e.c081102}
        \begin{split}
            {\rho^*}(x_t,y_t) 
                &= \phi_{\beta_0}(x_t,y_t)
            >\beta_0\geq 0.
        \end{split}
    \end{equation}
    where the equality follows from \eqref{e.c081102-1}, the strict inequality from $(x_t,y_t)\neq(x_0,y_0)$ and the form of $\phi_{\beta_0}$.

    If $y_t > 0$, then, due to~\eqref{e.sigma*} and~\eqref{e.c081102}, the following holds at $(x_1,y_t)$:
    \begin{equation}\label{e.c081103}
        \begin{split}
            0
                &\geq \phi_{\beta_0} - (x_t,y_t)\cdot \nabla \phi_{\beta_0} + |\nabla \phi_{\beta_0}|^2 + 1
                \\&
                > 0 - \tilde\alpha (x_t,y_t) \cdot \frac{(x_t-x_0,y_t-y_0)}{|(x_t-x_0,y_t-y_0)|}(1 + 4 |(x_t-x_0,y_t-y_0)|^3)
                \\&\qquad
                + \tilde\alpha^2 \bigg|\frac{(x_t-x_0,y_t-y_0)}{|(x_t-x_0,y_t-y_0)|}(1 + 4 |(x_t-x_0,y_t-y_0)|^3)\bigg|^2
                    + 1
                \\&
                \geq - \tilde\alpha R \cdot (1 + 4 |(x_t-x_0,y_t-y_0)|^3)
                \\&\qquad
                - \tilde\alpha |(x_t-x_0,y_t-y_0)|(1 + 4 |(x_t-x_0,y_t-y_0)|^3)
                \\&\qquad
                + \tilde\alpha^2 (1 + 4 |(x_t-x_0,y_t-y_0)|^3)^2
                    + 1.
        \end{split}
    \end{equation}
    where $\tilde\alpha = \alpha(1+y_0^{-\sfrac12})$ and $R=|x_0|+y_0$.
    The last term dominates when $\alpha$, and hence $\tilde\alpha$, is sufficiently large. Indeed, if $\alpha \geq 1+R$, then $\tilde{\alpha} \geq 1 + R$ and the last term containing the highest power of $|(x_t-x_0,y_t-y_0)|$ implies that \eqref{e.c081103} is impossible for $|(x_t-x_0,y_t-y_0)| \geq C_0$ for some constant $C_0$ chosen uniformly in ${\alpha} \geq 1$. By choosing $\alpha$ still larger, we deduce that \eqref{e.c081103} is also impossible for $|(x_t-x_0,y_t-y_0)| \leq C_0$. Hence, the case $y_t>0$ is impossible.

    Let us now consider the case $y_t = 0$.  Notice that~\eqref{e.c081103} does not use the positivity of $y_t$ at all, so we again deduce that
    \begin{equation}
        \phi_{\beta_0} - (x_t,0)\cdot \nabla \phi_{\beta_0} + |\nabla \phi_{\beta_0}|^2 + 1
            >0  \qquad \text{ at }(x_t,y_t),
    \end{equation}
    so that, from the second equation of~\eqref{e.sigma*}, we must have
    \begin{equation}\label{e.c081104}
        0 \geq \phi_{\beta_0} - (x_t,0)\cdot \nabla \phi_{\beta_0} + D|(\phi_{\beta_0})_x|^2 + \Bo((\phi_{\beta_0})_y)  \qquad \text{ at }(x_t,y_t).
    \end{equation}
We claim that, provided $\alpha$ is chosen large, then
    \begin{equation}\label{e.c082901}
        |x_t-x_0| \leq \frac{C(1+|x_0|^2)\sqrt{y_0}}{\alpha},
    \end{equation}
although we postpone the proof momentarily.
Since $y_t = 0$ and $y_0>0$, the definition of $\phi_{\beta_0}$ yields
    \begin{align}
        (\phi_{\beta_0})_y
            &= \alpha(1+{y_0}^{-\sfrac12)} \cdot\frac{ - y_0}{|(x_t-x_0,0-y_0)|}(1 + 4 |(x_t-x_0,0-y_0)|^3) \notag\\
            &< -\frac{\alpha\sqrt{y_0}}{|x_t-x_0|}.\label{e.c082902}
    \end{align}
    Plugging~\eqref{e.c082901} 
    into~\eqref{e.c082902}, we find
    \begin{equation}
        (\phi_{\beta_0})_y(x_t,0)
            < \frac{-\alpha \sqrt{y_0}}{\sqrt{ \sfrac{C^2 y_0}{\alpha^2} }}
            = \frac{-\alpha^2}{ C(1+|x_0|^2)}.
    \end{equation}
    By choosing $\alpha$ sufficiently large so that $(\phi_{\beta_0})_y(x_t,0) \leq -\kappa\nu$, we have $\Bo((\phi_{\beta_0})_y(x_t,0)) = +\infty$ and obtain a contradiction with \eqref{e.c081104}. 
    
We now establish~\eqref{e.c082901}. First, use $\beta_0 \geq 0$ (see \eqref{e.c082901-1}) and  \Cref{l.upper_bound} to find
    \begin{equation}
        \begin{split}
            \alpha(1+y_0^{-\sfrac12})( |x_t - x_0| + |x_t-x_0|^4)
                &\leq \phi_{\beta_0}(x_t,0)
                = {\rho^*}(x_t,0)
                \\&
                \leq A_T( 1 + |x_t|^2)
                \leq 2A_T( 1 + x_0^2 + |x_t-x_0|^2).
        \end{split}
    \end{equation}
    Then~\eqref{e.c082901} follows after an application of Young's inequality and suitably increasing $\alpha$ so that the $x_t-x_0$ term on the right hand side can be absorbed into the left hand side, i.e.
    \be
    \frac{\alpha}{\sqrt{y_0}} |x_t-x_0| \leq 2A_T(1+x_0^2).
    \ee
    By noting that the choice of $\alpha$ depends only on $R=|x_0|+y_0$, the proof is complete.
\end{proof}

\appendix
\section{Comparison principle}\label{sec:A}
In this section, we develop a comparison principle, and, thus, a uniqueness theorem for strong solutions to
\begin{equation}\label{e.HJ'}
		\min\{\rsig, \rsig_t + \Hf (\rsig_x,\rsig_y)\} = 0
			\qquad \text{ in }(0,\infty)\times \H
\ee
and
\be\label{e.HJ_F'}
		\min\{\rsig, w_t + F(\rsig_x, \rsig_y)\} = 0
			\qquad \text{ on } (0,\infty) \times \R \times \{0\},
\end{equation}
Written in this form, it is does not seem that classic results (as in~\cite{UsersGuide}) will apply, and one might think that the ideas of~\cite{Imbert2017flux, Lions2017well, Barles2023illustrated} are needed.  It turns out that our situation is significantly simpler because it is not, in some sense, a true junction problem. Indeed, our problem does not involve two or more hyperplanes glued together -- we have only {\em one} hyperplane.  Thus, the main idea is to use the ideas of~\cite{Lions2017well} to deduce a Neumann type boundary condition on $\partial \H$.  Then, after a suitable localization procedure, we may apply the classic comparison principle for Neumann boundary conditions.

Before we begin, we make a few comments.  First, for the comparison result to hold, we do not need the specific form of $\Hf$ and $\tH$. All we use is that $(q,p) \mapsto \Hf(q,p)$ and $q\mapsto \tH(q)$ are convex and coercive, and that $F(q,p) = \max\{\Hf^-(q,p),\tH(q)\}$.

Our first main result is the following:
\begin{theorem}
	[Comparison principle]
	\label{thm:scp}
Fix any $T>0$.  Let $\underline{w}$ and $\overline w$ be, respectively, strong sub- and supersolutions to~\eqref{e.HJ'}-\eqref{e.HJ_F'} on $(0,T)\times \overline \H$. If $\underline w(0,x,y) \leq \overline{w}(0,x,y)$ for $(x,y) \in \overline \H$, then
	\be
		\underline w \leq \overline w.
			\qquad\text{ on }[0,T)\times\overline \H.
	\ee 
\end{theorem}

Before we prove \Cref{thm:scp}, we show how to deduce uniqueness of (possibly infinite) solutions to~\eqref{e.HJ'} from it.  This is our second main result.
\begin{corollary}
	[Uniqueness]
    \label{cor:scp}
	Any two functions $w: [0,\infty)\times \overline \H \to \mathbb{R}\cup \{+\infty\}$ satisfying
	\begin{enumerate}[(i)]
       \item On $[0,\infty)\times \overline \H$, $w$ is lower semicontinuous and a strong supersolution to~\eqref{e.HJ'},
       
    \item On $(0,\infty)\times \overline \H$, $w$ is finite-valued, continuous, and is a strong subsolution to~\eqref{e.HJ'}, and
    
    \item\label{i.ellt}
    For $t>0$, $w(t,0,0) \leq 0$, while at $t=0$, we have $w(0,0,0) \geq 0$ 
    and $w(0,x,y) = +\infty$ for $(x,y)\neq 0$.
    
    \end{enumerate}
\end{corollary}
\begin{proof} 
We argue by contradiction.  Fix any two functions $w$ and $\tilde{w}$ that satisfy (i), (ii) and (iii).   By the arbitrariness of $w$ and $\tilde w$, we need only show that $w \leq \tilde{w}$.

Any supersolution must necessarily satisfy 
\be
    \tilde w(t,x,y) \geq 0
        \qquad\text{ for all } (t,x,y) \in (0,\infty)\times \bar \H.
\ee
Fix $\tau > 0$.  Using~\ref{i.ellt}, we see that
\be
    w(\tau,0,0)
        \leq 0
        \leq \tilde w(0,0,0).
\ee
Moreover, again using~\ref{i.ellt}, we have, for all $(x,y) \neq (0,0)$,
\be
    w(\tau,x,y)
        < +\infty
        = \tilde w(0,x,y).
\ee
Hence,
\be
    w(\tau, x, y)
        \leq \tilde w(0,x,y)
        \qquad\text{ for all } (x,y) \in \overline \H.
\ee
Finally, we immediately see that
\be
    w(t+\tau, x,y)
\ee
is a subsolution to~\eqref{e.HJ'}.

Applying \Cref{thm:scp}, it follows that
\begin{equation}
    w(t+\tau,x,y)
        \leq \tilde w(t,x,y)
        \quad \text{ for all }(t,x,y) \in (0,\infty)\times \overline \H.
\end{equation}
By the continuity of $w$ on $(0,\infty)\times \overline H$, We can then let $\tau \to 0$ to obtain $w \leq \tilde{w}$ on $(0,\infty)\times \overline \H$, as desired.
\end{proof}

\subsection{Strong solution implies Neumann-type boundary condition}

We follow the idea of~\cite{Lions2017well}, as presented in \cite{Barles2023illustrated}. For this purpose, we associate a Neumann-type boundary condition to strong sub- and supersolutions. 
\begin{definition}\label{d.subdifferential}
    Let $\overline{w}: (0,\infty)\times \overline{\H} \to \mathbb{R}$, and let $(t_0,x_0,y_0) \in (0,\infty)\times \overline{\H}$ be given.   We say that the constant vector $(-\lambda,q,p)$ is an element of the subdifferential at $(t_0,x_0,y_0)$, a set denoted by $D^-\overline{w}(t_0,x_0,y_0)$,
    if there exists $r_0>0$ such that
\be\label{e.c062801}
    \begin{split}
        w(t,x,y)
        &\geq w(t_0,x_0,y_0)
            + (-\lambda,q,p)\cdot(t-t_0,x-x_0,y-y_0)
            + o(|t-t_0| +|x-x_0|+ |y-y_0|)
        \\&
        \text{ for } (t,x,y) \in \{(t',x',y') \in (0,\infty)\times \overline{\H}:~ |(t'-t_0, x'-x_0,y'-y_0)| <r_0\}.
    \end{split}
\ee
The superdifferential at $(t_0,x_0,y_0)$, denoted $D^+w(t_0,x_0,y_0)$ is defined similarly up to reversing the inequality in~\eqref{e.c062801}.

For a given function $\phi(t,x)$ (that is, not depending on $y$), we denote by the sub and superdifferentials, denoted $D^-_{t,x} \phi(t_0,x_0)$ and $D^+_{t,x}\phi(t_0,x_0)$, analogously. 
\end{definition}

We now state the boundary condition: 
\be\label{e.kirchhoff}
    - \rsig_y + \kirch(\rsig_x) = 0.
\ee
We choose the notation $\kirch$ to match that of Lions and Souganidis~\cite{Lions2017well}. 
Then a (weak) solution of~\eqref{e.HJ'}-\eqref{e.kirchhoff} is one such that
\be\label{e.kirchhoff'}
    \begin{cases}
        \min\left\{\rsig, 
            \max\left\{
                - \rsig_y + \kirch(\rsig_x),
                \rsig_t + \Hf(\rsig_x,\rsig_y)
            \right\}
        \right\}
        \geq 0
            \qquad &\text{ on } (0,\infty)\times \R\times\{0\},
        \\
        \min\left\{
            \rsig ,
            - \rsig_y + \kirch(\rsig_x), 
            \rsig_t + \Hf(\rsig_x,\rsig_y)
        \right\}
            \leq 0
                \qquad &\text{ on } (0,\infty)\times \R\times\{0\}.
    \end{cases}
\ee
The first inequality in~\eqref{e.kirchhoff'} corresponds to supersolutions (along with the condition that $\rsig$ is lower semicontinuous), while the second inequality in~\eqref{e.kirchhoff'} corresponds to subsolutions (along with the condition that $\rsig$ is upper semicontinuous).  Let us point out that if $w$ is a supersolution to~\eqref{e.HJ'}-\eqref{e.HJ'}, it must be that
\be\label{e.w_geq_ell_t}
    w \geq 0
        \qquad\text{ on } (0,\infty) \times \bar \H.
\ee

We now show that strong solutions satisfy the Kirchhoff condition.  This was originally observed by Lions and Souganidis in~\cite{Lions2017well} in a slightly different context, and we follow their proof.
\begin{lemma}\label{lem.subk}
    Let $\rsig$ be a strong subsolution (resp. strong supersolution) of \eqref{e.HJ'}-\eqref{e.HJ_F'}, then it is a weak solution of the Neumann-type condition \eqref{e.HJ'}-\eqref{e.kirchhoff} with coefficient
    \be 
        \kirch(q_0) = p_{q_0},
    \ee
    where we recall that $p_{q_0} \in [0,\infty)$ is given by \eqref{e.pq}.
\end{lemma}
\begin{proof}
First, we assume $\rsig$ is a strong subsolution. 
It suffices to check the condition on the boundary $\{y=0\}$. Fix $(t_0,x_0,0)$ and $(-\lambda, p_0,q_0) \in D^+ \rsig(t_0,x_0,0).$  Then
\be
    \max\{\Hf^-(q_0,p_0), \tH(q_0)\}
        = F(q_0,p_0)
        \leq \lambda.
\ee
If $\Hf(q_0,p_0) \leq \lambda$ or $w(t_0,x_0,0) \leq 0$, then we are finished by~\eqref{e.kirchhoff'}

It remains only to argue in the case where $\Hf(q_0,p_0) > \lambda$ and $w(t_0,x_0,0) > 0$.  In this case, we must show that
\be\label{e.c062802}
    p_0 \geq p_{q_0}.
\ee
Since
\be\label{e.c062803}
    \Hf^-(q_0,p_0)
        \leq \lambda
        < \Hf(q_0,p_0),
\ee
it follows that $p_0>0$ (recall~\eqref{e.Hf-}).

By construction, we have that
\be\label{e.c062804}
    \Hf(q_0, p_{q_0})
        \leq \tH(q_0)
        \leq \lambda.
\ee
Since $p_{q_0} \geq 0$ and $\Hf(q_0,\cdot)$ is increasing on $[0,\infty)$, we deduce from~\eqref{e.c062803}-\eqref{e.c062804} that~\eqref{e.c062802} holds, as desired.

Next, assume that $\rsig$ is a strong supersolution and $(-\lambda, p_0, q_0) \in D^-w(t_0,x_0,0)$.  Then
\begin{equation}\label{e.subk.2}
    \max\{\Hf^-(q_0,p_0), \tH(q_0)\}
        \geq \lambda.
\end{equation}
We are finished if
\be
    - \lambda + \Hf(q_0,p_0)
        \geq 0.
\ee
Hence, we consider when
\be
    \Hf(q_0,p_0) < \lambda,
\ee
in which case we need to show that $p_0 \leq p_{q_0}$. By \eqref{e.subk.2}, we divide into two cases: (i) $\Hf^-(q_0,p_0) \geq \lambda$; (ii) $\tH(q_0) \geq \lambda$.

In case (i), $\Hf^-(q_0,p_0) > \Hf(q_0,p_0)$, which implies that $p_0 < 0$.  By definition, $0\leq  p_{q_0}$.   Thus, the proof is complete in this case.

In case (ii), $\Hf(q_0,p_0) < \lambda \leq \tH(q_0)$.  We are, thus, in the setting of \Cref{lem.pq}.(i), whence we conclude that $\tH(q_0) = \Hf(q_0,p_{q_0})$. It follows that
\be
    \Hf(q_0, |p_0|)
        = \Hf(q_0, p_0)
        < \Hf(q_0, p_{q_0}).
\ee
Since $\Hf$ is increasing on $[0,\infty)$ and $p_{q_0} \geq 0$, we deduce that
\be
    p_0
        \leq |p_0|
        < p_{q_0}.
\ee
This completes the proof.
\end{proof}

\subsection{Proof of the comparison principle}\label{ss.comparison}

\begin{proof}[Proof of \Cref{thm:scp}]
In view of \Cref{lem.subk}, we are nearly in the classical setting of, e.g., \cite[Theorem~7.12]{UsersGuide}.  However, such arguments rely on the boundedness of the sub- and supersolutions as well as the domain.  In the next three steps, we perform a reduction to this setting.

\medskip
\noindent
{\bf \# Step one: Without loss of generality, we may assume that $\underline w$ is bounded from above.} 
    We claim that $\underline{w}_K = \min\{\underline w - t,K\} + t$ is a strong subsolution to \eqref{e.HJ'}-\eqref{e.HJ_F'} for each $K>0$. Indeed, take a sequence $\{g_j\}$ of smooth functions satisfying 
    \be\label{e.c071807}
        0 \leq g'_j(r)\leq 1
        \qquad\text{ and }\qquad
        g_j(r)\nearrow \min\{r,K\}
            \quad \text{ for }r\in\mathbb{R}.
    \ee
    Notice that
    \be\label{e.c070302}
        |\nabla g_j(\underline w)|
            \leq |\nabla w|.
    \ee
    We claim that
    \be 
        \hat w = g_j(\underline w + t) - t
    \ee
    is a viscosity subsolution to \eqref{e.HJ'}-\eqref{e.HJ_F'}. 
    To see this, first note that we need only check the set $\{(t,x,y):~ \hat{w} > 0\}$.  We check this case formally assuming that $\underline w$ is $C^1$, although it is easy to see that these computations can easily be made rigorous. For $G = \Hf, \Hf^-$ or $\Hr$,
    \be
        \partial_t \hat w + G(\nabla \hat w)
            = g_j' \partial_t \hat w
                + g_j'
                - 1
                + G(g_j' \nabla \underline w).
    \ee
    Next, since $G(0) = 1$, and $G$ is convex, it follows that, for any $\lambda \in (0,1)$ and for any $p$,
    \be
        G(\lambda p) \leq \lambda G(p) + (1-\lambda) G(0)
            = \lambda G(p) + 1-\lambda.
    \ee
    Here $p\in \R$ or $\R^2$, depending on the choice of $G$.  Hence,
    \be
        \begin{split}
            \partial_t \hat w + G(\nabla \hat w)
                &\leq g_j' \partial_t \hat w
                    + g_j'
                    - 1
                    + g_j'G(\nabla \underline w)
                    + 1 - g_j'
                \\&
                = g_j' \left(\partial_t \underline w + G(\nabla \underline w)\right)
                \leq 0.
        \end{split}
    \ee
    Using the stability of strong subsolutions (see, e.g., \cite[Theorem 14.2.1]{Barles2023illustrated}), we take $j \to \infty$ and deduce that $\underline{w}_K = \min\{\underline w,K\}$ is a strong subsolution to \eqref{e.HJ'}-\eqref{e.HJ_F'}.
    
    Notice that, if we prove that $\underline{w}_K \leq \overline w$ for all $K$, then we deduce that $\underline W \leq \overline w$ in the limit $K\to\infty$.  We may, thus, assume that $\underline w$ is bounded from above.

\medskip
\noindent
{\bf \# Step two: reduction to a strict subsolution.} 
Without loss of generality, we may assume that there is $\eta>0$ such that
\be\label{e.c070303}
    \limsup_{t\to T-} \underline w
        = \limsup_{|x|+|y|\to\infty} \underline w
        = -\infty
    \quad\text{ and }\quad
    \min_{\{0\}\times\overline \H} (\overline w - \underline w) > 0,
\ee
while
\be\label{e.c070304}
\begin{cases}
    \min\{\underline w ,\partial_t \underline w + \Hf(\nabla \underline w) + 2\eta\} \leq 0
        &\qquad\text{ on }(0,T)\times \H,\\
    \min\{\underline w,\partial_t \underline w + F(\nabla \underline w) + 2\eta\} \leq 0
        &\qquad \text{ on }(0,T)\times \partial \H.
\end{cases}
\ee
It is easy to see that
\be
    \tilde{w}_1(t,x,y)=-\frac{K}{T-t} -\log(1+|x|^2 + |y|^2)
\ee
is a strong subsolution to~\eqref{e.HJ'}-\eqref{e.HJ_F'} for $K$ sufficiently large. Thanks to the convexity of $H$, $\Hf^-$ and $\tH$, the function
\be
    \underline{w}_\mu = (1-\mu)\underline{w} + \mu \tilde{w}_1
\ee
satisfies~\eqref{e.c070303}-\eqref{e.c070304} 
for any $0<\mu<1$ (recall that $\underline w$ is bounded from above by the previous step). Again, it suffices to show that $\underline{w}_\mu \leq \overline w$ for all sufficiently small $\mu >0$.

\medskip
\noindent
{\bf \# Step three: reduction to a compact portion of the boundary and the conclusion.}
Let us note that, due to~\eqref{e.w_geq_ell_t} and the work in Step two, there is $R>0$ such that the
\be\label{e.c070305}
    \inf_{[0,T]\times\overline \H} (\overline w - \underline w)
        = \min_{Q_R} (\overline w - \underline w)
        <0,
\ee
where
\be
    Q_R = \{(t,x,y) \in [0,T]\times\overline \H : \sfrac1R < t < T-\sfrac1R, |x| + |y| \leq R\}.
\ee
At this point, we are essentially in the classical setting where the standard technique of double variables can be applied.  This can be done in the same vein as the time independent result \cite[Theorem~7.12]{UsersGuide}.  Indeed, on can check that our Neumann-type boundary condition~\eqref{e.kirchhoff} satisfies the conditions stated there.  
This concludes the proof.
\end{proof}

\bibliographystyle{plain}
\bibliography{refs}

\end{document}